\documentclass{amsart}
\usepackage{amssymb,amsmath, amsthm,latexsym}
\usepackage{graphics}
\usepackage{psfrag}
\usepackage{amscd}
\usepackage{graphicx}
\newcommand{\cal}[1]{\mathcal{#1}}
\theoremstyle{plain}

\newtheorem{theo}{Theorem}

\newtheorem{lemma}{Lemma}[section]
\newtheorem{theorem}[lemma]{Theorem}
\newtheorem{proposition}[lemma]{Proposition}
\newtheorem{corollary}[lemma]{Corollary}
\theoremstyle{definition}
\newtheorem{definition}[lemma]{Definition}

\newtheorem{remark}[lemma]{Remark}

\parskip=\bigskipamount

\let\egthree=\phi
\let\phi=\varphi
\let\varphi=\egthree

\begin{document}
\title{Periodic orbits
 in the thin part of strata}
\author{Ursula Hamenst\"adt}
\date{November 16, 2023}

\begin{abstract}
Let $S$ be a closed oriented surface of 
genus $g\geq 0$ with $n\geq 0$ punctures and
$3g-3+n\geq 5$.  
Let ${\cal Q}$ be a connected component 
of a stratum in the moduli space ${\cal Q}(S)$ 
of area one
meromorphic quadratic differentials on $S$ with $n$
simple poles at the punctures 
or in the moduli space ${\cal H}(S)$ 
of abelian differentials on $S$ if $n=0$. 
For a compact subset $K$ of ${\cal Q}(S)$ or of ${\cal H}(S)$, 
we show that 
the asymptotic growth rate of the number of periodic orbits for the
Teichm\"uller flow $\Phi^t$ on ${\cal Q}$ which are entirely contained in 
${\cal Q}-K$ is at least  $h({\cal Q})-1$
where $h({\cal Q})>0$ is the complex dimension of 
$\mathbb{R}^+{\cal Q}$. 

\end{abstract}

\maketitle

\section{Introduction}

For a closed oriented surface $S$ of genus
$g\geq 0$, the moduli space ${\cal Q}(S)$ 
of area one meromorphic 
\emph{quadratic differentials} with at most simple poles
which are not squares of holomorphic one-forms
decomposes into \emph{strata}. Such a stratum
is the subset of ${\cal Q}(S)$
of all quadratic 
differentials with the same number $n\geq 0$ of 
simple poles and the same number $\ell\geq 0$ of 
zeros of the same order $m_i$ $(1\leq i\leq \ell)$.
Strata need not be connected, but they
have only finitely many connected components \cite{L08}. 
A component ${\cal Q}$ of a stratum 
is a real hypersurface in a complex algebraic
orbifold of complex dimension
\[h({\cal Q})=2g-2+\ell +n.\]

Similarly, for $g\geq 2$ 
 the moduli space ${\cal H}(S)$ of 
area one \emph{abelian
differentials} on $S$ decomposes
into strata. A stratum  
is the subset of ${\cal H}(S)$ of
holomorphic one-forms with the same number
$s\geq 0$ of zeros of the same order $k_i$
$(1\leq i\leq s)$. Again, strata need not be connected,
but they have at most 3 connected components \cite{KZ03}. 
A component ${\cal Q}$ of a stratum 
is a real hypersurface in a 
complex algebraic orbifold of complex dimension
\[h({\cal Q})=2g-1+s.\] 

The \emph{Teichm\"uller flow} $\Phi^t$ acts on 
${\cal Q}(S)$ and ${\cal H}(S)$, and this action
preserves the strata. Each component of a stratum contains
periodic orbits, and these orbits can be counted: 
Namely, for a subset $A$ of ${\cal Q}(S)$ or of 
${\cal H}(S)$ and a number $R>0$, denote by 
$n_{A}(R)$ the number of period orbits in 
$A$ of length at most $R$. Then for any 
component ${\cal Q}$ of a stratum, 
we have \cite{H13}
\[n_{\cal Q}(R)\sim \frac{1}{h({\cal Q})R}e^{h({\cal Q})R}\quad
(R\to \infty)\]
which means that the ratio of the numbers on both sides of $\sim$
tend to $1$ as $R\to \infty$.

The mechanism behind this result is that the 
Teichm\"uller flow on components of strata is non-uniformly
hyperbolic in a precise sense \cite{H22}. However, 
components of strata are 
non-compact, and a major difficulty is the possibility that 
as $R\to \infty$, the number of periodic orbits of length at most $R$ 
which do not intersect some
fixed compact set $K$ grows faster than the number of  
periodic orbits of length at 
most $R$ which intersect $K$. 
That this is not the case was established 
by Eskin and Mirzakhani
\cite{EM11} and Eskin, Mirzakhani 
and Rafi \cite{EMR13} who showed that
for every $\epsilon >0$, there is a compact subset 
$K$ of ${\cal Q}$ with the property that the  growth rate
of the number of periodic orbits in ${\cal Q}$ which 
are entirely contained in ${\cal Q}-K$ is at most 
$h({\cal Q})-1+\epsilon$.

The main goal of 
this article is to establish a converse of this result. 
We are interested in periodic orbits
in components of strata which project into the thin part of 
moduli space, which ignores the possibility of periodic orbits
in components of strata which are arbitrarily close to a component 
in the boundary of the stratum, obtained by colliding zeros or
merging zeros and poles of the differentials.

\begin{theo}\label{main} Let ${\cal Q}$ be a component
of a stratum  
of area one meromorphic quadratic differentials
with $n$ poles 
on a closed surface of genus $g\geq 0$
where $3g-3+n\geq 5$ or of a stratum of
area one abelian differentials on a surface of genus $g\geq 3$. 
Then for every compact set $K\subset {\cal Q}(S)$ 
we have 
\[\lim\inf_{R\to \infty}Re^{-(h({\cal Q})-1)R}n_{\cal Q-K}(R)
>0.\]
\end{theo}

Note that the Teichm\"uller flow $\Phi^t$  on the space
${\cal Q}(1;-1)$ of area one meromorphic quadratic differentials
with a single simple pole  on a torus $T^2$
can be identified with the geodesic flow on the
unit tangent bundle of the 
modular surface $PSL(2,\mathbb{Z})\backslash {\mathbb{H}}^2$.
Thus in this case, there is a 
compact set $K$ which is intersected by every periodic orbit
for $\Phi^t$. This shows that 
a constraint on the complexity of the stratum
is necessary.

In analogy to finite volume locally symmetric manfolds
of $\mathbb{Q}$-rank at least $2$,
Theorem \ref{main} can be viewed as a witness of 
higher rank for components of strata, with a small number of
exceptions in low dimension.
Indirectly, it draws on the fact that as is the case
for locally symmetric manifolds of $\mathbb{Q}$-rank at least $2$, 
if $3g-3+n\geq 5$ then the component 
${\cal Q}$ (or its projectivization) 
admits a (partial) compactification which is
built from components of strata of smaller complexity 
\cite{MW15,BCGGM19}.

Theorem \ref{main} is not optimal. In
forthcoming work which builds on
the results in this article,
we use a flexible symbolic coding of the Teichm\"uller flow to show
that for every component ${\cal Q}$ of a stratum 
as in Theorem \ref{main} and every 
compact set $K\subset {\cal Q}(S)$,
the asymptotic growth rate of periodic orbits
for $\Phi^t$ which are contained in ${\cal Q}-K$ 
is strictly larger than 
$h({\cal Q})-1$. Periodic orbits entirely contained in 
a fixed compact set $K\subset {\cal Q}(S)$ are
counted in \cite{H10}.

The main technical 
tool for the proof of Theorem \ref{main} 
is the construction of combinatorial models for components
${\cal Q}$ of strata of area one 
abelian or quadratic differentials. These models are adapted to the 
study of the dynamics of the Teichm\"uller flow and do not use 
cylinder decompositions. They are used to investigate 
pseudo-Anosov mapping classes 
which stabilize a component $\tilde {\cal Q}$ of the
preimage of ${\cal Q}$ in the Teichm\"uller space of area one marked
abelian or quadratic differentials.  
Unit cotangent lines of axes of these mapping classes acting 
on Teichm\"uller space project 
to periodic orbits in ${\cal Q}$, and each periodic orbit can be obtained in this way.

We use this construction 
to lift periodic orbits of the Teichm\"uller flow 
from the principal boundary of ${\cal Q}$ in the sense of \cite{EMZ03}
into ${\cal Q}-K$. 
The resulting periodic orbits fellow travel the orbits 
in the principal boundary used for their
construction in a controlled way except for a subsegment 
whose length only depends on the compact set $K$.

The tools developed in this article can also be applied to 
construct orbits in a given stratum with an 
arbitrarily prescribed recursion behavior to 
compact subsets of moduli space. An example for this is
given in the following statement. For its formulation, 
for a point $X$ in the moduli space ${\cal M}(S)$ of 
hyperbolic metrics  on the surface $S$  denote by ${\rm syst}(X)$ 
the \emph{systole} of $X$, that is, the minimal length of 
a closed geodesic for the hyperbolic metric $X$.

\begin{theo}\label{thm6}
Let ${\cal Q}$ be a component of a stratum of 
area one meromorphic
(or abelian) differentials on a surface of genus $g\geq 0$
with $n\geq 0$ simple poles. If $3g-3+n\geq 5$ then 
there is a Teichm\"uller geodesic ray $\gamma:[0,\infty)\to {\cal M}(S)$
defined by a differential 
with uniquely ergodic vertical measured geodesic lamination and
such that 
\[\lim\sup_{t\to \infty}
\frac{1}{t}\log {\rm syst}(\gamma(t))<0.\] 
\end{theo}

The organization of the article is as follows. 
In Section 2 we collect some results from the literature
in the form needed later on.
In Section 3, we construct
combinatorial models for components of strata. 
In Section 4 we use the classification of components of 
strata by Kontsevich and Zorich \cite{KZ03} (for strata of abelian
differentials) and 
Lanneau \cite{L08} (for strata of quadratic differentials which are
not squares of holomophic one-forms)
to find for each component of a stratum with $3g-3+n\geq 5$ 
such combinatorial models which encode the
degeneration of differentials into 
suitably chosen components of the principal boundary.
Section 5 translates information
on dynamical properties of the Teichm\"uller flow on
strata into the combinatorial setup. 
This is then used in Section 6 to prove Theorem \ref{main}.

\bigskip\noindent
{\bf Acknowledgement:} A major part of this work was carried out
in spring 2010 during a special semester 
at the Hausdorff Institute for Mathematics in 
Bonn and in spring 2011 
during a visit of the MSRI in Berkeley.
I thank both institutes for their hospitality and for the 
excellent working conditions.

\section{Train tracks and geodesic laminations}\label{traintracks}

In this section we introduce 
some technical tools needed in the sequel.
We begin with summarizing some constructions from
\cite{PH92,H09}
which will be used throughout the paper.
We then 
introduce a class of train tracks 
which will serve as combinatorial models for components of strata in the later
sections, and we discuss some
of their properties. 

\subsection{Geodesic laminations}
Let $S$ be an
oriented surface of
genus $g\geq 0$ with $n\geq 0$ marked points (punctures) 
and where $3g-3+n\geq 2$.
A \emph{geodesic lamination} for a complete
hyperbolic structure on $S$ of finite volume is
a \emph{compact} subset of $S$ which is foliated into simple
geodesics.
A geodesic lamination $\lambda$ is called \emph{minimal}
if each of its half-leaves is dense in $\lambda$. Thus a simple
closed geodesic is a minimal geodesic lamination. A minimal
geodesic lamination with more than one leaf has uncountably
many leaves and is called \emph{minimal arational}.
Every geodesic lamination $\lambda$ consists of a disjoint union of
finitely many minimal components and a finite number of isolated
leaves. Each of the isolated leaves of $\lambda$ either is an
isolated closed geodesic and hence a minimal component, or it
\emph{spirals} about one or two minimal components
\cite{CEG87}.

A geodesic lamination $\lambda$ on $S$ is said to
\emph{fill up $S$} if its complementary regions 
are all topological disks or once
punctured  monogons or once punctured bigons. 
Here a once puncture monogon is a once punctured disk
with a single cusp at the boundary.
A \emph{maximal} geodesic
lamination is a geodesic lamination whose 
complementary regions are all ideal triangles
or once punctured monogons. 

\begin{definition}\label{large}
A geodesic lamination $\lambda$
is called \emph{large} if $\lambda$ fills up
$S$ and if
moreover $\lambda$ can be 
approximated in the \emph{Hausdorff topology}
by simple closed geodesics.
\end{definition}

Since every minimal geodesic lamination can be 
approximated in the Hausdorff topology by simple 
closed geodesics \cite{CEG87}, a minimal geodesic
lamination which fills up $S$ is large. However, there are large
geodesic laminations with finitely many leaves. We refer
to \cite{H09} for more detailed information.

The \emph{topological type} of a large geodesic
lamination $\nu$ is a tuple 
\[(m_1,\dots,m_\ell;-p_1,p_2)\text{ where }1\leq m_1\leq \dots \leq m_\ell,\,
\sum_{i}m_i=4g-4+p_1,\,p_1+p_2=n.\]
Here $\ell\geq 1$ is the number of complementary regions which 
are topological disks, and these disks 
are $m_i+2$-gons $(i\leq \ell)$.  
There are $p_1$ once punctured monogons and 
$p_2$ once punctured bigons. Let 
\[{\cal L\cal L}(m_1,\dots,m_\ell;-p_1,p_2)\]
be the space of all large geodesic laminations of type 
$(m_1,\dots,m_\ell;-p_1,p_2)$ equipped with the 
restriction of the Hausdorff
topology for compact subsets of $S$.

A \emph{measured geodesic lamination} is a geodesic lamination
$\lambda$ together with a translation invariant transverse
measure. Such a measure assigns a positive weight to each compact
arc in $S$ with endpoints in the complementary regions of
$\lambda$ which intersects $\lambda$ nontrivially and
transversely. The geodesic lamination $\lambda$ is called the
\emph{support} of the measured geodesic lamination; it consists of
a disjoint union of minimal components. The space ${\cal M\cal L}$
of all measured geodesic laminations on $S$ equipped with the
weak$^*$-topology is homeomorphic to $S^{6g-7+2n}\times
(0,\infty)$. Its projectivization is the space ${\cal P\cal M\cal
L}$ of all \emph{projective measured geodesic laminations}. 

The measured geodesic lamination $\mu\in {\cal
M\cal L}$ \emph{fills up $S$} if its support fills up $S$.
This support is then necessarily connected and hence minimal. Since
a minimal geodesic lamination can be approximated in the Hausdorff
topology by simple closed curves \cite{CEG87}, there exists a
tuple $(m_1,\dots,m_\ell;-p_1,p_2)$ such that the support of $\mu$ 
defines a point in the
set ${\cal L\cal L}(m_1,\dots,m_\ell;-p_1,p_2)$. 
The projectivization of a measured geodesic lamination
which fills up $S$ is also said to fill up $S$.

There
is a continuous symmetric pairing $\iota:{\cal M\cal L}\times {\cal
M\cal L}\to [0,\infty)$, the so-called \emph{intersection form},
which extends the geometric intersection number between simple
closed curves.

\subsection{Train tracks}
A \emph{train track} on $S$ is an embedded
1-complex $\tau\subset S$ whose edges
(called \emph{branches}) are smooth arcs with
well-defined tangent vectors at the endpoints. At any vertex
(called a \emph{switch}) the incident edges are mutually tangent.
Through each switch there is a path of class $C^1$
which is embedded
in $\tau$ and contains the switch in its interior. 
A simple closed curve component of $\tau$ contains
a unique bivalent switch, and all other switches are at least
trivalent.
The complementary regions of the
train track have negative Euler characteristic, which means
that they are different from disks with $0,1$ or
$2$ cusps at the boundary and different from
annuli and once-punctured disks
with no cusps at the boundary.
We always identify train
tracks which are isotopic.
Throughout we use the book \cite{PH92} as the main reference for 
train tracks. All train tracks will be  \emph{marked}, that is, we think
of a train track $\tau$ as a (coarsely well defined) 
point in the \emph{marking graph} of the subsurface of $S$ filled by $\tau$.
This subsurface is a small neighborhood of the union of 
$\tau$ with all complementary components of $\tau$ which are
topological disks or once punctured topological disks. 

A train track is called \emph{generic} if all switches are
at most trivalent. For each switch $v$ of 
a generic train track $\tau$ which is not contained in 
a simple closed curve component, there is a unique
half-branch $b$ of $\tau$ which is incident on $v$ and which is
\emph{large} at $v$. This means that every germ of an
arc of class $C^1$ on $\tau$ which passes through $v$ also
passes through the interior of $b$. 
A half-branch which is not large is called \emph{small}.
A branch $b$ of $\tau$ is
called \emph{large} (or \emph{small}) if each of its
two half-branches is large (or small). A branch which 
is neither large nor small is called \emph{mixed}.

\begin{remark} As in \cite{H09}, all train tracks
are assumed to be generic. Unfortunately this leads to 
a small inconsistency of our terminology with the
terminology found in the literature.
\end{remark}

A \emph{trainpath} on a train track $\tau$ is a
$C^1$-immersion $\rho:[k,\ell]\to \tau$ such that
for every $i< \ell-k$ the restriction of $\rho$ to 
$[k+i,k+i+1]$ is a homeomorphism onto a branch of $\tau$.
More generally, we call a $C^1$-immersion $\rho:[a,b]\to \tau$
a \emph{generalized trainpath}. 
A trainpath $\rho:[k,\ell]\to \tau$ is \emph{closed}
if $\rho(k)=\rho(\ell)$ and if either the image of $\rho$ is
a closed curve component of $\tau$ or if 
precisely one of the half-branches
$\rho[k,k+1/2],\rho[\ell-1/2,\ell]$ is large.

A generic 
train track $\tau$ is \emph{orientable} 
if there is a consistent orientation of the 
branches of $\tau$ such that 
at any switch $s$ of $\tau$, the orientation of the large
half-branch incident on $s$ extends to the orientation
of the two small half-branches incident on $s$.
If $C$ is a complementary polygon of an oriented
train track then the number of sides of $C$ is even.
In particular, a train track which contains a once
punctured monogon component
is not orientable
(see p.31 of \cite{PH92} for 
a more detailed discussion).

A train track or a geodesic lamination $\eta$ is
\emph{carried} by a train track $\tau$ if
there is a map $F:S\to S$ of class $C^1$ which is homotopic to the
identity and maps $\eta$ into $\tau$ in such a way 
that the restriction of the differential of $F$
to the tangent space of $\eta$ vanishes nowhere;
note that this makes sense since a train track has a tangent
line everywhere. We call the restriction of $F$ to
$\eta$ a \emph{carrying map} for $\eta$.
Write $\eta\prec
\tau$ if the train track $\eta$ is carried by the train track
$\tau$. Then every geodesic lamination $\nu$ which is carried
by $\eta$ is also carried by $\tau$.

A train track \emph{fills up} $S$ if its complementary
components are topological disks or once punctured 
monogons or once punctured bigons.  Note that such a train track
$\tau$ is connected.
Let $\ell\geq 1$ be the number of those complementary 
components of $\tau$ which are topological disks.
Each of these disks is an $m_i+2$-gon for some $m_i\geq 1$
$(i=1,\dots,\ell)$. The
\emph{topological type} of $\tau$ is defined to be
the ordered tuple $(m_1,\dots,m_\ell;-p_1,p_2)$ where
$1\leq m_1\leq \dots \leq m_\ell$
and $p_1$ (or $p_2$) is the number of once punctured monogons
(or once punctured bigons);
then $\sum_im_i=4g-4+p_1$ and $p_1+p_2=n$.
If $\tau$ is orientable then $p_1=0$ and $m_i$ is even 
for all $i$. 
A train track of topological type $(1,\dots,1;-p_1,0)$ is called 
\emph{maximal}. The complementary components
of a maximal train track are all trigons,
that is, topological disks with three cusps at the boundary,
or once punctured monogons.

A \emph{transverse measure} on a generic train track $\tau$ is a
nonnegative weight function $\mu$ on the branches of $\tau$
satisfying the \emph{switch condition}:
for every trivalent switch $s$ of $\tau$,  the sum of the weights
of the two small half-branches incident on $s$ 
equals the weight of the large half-branch.
Particular such transverse measures are the counting measures
of simple multicurves $c$ carried by $\tau$. Such a measure associates
to a branch $b$ the number of the preimages of an interior point of $b$
under the carrying map. The weight of every branch
with respect to this measure is integral.
In particular, the ratio of weights of 
any two branches is rational, and we call a transverse measure with this
property \emph{rational}. 
The set of rational measures is invariant under scaling, and it is dense
is the cone of all transverse measures on $\tau$.

 A \emph{subtrack} $\sigma$ of a train track $\tau$
is a subset of $\tau$ which is itself a train track.
Then $\sigma$ is obtained from $\tau$ by removing some
of the branches, and 
we write $\sigma <\tau$.
A \emph{vertex cycle} for $\tau$ is defined
to be an embedded subtrack of $\tau$ which 
either is a simple closed curve or a \emph{dumbbell},
that is, it consists of two loops
with one cusp which are connected by an embedded segment joining 
the cusps (that this definition is equivalent to the
definition defined in other works can for example be found in
\cite{Mo03}, see also \cite{H06}).
An orientable train track does not contain dumbbells.
Each vertex cycle supports a single transverse measure up to scale.

The following is well known and will be used several times in the sequel. 
We refer to \cite{Mo03} for a comprehensive discussion.

\begin{lemma}\label{polyhedron}
Let ${\cal V}(\tau)$ be the space of all transverse measures
on $\tau$.
\begin{enumerate}
\item ${\cal V}(\tau)$ 
 has the structure of a cone over a compact 
convex polyhedron
in a finite dimensional vector space. 
\item 
The vertices 
of the polyhedron are up to scaling the measures supported on
the vertex cycles. 
\item There exists a natural homeomorphism of ${\cal V}(\tau)$, 
equipped with the euclidean topology, onto the closed subspace of 
${\cal M\cal L}$ of all measured geodesic laminations carried by $\tau$.
\end{enumerate}
\end{lemma}

The train track is called
\emph{recurrent} if it admits a transverse measure which is
positive on every branch. We call such a transverse measure $\mu$
\emph{positive}, and we write $\mu>0$ (see \cite{PH92} for 
more details).

If $b$ is a small branch of $\tau$ which is incident on two
distinct switches of $\tau$ then
the graph $\sigma$
obtained from $\tau$ by removing $b$ is a subtrack of $\tau$.
We then call $\tau$ a 
\emph{simple extension} of $\sigma$.
Note that formally to obtain the subtrack $\sigma $ from $\tau-b$
we may have to delete the switches on which the branch $b$ is incident. 

\begin{lemma} \label{simpleex}
\begin{enumerate}
\item
A simple extension $\tau$ of 
a recurrent non-orientable connected train track $\sigma$
is recurrent. Moreover,
\[{\rm dim}{\cal V}(\tau)={\rm dim}{\cal V}(\sigma)+1.\]
\item An orientable simple extension $\tau$ of a recurrent orientable 
connected train track $\sigma$
is recurrent. Moreover,
\[{\rm dim}{\cal V}(\tau)={\rm dim}{\cal V}(\sigma)+1.\]
\end{enumerate}
\end{lemma}
\begin{proof}
If $\tau$ is a simple extension of a connected
train track $\sigma$ then $\sigma$ can be
obtained from $\tau$ by the removal of a small branch $b$ which
is incident on two distinct switches
$s_1,s_2$. Then $s_i$ is an interior point of a branch
$b_i$ of $\sigma$ $(i=1,2)$. 

If $\sigma$ is moreover non-orientable and recurrent then there is 
a trainpath $\rho_0:[0,t]\to 
\tau-b$ which begins at $s_1$, ends at $s_2$ and such that
the half-branch $\rho_0[0,1/2]$ is small at $s_1=\rho_0(0)$ and  
that the half-branch $\rho_0[t-1/2,t]$ is small at $s_2=\rho_0(t)$. 
Extend $\rho_0$ to a closed trainpath $\rho$ on $\tau -b$ 
which begins and ends at $s_1$. This is possible since
$\sigma$ is non-orientable, connected and recurrent.
There is a closed trainpath $\rho^\prime:[0,u]\to \tau$ 
which can be obtained from $\rho$ by replacing the trainpath $\rho_0$ by
the branch $b$ traveled through
from $s_1$ to $s_2$. The counting measure of $\rho^\prime$ on $\tau$ 
satisfies the switch condition and hence it
defines a transverse measure on $\tau$ which is positive on $b$.
On the other hand, every transverse measure on $\sigma$
defines a transverse measure on $\tau$. Thus 
since $\sigma$ is recurrent and since the 
sum of two transverse measures on $\tau$ is 
again a transverse
measure,  the train track $\tau$ is recurrent as well.
Moreover, we have 
${\rm dim}{\cal V}(\tau)\geq {\rm dim}{\cal V}(\sigma)+1$.

Let
$k$ be the number of branches of $\tau$.
Label the branches of $\tau$ 
with the numbers $\{1,\dots,k\}$ so that the number $k$ is 
assigned to $b$. 
Let $e_1,\dots,e_k$ be the standard basis of $\mathbb{R}^k$ and
define a linear
map $A:\mathbb{R}^k\to \mathbb{R}^k$ by 
$A(e_i)=e_i$ for $i\leq k-1$ and 
$A(e_k)=\sum_i\nu(i)e_i$ where
$\nu$ is the weight function on $\{1,\dots,k-1\}$ defined by the
trainpath $\rho_0$. The map $A$ is a surjection onto a 
linear subspace of $\mathbb{R}^k$ of codimension one, moreover
$A$ preserves the linear subspace $V$ of $\mathbb{R}^k$ defined
by the switch conditions for $\tau$. In particular,
the corank of $A(V)$ in $V$ is at most one.
But $A(V)$ is contained in the space of 
solutions of the switch conditions on $\sigma$ 
and consequently its corank in $V$ is at least one. 
Thus equality holds. 

To summarize, we obtain that indeed,
${\rm dim}{\cal V}(\tau)=
{\rm dim}{\cal V}(\sigma)+1$ which
completes the proof of the first part of 
the lemma. The second part follows in exactly the same way, 
and its proof will be omitted.
\end{proof}

As a consequence we obtain

\begin{corollary}\label{dimensioncount}
\begin{enumerate}
\item
${\rm dim}{\cal V}(\tau)= 2g-2+\ell+p_1+p_2$
for every non-orientable recurrent
train track $\tau$ of topological type $(m_1,\dots,m_\ell;-p_1,p_2)$.
\item ${\rm dim}{\cal V}(\tau)= 2g-1+\ell+p_2$ for every orientable
recurrent train track $\tau$ of topological type $(m_1,\dots,m_\ell;0,p_2)$.
\end{enumerate}
\end{corollary}
\begin{proof}
The complementary components of a 
non-orientable recurrent train track $\tau$ of topological type $(m_1,\dots,m_\ell;-p_1,p_2)$ 
can be  subdivided
in $4g-4+p_1-\ell$ steps into trigons by successively adding small
branches. The once punctured bigon components can
be subdivided into a trigon and a once punctured monogon. 
A repeated application of the first part of 
Lemma \ref{simpleex} shows
that the resulting train track $\eta$ is maximal and
recurrent. Since 
for every maximal recurrent train track $\eta$ on a surface with $n=p_1+p_2$ 
punctures we have
${\rm dim}{\cal V}(\eta)=6g-6+2n$ (see \cite{PH92}), 
the first part of the corollary
follows from the formula in the first part of Lemma \ref{simpleex}.

To show the second part of the corollary, 
let $\tau$ be an orientable recurrent train track
of type $(m_1,\dots,m_\ell;0,p_2)$. Then $m_i$ is even
for all $i$.
Add a branch $b_0$ to $\tau$ which
cuts some complementary component of $\tau$ into a trigon
and a second polygon with an odd number of sides. The resulting
train track $\eta_0$ is not recurrent since a trainpath on $\eta_0$
can pass through $b_0$ at most once. However, we can add 
to $\eta_0$ another small branch $b_1$ which cuts some complementary
component of $\eta_0$ with at least 4 sides into a trigon and a second 
polygon such that the resulting
train track $\eta$ is non-orientable
and recurrent. The inward pointing tangent of $b_1$
is chosen in such a way that there is a trainpath traveling
through both $b_0$ and $b_1$. The counting measure of any simple 
closed curve which is carried by $\eta$ gives equal weight to the branches
$b_0$ and $b_1$. But this just means that ${\rm dim}{\cal V}(\eta)=
{\rm dim}{\cal V}(\tau)+1$ (see the proof of Lemma \ref{simpleex}
for a detailed argument). 
By the first part of the corollary,
we have ${\rm dim}{\cal V}(\eta)=2g-2+\ell +p_2+2$ and
consequently ${\rm dim}{\cal V}(\tau)=2g-1+\ell+p_2$ as claimed.
\end{proof}

\begin{definition}\label{deffullyrec} 
A train track $\tau$ of topological type $(m_1,\dots,m_\ell;-p_1,p_2)$
which carries
a minimal large geodesic lamination 
$\nu\in {\cal L\cal L}(m_1,\dots,m_\ell;-p_1,p_2)$
is called \emph{fully recurrent}.
\end{definition}

\begin{remark}\label{fullnecessary}
It is not hard to construct train tracks which are 
recurrent but not fully recurrent. Since this fact is 
not important for what follows we do not give explicit examples
here.

If a train track $\eta$ is carried by a train track $\tau$, then the
identity of $S$ induces a map from the set of complementary
components of $\tau$ into the set of complementary components 
of $\eta$. Thus up to homotopy, the complementary components
of $\tau$ are obtained from the complementary components of $\eta$
by subdivision. In particular, the number of complementary 
components of $\tau$ is not smaller than the number of 
complementary components of $\eta$, and 
if $\nu\in {\cal L\cal L}(m_1,\dots,m_\ell;-p_1,p_2)$ is carried by a
train track $\tau$ of topological type $(m_1,\dots,m_\ell;-p_1,p_2)$, 
then a carrying map $\nu\to \tau$ is surjective.
\end{remark}
 
Note that by definition, a fully recurrent train track is connected
and fills up $S$. Since a minimal geodesic lamination
supports a transverse measure, 
a fully recurrent train track $\tau$ is recurrent. 

%

There are two simple ways to modify a fully recurrent
train track $\tau$
to another fully recurrent train track.
Namely, if $b$ is a mixed branch of $\tau$ 
then we can \emph{shift} $\tau$
along $b$ to a new train track $\tau^\prime$. 
This new train track carries $\tau$ and hence it 
is fully recurrent since it carries 
every geodesic lamination
which is carried by $\tau$ \cite{PH92,H09}.

Similarly, 
if $e$ is a large branch of $\tau$ then we can perform a
right or left \emph{split} of $\tau$ at $e$
as shown in Figure A. 
\begin{figure}[ht]
\begin{center}
\psfrag{Figure A}{Figure C} 
\includegraphics[width=0.8\textwidth]{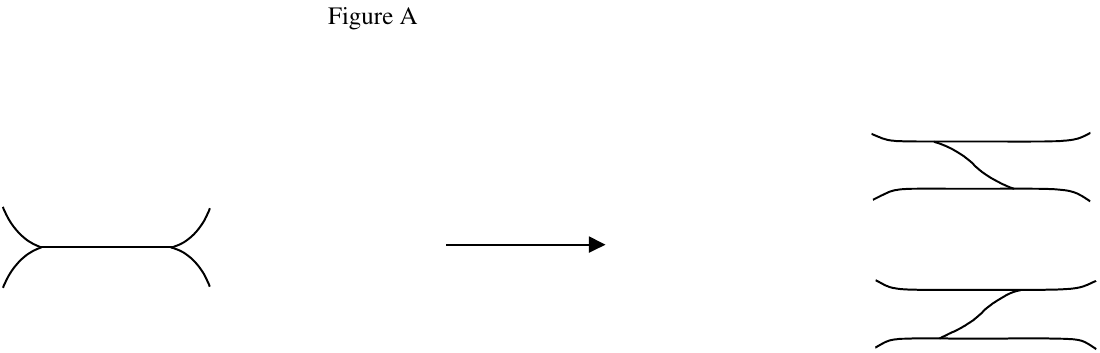}
\end{center}
\end{figure}
The new small branch in the split
track is called the \emph{diagonal} of the split.
A (right or left) split $\tau^\prime$ of a 
train track $\tau$ is carried
by $\tau$. 
If $\tau$ is of topological type 
$(m_1,\dots,m_\ell;-p_1,p_2)$, 
if $\nu\in {\cal L\cal L}(m_1,\dots,m_\ell;-p_1,p_2)$
is carried by $\tau$ and if $e$ is a large branch
of $\tau$, then there is a unique choice of a right or
left split of $\tau$ at $e$ such that the split track $\eta$ 
carries $\nu$. In particular, $\eta$ is fully recurrent. 
Note however that 
there may be a split of $\tau$ at $e$ such that
the split track is not fully recurrent any more
(see Section 2 of \cite{H09} for details).

To each train track $\tau$ 
which fills up $S$ one can
associate a \emph{dual bigon track} $\tau^*$ 
(Section 3.4 of \cite{PH92}).
There is a bijection between
the complementary components of $\tau$ and those
complementary components of $\tau^*$ which are
not \emph{bigons}, i.e. disks with two cusps at the
boundary. This bijection maps
a complementary component $C$ of $\tau$ which is an $n$-gon for some
$n\geq 3$ to an $n$-gon component of 
$\tau^*$ contained in $C$, and it maps a once punctured
monogon or bigon $C$ to a once punctured monogon 
or bigon contained in $C$.
If $\tau$ is orientable then the orientation of $S$ and
an orientation of $\tau$ induce an orientation on 
$\tau^*$, that is, $\tau^*$ is orientable.

There is a notion of carrying for bigon tracks which 
is analogous to the notion of carrying for train tracks.
Measured geodesic laminations which are carried by 
the bigon track $\tau^*$ 
can be described as follows.
A \emph{tangential measure} on a 
train track $\tau$ of type $(m_1,\dots,m_\ell;-p_1,p_2)$ 
assigns to a branch $b$ of $\tau$
a weight $\mu(b)\geq 0$ such that for every
complementary $k$-gon of $\tau$ or once punctured bigon 
with consecutive sides
$c_1,\dots,c_k$ and total mass $\mu(c_i)$ (counted
with multiplicities) the following holds true.
\begin{enumerate}
\item $\mu(c_i)\leq \mu(c_{i-1})+
\mu(c_{i+1})$.
\item $\sum_{i=j}^{k+j-1}(-1)^{i-j}\mu(c_i)\geq 0$, $j=1,\dots,k$.
\end{enumerate}
The complementary once punctured monogons 
define no constraint on tangential measures.
Our definition of tangential measure on $\tau$ is stronger than the
definition given on p.22 of \cite{PH92} and corresponds to
the notion of a \emph{metric} as defined 
on p.184 of \cite{P88}. We do not use this terminology here since
we find it misleading. 

The space of
all tangential measures on $\tau$ has the structure
of a convex cone in a finite dimensional real vector space.
By Lemma 2.1 of \cite{P88},
every tangential measure on $\tau$ determines a simplex of
measured geodesic laminations which
\emph{hit $\tau$ efficiently}. 
The dimension of this simplex equals the number of
complementary components of $\tau$ with an even number of sides.
The supports of these
measured geodesic laminations 
are carried by the bigon track $\tau^*$,
and every measured geodesic
lamination which is carried by $\tau^*$ can be obtained in this way.
The train track $\tau$ is called \emph{transversely recurrent}
if it admits a tangential measure which is positive on
every branch.

In general, 
a measured geodesic lamination
$\nu$ which hits $\tau$ efficiently does not determine 
uniquely a tangential measure on $\tau$ either.
Namely, let
$s$ be a switch of $\tau$ and let $a,b,c$ be the 
half-branches of $\tau$ incident on $s$ and such that
the half-branch $a$ is large. If $\beta$ is a tangential
measure on $\tau$ and if $\nu$ is a measured geodesic lamination
 in the simplex determined by $\beta$ then it may
be possible to drag the switch $s$ across some of 
the leaves of $\nu$ and modify the tangential
measure $\beta$ on $\tau$ to a tangential measure
$\mu\not=\beta$. Then $\beta-\mu$ is a multiple of 
a vector of the form $\delta_a-\delta_b-\delta_c$ where
$\delta_w$ denotes the function on the
branches of $\tau$ defined by $\delta_w(w)=1$ and
$\delta_w(a)=0$ for $a\not=w$.

\begin{definition}\label{largett}
Let $\tau$ be a train track of topological type 
$(m_1,\dots,m_\ell;-p_1,p_2)$. 
\begin{enumerate}
\item $\tau$  
is called \emph{fully transversely recurrent}
if its dual bigon track  
$\tau^*$ carries a minimal large geodesic lamination
$\nu\in {\cal L\cal L}(m_1,\dots,m_\ell;-p_1,p_2)$. 
\item 
$\tau$ is called \emph{large} if
$\tau$ is fully recurrent and fully transversely recurrent.
\end{enumerate}
\end{definition}

For a large train track $\tau$ let 
${\cal V}^*(\tau)\subset {\cal M\cal L}$ 
be the set of all measured geodesic
laminations whose support is carried by $\tau^*$. Each of 
these measured geodesic 
laminations corresponds to a family of tangential measures on 
$\tau$.  With this identification,
the pairing
\begin{equation}\label{intersectionpairing}
(\nu,\mu)\in {\cal V}(\tau)\times 
{\cal V}^*(\tau)\to \sum_b\nu(b)\mu(b)
\end{equation}
is just the restriction of the intersection form
on measured lamination space
(Section 3.4 of \cite{PH92}).
Moreover, 
${\cal V}^*(\tau)$ is naturally homeomorphic to 
a convex cone in a real vector space. The dimension of this cone
coincides with the dimension of ${\cal V}(\tau)$.



From now on we 
denote by ${\cal L\cal T}(m_1,\dots,m_\ell;-p_1,p_2)$ the set of all
isotopy classes of large train tracks on $S$ of type
$(m_1,\dots,m_\ell;-p_1,p_2)$.

\section{Combinatorial models for components of strata}\label{strata}

The goal of this section is to relate large train tracks to components
of strata of abelian or quadratic differentials. 

For a closed oriented surface 
$S_{g,n}$ of genus $g\geq 0$ with $n\geq 0$ marked points 
(punctures) let $\tilde{\cal Q}(S_{g,n})$ be the
bundle of \emph{marked} area one holomorphic
quadratic differentials with either a  simple pole or a regular
point at 
each of the marked points and no other pole
over the \emph{Teichm\"uller space} ${\cal T}(S_{g,n})$ 
of marked complex structures
on $S_{g,n}$.

Fix a complete hyperbolic metric on $S_{g,n}$ of
finite area. A quadratic differential 
$q\in \tilde{\cal Q}(S_{g,n})$ is 
determined by a pair $(\lambda^+,\lambda^-)$ of 
measured geodesic laminations 
which \emph{bind} $S$, which means that  
we have $\iota(\lambda^+,\mu)+
\iota(\lambda^-,\mu)>0$ for every measured geodesic
lamination $\mu$.
The \emph{vertical} measured geodesic 
lamination $\lambda^+$ for $q$
corresponds to the equivalence class of the 
\emph{vertical measured
foliation} of $q$. 
The \emph{horizontal} measured geodesic lamination
$\lambda^-$ for $q$ corresponds to the equivalence
class of the \emph{horizontal measured foliation} of $q$.
These foliations are the pull-back of the foliation of $\mathbb{C}$
into straight lines parallel to the imaginary or real axis, respectively,
by a system of charts on the complement of the singular points
of $q$ for which $q$ takes the form $dz^2$ (or $dz$ if $z$ is a holomorphic
one-form).

For $p_1\leq n$, $p_2=n-p_1$ and $\ell\geq 1$, 
an $\ell$-tuple $(m_1,\dots,m_\ell)$ of positive integers 
$1\leq m_1\leq \dots \leq m_\ell$ with $\sum_im_i=4g-4+p_1$
defines a \emph{stratum} $\tilde{\cal Q}(m_1,\dots,m_\ell;-p_1,p_2)$ 
in $\tilde{\cal Q}(S_{g,n})$. This stratum consists 
of all marked 
quadratic differentials with $p_1$ simple poles, $p_2$ regular marked points and
$\ell$ zeros of 
order $m_1,\dots,m_\ell$. We require that
these differentials are not squares of holomorphic
one-forms. 
The stratum is a 
complex manifold of dimension 
\begin{equation}\label{h}
h=2g-2+\ell+p_1+p_2.\end{equation}
In general, such a stratum is not connected, 
but most strata have only finitely many connected components
\cite{CS21,H21}. These components are permuted by the
mapping class group ${\rm Mod}(S_{g,n})$ of $S_{g,n}$. 

The closure in $\tilde{\cal Q}(S_{g,n})$ of 
a stratum is a union of components of strata. 
As strata are invariant
under the action of the mapping class group
${\rm Mod}(S_{g,n})$ of $S_{g,n}$, 
they project to strata in the moduli space
${\cal Q}(S_{g,n})=\tilde{\cal Q}(S_{g,n})/{\rm Mod}(S_{g,n})$ 
of quadratic differentials on $S_{g,n}$. 
Denote by
${\cal Q}(m_1,\dots,m_\ell;-p_1,p_2)$ the projection of the
stratum 
$\tilde {\cal Q}(m_1,\dots,m_\ell;-p_1,p_2)$. The strata in moduli
space need not be connected, but their connected
components have been identified by 
Lanneau \cite{L08}. A stratum in ${\cal Q}(S_{g,n})$ has
at most two connected components. The number of components of 
the stratum ${\cal Q}(m_1,\dots,m_\ell;-p_1,p_2)$ equals the number of 
components of ${\cal Q}(m_1,\dots,m_\ell;-p_1,0)$.

Similarly, 
let $\tilde{\cal H}(S_{g,n})$ be the bundle of marked 
holomorphic one-forms over Teichm\"uller space
${\cal T}(S_{g,n})$ of $S_{g,n}$. Each of the 
marked points of $S_{g,n}$ is required to 
be a regular marked point for the differential. In particular, the bundle
is non-empty only if $g\geq 1$.  
For an $\ell$-tuple $k_1\leq \dots \leq k_\ell$
of positive integers with $\sum_ik_i=2g-2$, the stratum 
$\tilde{\cal H}(k_1,\dots,k_\ell;n)$ of marked holomorphic one-forms
on $S$ with $\ell$ zeros of order $k_i$ $(i=1,\dots,\ell)$
and $n$ regular marked points 
is a complex manifold of dimension
\begin{equation}\label{h2} h=2g-1+\ell+n.\end{equation} 
It projects to a stratum 
${\cal H}(k_1,\dots,k_\ell;n)$ in the moduli space
${\cal H}(S_{g,n})$ of area one holomorphic one-forms on $S_{g,n}$. 
Strata of holomorphic one-forms in moduli space
need not be connected, 
but the number
of connected components of a stratum is at most three
\cite{KZ03}.

We continue to use the assumptions
and notations from Section \ref{traintracks}. 
For a marked large train track
$\tau\in {\cal L\cal T}(m_1,\dots,m_\ell;-p_1,p_2)$ 
let \[{\cal Q}(\tau)\subset \tilde{\cal Q}(S_{g,n})\] be 
the set of all marked quadratic differentials 
whose horizontal measured geodesic lamination 
is contained in ${\cal V}(\tau)$ via the identification of 
${\cal V}(\tau)$ with a (not necessarily open)
cone in ${\cal M\cal L}$ 
and whose vertical
measured geodesic lamination is carried by the dual
bigon track $\tau^*$ of $\tau$. Since $\tau$ and $\tau^*$ both
carry a minimal large geodesic lamination, and such a lamination
supports a transverse measure and fills $S=S_{g,n}$, for a large
train track $\tau$ on
$S=S_{g,n}$ the set ${\cal Q}(\tau)$ is not empty.
Recall that no geodesic lamination can be carried
by both $\tau$ and $\tau^*$.

%

Given two measured laminations $(\mu,\nu)$ which bind $S$, 
it is in general not easy to determine the stratum of 
the quadratic or abelian differential $z$ determined 
by $(\mu,\nu)$ due to possibility of horizontal or vertical 
\emph{saddle connections}. Such a saddle connection is 
a geodesic segment for the singular euclidean metric defined by $z$ 
which connects two singular points of $z$ 
(here we exclude a regular marked point) and does not contain a singular
point in its interior. 
The next lemma shows that train tracks can to used to this end.

\begin{lemma}\label{containedinstratum}
\begin{enumerate}
\item 
Let $\tau\in {\cal L\cal T}(m_1,\dots,m_\ell;-p_1,p_2)$ be non-orientable and
let $q\in {\cal Q}(\tau)$. If the support of the horizontal measured
geodesic lamination of $q$ is contained in
${\cal L\cal L}(m_1,\dots,m_\ell;-p_1,p_2)$ then
$q\in \tilde {\cal Q}(m_1,\dots,m_\ell;-p_1,p_2)$.
\item Let $\tau\in {\cal L\cal T}(m_1,\dots,m_\ell;0,p_2)$ be orientable and
let $q\in {\cal Q}(\tau)$. If the support of the horizontal measured
geodesic lamination of $q$ is contained in
${\cal L\cal L}(m_1,\dots,m_\ell;0,p_2)$ then
$q\in \tilde {\cal H}(m_1/2,\dots,m_\ell/2;p_2)$.
\end{enumerate}
\end{lemma}
\begin{proof} A marked quadratic differential 
$z\in \tilde {\cal Q}(S_{g,n})$ 
defines a singular euclidean metric on $S_{g,n}$.
A singular point for $z$ 
is a zero or a pole or a marked regular point. 
A \emph{separatrix} is a maximal geodesic segment or ray 
which begins at a singular point and does not contain 
a singular point in its interior. 

The complex structure on $S_{g,n}$ underlying $z$ determines a complete finite
area hyperbolic metric $h$ on $S_{g,n}$ with cusps at the $p_1$ 
marked points appearing in the definition. 
Let $\xi$ be the support of the horizontal measured geodesic lamination
of the quadratic differential $z$, realized in the hyperbolic metric $h$.
By \cite{L83}, the geodesic lamination  $\xi$
can be obtained from the horizontal foliation of $z$ by
cutting $S_{g,n}$ open
along each horizontal separatrix and straightening the remaining 
leaves so that they become geodesics for $h$.
In particular, up to homotopy, a horizontal saddle connection $s$ of $z$  is
contained in the interior of a complementary component $C$ of $\xi$
which is uniquely determined by $s$.

Let $\tau\in {\cal L\cal T}(m_1,\dots,m_\ell;-p_1,p_2)$ 
be non-orientable. Let
$q\in {\cal Q}(\tau)$ and assume that the 
support ${\rm supp}(\mu)$ of the horizontal
measured geodesic lamination 
$\mu\in {\cal V}(\tau)$ of $q$ is 
contained in  
${\cal L\cal L}(m_1,\dots,m_\ell;-p_1,p_2)$.
Then ${\rm supp}(\mu)$ is non-orientable
since otherwise $\tau$ inherits an orientation from ${\rm supp}(\mu)$
via a carrying map ${\rm supp}(\mu)\to \tau$. 
Since ${\rm supp}(\mu)\in 
{\cal L\cal L}(m_1,\dots,m_\ell;-p_1,p_2)$,
the orders of the zeros of the quadratic differential $q$ are obtained from
the orders $m_1,\dots,m_\ell$ by subdivision. 
Moreover, $q\in \tilde {\cal Q}(m_1,\dots,m_\ell;-p_1,p_2)$ if and 
only if this subdivision is trivial, which is equivalent to stating that 
$q$ does not have any horizontal saddle connection.

To show that this is indeed the case note first that 
the closure in $\tilde {\cal Q}(S_{g,n})$ 
of the stratum $\tilde {\cal Q}(m_1,\dots,m_\ell;-p_1,p_2)$ 
 is the union of 
$\tilde {\cal Q}(m_1,\dots,m_\ell;-p_1,p_2)$
with components of strata obtained by colliding some singular points.
Thus 
it suffices to find a sequence
$q_j$ of marked quadratic differentials which are contained in the closure
of $\tilde {\cal Q}(m_1,\dots,m_\ell;-p_1,p_2)$ and such that 
$q_j\to q$. 

For the construction of such a sequence, 
let $\beta_j\in {\cal V}^*(\tau)$ be a sequence of 
rational points, that is, measured geodesic
laminations supported on simple closed
multicurves, so that $\beta_j$ converges 
as $j\to \infty$ to the
vertical measured lamination of $q$. Such a sequence 
exists since rational points are dense in 
${\cal V}^*(\tau)$. 
As $\mu$ is minimal and fills $S_{g,n}$,
for all $j$ the pair
$(\mu,\beta_j)$ binds $S_{g,n}$ 
(since the only measured laminations on $S_{g,n}$ whose intersection 
with $\mu$ vanish have the same support as $\mu$)
and hence defines a quadratic differential 
$q_j\in \tilde{\cal Q}(S_{g,n})$ with $q_j\to q$. 
Our goal is to show that $q_j$ is contained in the closure
of $\tilde {\cal Q}(m_1,\dots,m_\ell;-p_1,p_2)$ and hence
in $\tilde {\cal Q}(m_1,\dots,m_\ell;-p_1,p_2)$ 
by the choice of $\mu$.

Consider for the moment an arbitrary quadratic differential 
$u$ with horizontal 
measured geodesic lamination $\mu$ which admits a horizontal saddle 
connection $\alpha$ connecting 
the zeros $x_1,x_2$. The weight deposited
on $\alpha$ 
by the transverse measure of 
the vertical measured geodesic lamination of $u$ is 
positive. 
By Remark \ref{fullnecessary} and the discussion in the beginning of this proof, 
there exists a homotopy of $S_{g,n}$ which maps $\mu$ onto $\tau$ and
which maps $x_1,x_2$ into a (uniquely determined) complementary
component $C$ of $\tau$. The component 
$C$ has at least four sides, and if $D$ is the complementary component
of $\mu$ which corresponds to $C$, 
then there exists a pair of non-adjacent sides $a,b$ of $D$ corresponding to 
a pair of non-adjacent sides of $C$ such that the transverse measure
of the vertical measured geodesic lamination of $u$
gives positive mass to geodesic
lines whose intersection components
with $D$ have endpoints on the sides $a,b$. 

We use this observation as follows.
For each $j$,
the simple closed multicurve $\beta_j$ can be homotoped to
a collection of closed trainpaths on the dual bigon track
$\tau^*$ of $\tau$. These paths 
intersect $\tau$ transversely in interior points of 
branches. If $C$ is any component of 
$S_{g,n}-\tau$, then any component of the intersection of 
$\beta_j$ with 
$C$ has its endpoints on consecutive sides of $C$
(see Section 3.4 of \cite{PH92} for details on this fact).

The geodesic lamination ${\rm supp}(\mu)$ lifts 
to a geodesic lamination $\hat \mu$ 
in the hyperbolic plane $\mathbb{H}^2$ which is the universal covering
of $S_{g,n}$, equipped with a complete finite volume hyperbolic  metric.
The lift of $D$ is a $\pi_1(S_{g,n})$-invariant 
collection of ideal polygons in $\mathbb{H}^2$.
Trainpaths on $\tau^*$ lift to uniform
quasi-geodesic in $\mathbb{H}^2$ 
which uniformly fellow-travel their geodesic representatives. 
Thus if $\hat D\subset \mathbb{H}^2$ 
is a component of the preimage of $D$, then the geodesic representatives
of  the lifts of the trainpaths corresponding to a component
of the multi-curve $\beta_j$ do not 
contain any subarc crossing through $\hat D$, with endpoints on non-adjacent 
sides of $\hat D$.  
Together with 
the discussion in the previous paragraph, we conclude 
that $q_j$  does not have a horizontal saddle connection and hence
it is contained in the closure of
$\tilde {\cal Q}(m_1,\dots,m_\ell;-p_1,p_2)$ which is what we wanted to show.

This yields the first part of the lemma, and the second part follows
with precisely the same argument. 
\end{proof}

We use Lemma \ref{containedinstratum} to show

\begin{proposition}\label{structure}
\begin{enumerate}
\item 
Let $\tau\in {\cal L\cal T}(m_1,\dots,m_\ell;-p_1,p_2)$ be
a large non-orien\-ta\-ble train track. Then 
there is a component $\tilde{\cal Q}$ of 
$\tilde{\cal Q}(m_1,\dots,m_\ell;-p_1,p_2)$ such that
${\cal Q}(\tau)$ 
is the closure 
in $\tilde{\cal Q}(S_{g,n})$  of 
an open path connected subset of $\tilde{\cal Q}$.
\item For every large orientable train track 
$\tau\in {\cal L\cal T}(m_1,\dots,m_\ell;0,n)$ 
there is a component $\tilde{\cal Q}$ of 
$\tilde{\cal H}(m_1/2,\dots,m_\ell/2,n)$
such that ${\cal Q}(\tau)$ is the closure in $\tilde {\cal H}(S_{g,n})$ of
an open path connected subset of $\tilde{\cal Q}$.
\end{enumerate}
\end{proposition}
\begin{proof} In the proof of the proposition, we do not
distinguish between the orientable and the non-orientable case.

Let $\tau\in {\cal L\cal L}(m_1,\dots,m_\ell;-p_1,p_2)$ and
let $\mu\in {\cal V}(\tau)$ be such that the support
${\rm supp}(\mu)$ of $\mu$ is contained
in ${\cal L\cal L}(m_1,\dots,m_\ell;-p_1,p_2)$. 
Notice that such a point is contained in the 
interior of ${\cal V}(\tau)$. 
If  $\beta\in {\cal V}^*(\tau)$ is arbitrary
then the measured geodesic laminations
$\mu,\beta$ bind $S_{g,n}$ (since the support of 
$\beta$ is different from the support of $\mu$ and 
${\rm supp}(\mu)$ fills up $S_{g,n}$).  Hence
if we put $\hat \beta=\beta/\iota(\mu,\beta)$ then 
the pair $(\mu,\hat \beta)$ defines
a point $q(\mu,\hat \beta)\in {\cal Q}(\tau)$.
By Lemma \ref{containedinstratum}, we have 
$q(\mu,\hat \beta)\in \tilde{\cal Q}(m_1,\dots,m_\ell;-p_1,p_2)$. 

Recall that ${\cal V}^*(\tau)$ is homeomorphic to a cone
over a closed cell. The dimension of ${\cal V}^*(\tau)$ 
coincides with the
dimension of ${\cal V}(\tau)$ (see \cite{PH92} for a precise 
statement which shows that collapsing the bigons in $\tau^*$ yields
a large train track of the same type as $\tau$).
Let $V$ be the interior of this cell. By continuity and
invariance of domain, we conclude that the 
set $\{q(\mu,\beta)\mid \beta\in V\}$ is an open subset  
of the \emph{strong stable manifold} defined by $\mu$
in $\tilde {\cal Q}(m_1,\dots,m_\ell;-p_1,p_2)$.
In \emph{period coordinates}, such a
strong stable manifold consists of quadratic 
differentials in $\tilde {\cal Q}(m_1,\dots,m_\ell;-p_1,p_2)$ 
with the same real part. 

As measured geodesic laminations which are minimal and of 
the same topological type as $\tau$ are dense in the set of all
measured laminations in such a strong stable manifold
(see \cite{KMS86} for a comprehensive discussion of this fact),
we conclude that measured geodesic 
laminations with this property 
are dense in ${\cal V}^*(\tau)$.

Choose a measured geodesic lamination
$\nu\in {\cal V}^*(\tau)$ whose support 
${\rm supp}(\nu)$ is contained 
in ${\cal L\cal L}(m_1,\dots,m_\ell;-p_1,p_2)$. 
Using exactly the same reasoning as above, 
we deduce that for each $\alpha\in {\cal V}(\tau)$,
the pair $(\hat \alpha,\nu)$ defines a quadratic differential
$q(\hat \alpha,\nu)\in \tilde {\cal Q}(m_1,\dots,m_\ell;-p_1,p_2)$
where $\hat \alpha=\alpha/\iota(\alpha,\nu)$. 
Furthermore, the measured geodesic laminations whose
supports are minimal and of the same topological
type as $\tau$ are dense in ${\cal V}(\tau)$.

As a consequence, there exists a dense subset 
of $Q(\tau)$ which is contained in 
$\tilde {\cal Q}(m_1,\dots,m_\ell;-p_1,p_2)$.
As $\tilde {\cal Q}(m_1,\dots,m_\ell;-p_1,p_2)$ is locally
closed in $\tilde {\cal Q}(S_{g,n})$, the intersection
$\tilde {\cal Q}(m_1,\dots,m_\ell;-p_1,p_2)\cap {\cal Q}(\tau)$
is open and dense in ${\cal Q}(\tau)$. 
Thus to complete the proof of the proposition, it suffices to show
that the set of all pairs 
$(\alpha,\beta)\in {\cal V}(\tau)\times {\cal V}^*(\tau)$ 
which give rise to a differential $q(\alpha,\hat \beta)\in 
\tilde {\cal Q}(m_1,\dots,m_\ell;-p_1,p_2)$ is path connected.
Since area renormalization is continuous, 
to this end it suffices to construct paths of differentials 
whose areas may be different from one.

Thus let $(\alpha,\beta),(\alpha^\prime,\beta^\prime)$ be two pairs with 
$q(\alpha,\beta)\in \tilde {\cal Q}(m_1,\dots,m_\ell;-p_1,p_2)$ where
$q(\alpha,\beta)$ denotes the differential defined by $\alpha,\beta$, of area
$\iota(\alpha,\beta)$. 
Now binding $S_{g,n}$ is an open condition
for pairs of measured laminations, and
$\tilde {\cal Q}(m_1,\dots,m_\ell;-p_1,p_2)\cap {\cal Q}(\tau)$ 
is open in ${\cal Q}(\tau)$ by the above discussion. Since
moreover  
the set of all measured laminations whose support is of type
${\cal L\cal L}(m_1,\dots,m_\ell;-p_1,p_2)$ is dense in 
${\cal V}(\tau)$, there exist laminations 
$\mu,\mu^\prime\in {\cal V}(\tau)$ with the following properties.
\begin{enumerate}
\item ${\rm supp}(\mu),{\rm supp}(\mu^\prime)$ are of 
type ${\cal L\cal L}(m_1,\dots,m_\ell;-p_1,p_2)$.
\item There is a path $c,c^\prime:[0,1]\to {\cal V}(\tau)$ 
connecting $\alpha$ to $\mu$, 
$\alpha^\prime$ to $\mu^\prime$ so that 
for every $t\in [0,1]$, the pair $(c(t),\beta)$ and 
$(c^\prime(t),\beta^\prime)$ binds $S_{g,n}$, and 
the paths $t\to q(c(t),\beta)$ and $t\to q(c^\prime(t),\beta^\prime)$ 
are contained 
in $\tilde {\cal Q}(m_1,\dots,m_\ell;-p_1,p_2)$ up to area renormalization.
\end{enumerate}
 
By the beginning of
this proof, for every measured
lamination $\xi\in {\cal V}^*(\tau)$, 
the pair
$(\mu,\xi)$ determines a quadratic differential $q(\mu,\xi)\in  
\tilde {\cal Q}(m_1,\dots,m_\ell;-p_1,p_2)$
(up to area renormalization), 
and the same holds true
for the pair $(\mu^\prime,\xi)$.
Choose a measured lamination $\nu\in {\cal V}^*(\tau)$ whose
support is contained in 
${\cal L\cal L}(m_1,\dots,m_\ell;-p_1,p_2)$.
Using the discussion in the previous paragraphs,
the differential $q(\mu,\beta)$ 
can be connected to $q(\mu,\nu)$ by a path in $Q(\tau)$ which is contained
in $\tilde {\cal Q}(m_1,\dots,m_\ell;-p_1,p_2)\cap {\cal Q}(\tau)$, 
and there also is such a path connecting
$q(\mu^\prime,\beta^\prime)$ and $q(\mu^\prime,\nu)$.
Another application of this 
argument shows that the differential
$q(\mu^\prime,\nu)$ can be connected to $q(\mu,\nu)$ by a path 
in $\tilde {\cal Q}(m_1,\dots,m_\ell;-p_1,p_2)\cap {\cal Q}(\tau)$. 

Together we find that the differentials $q(\alpha,\beta)$ and 
$q(\alpha^\prime,\beta^\prime)$ can both be connected to $q(\mu,\nu)$ 
by a path in $\tilde {\cal Q}(m_1,\dots,m_\ell;-p_1,p_2)\cap Q(\tau)$
(up to area renormalization). As the pairs $(\alpha,\beta),
(\alpha^\prime,\beta^\prime)\in {\cal V}(\tau)\times {\cal V}^*(\tau)$ 
were arbitrarily chosen
with the property that they determine
quadratic differentials in $\tilde {\cal Q}(m_1,\dots,m_\ell;-p_1,p_2)$, 
the proposition follows.  
\end{proof}

The next proposition is a converse
to Proposition \ref{structure} 
and  shows that train tracks can be used to define
coordinates on components of strata.

\begin{proposition}\label{all}
\begin{enumerate}
\item For every $q\in \tilde{\cal H}(k_1,\dots,k_s;n)$
there is an orientable train track
$\tau\in {\cal L\cal T}(2k_1,\dots,2k_s;0,n)$ 
so that $q$ is an interior
point of ${\cal Q}(\tau)$.
\item For every
$q\in \tilde{\cal Q}(m_1,\dots,m_\ell;-p_1,p_2)$ 
there is a non-orientable train track
$\tau\in {\cal L\cal T}(m_1,\dots,m_\ell;-p_1,p_2)$ 
so that $q$ is an interior point of ${\cal Q}(\tau)$.
\end{enumerate}
Furthermore, if $q$ contains a horizontal cylinder then 
$\tau$ can be chosen in such a way that the core curve
of this cylinder is embedded in $\tau$.
\end{proposition}
\begin{proof}
Let $q\in \tilde{\cal Q}(m_1,\dots,m_\ell;-p_1,p_2)$ and let 
$\Sigma=\{u_1,\dots,u_s\}$ $(s=\ell+p_1+p_2)$ be the
singular set of $q$, that is, the union of the
zeros and poles and marked regular points.

Recall that $q$ defines a singular euclidean metric $h$ on $S_{g,n}$ as well as two 
measured foliations, the horizontal and the vertical measured foliation. 
If $x$ is a singular point of the metric $h$,
then $x$ is a cone point of cone angle
$k\pi$ for some positive integer
$k\not=2$. There are precisely $k$ horizontal and precisely $k$
vertical separatrices which begin at $x$. The zeros of the differential
correspond to cone points with cone angle $k\pi$ for some 
$k\geq 3$.

Choose a number $\epsilon >0$ which is smaller than $1/8$-th of the
distance in the metric $h$ between any two singular points. 
Let $u_i\in \Sigma$ be a singular point of cone angle $k\pi$ for some
$k\geq 1$. There exists a neighborhood $V_i$ of $u_i$ with the
following properties. The boundary $\partial V_i$ of $V_i$ is 
a polygon with $2k$ sides. The sides are alternating between vertical 
arcs of fixed length $\sigma <\epsilon/10$ and horizontal arcs. 
The midpoint of a vertical arc is a point of distance $\epsilon$
on a horizontal separatrix through $u_i$. Note that 
the polygon is uniquely determined by these requirements.

Out of the polygons $V_i$ $(i\leq s)$ we construct a 
train track $\eta_i$ with stops whose 
switches are the midpoints of the vertical sides of the polygon
$\partial V_i$. Thus each switch 
is a point of distance $\epsilon$ to the singular
point $u_i$ on a horizontal separatrix $\zeta_i$.

Two different switches 
on separatrices $\zeta_i^1,\zeta_i^2$ starting at $u_i$ 
are connected by
a branch in $\eta_i$ if the angle at $u_i$ between $\zeta_i^1,\zeta_i^2$
equals $\pi$, or, equivalently, if there is a path in $\partial V_i$ connecting
$\zeta_i^1,\zeta_i^2$ which travels through precisely one horizontal side
of $\partial V_i$.  
These branches are constructed in such a way that
all the vertical
sides of the polygons $\partial V_i$ are replaced by a cusp.
Furthermore, we require that all branches are
contained in $V_i$ and do not intersect $u_i$.
Figure B shows this construction.
\begin{figure}[ht]
  \begin{center}
\psfrag{Figure D}{Figure B} 
\includegraphics[width=0.7\textwidth] 
{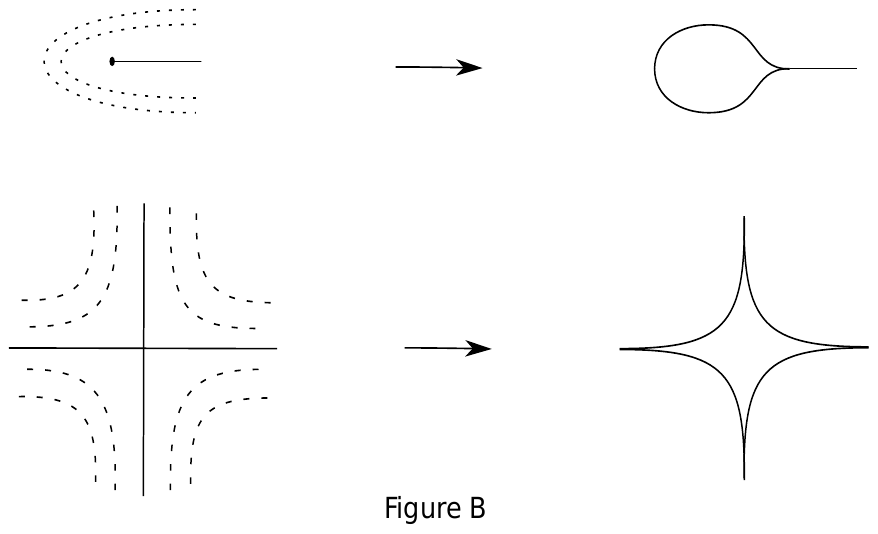}
\end{center}
\end{figure}

The construction can be done in such a way that 
$\eta_i$ is transverse
to the vertical measured foliation of $q$- more precisely, by
adjusting the constant $\sigma$ we can assume that the
tangent of $\eta_i$ is arbitrarily close to the horizontal
direction and that each branch of $\eta_i$ is an arc of arbitrarily 
small geodesic
curvature for the euclidean metric $h$.
There is a complementary
component $C_i$ of $V_i-\eta_i$
which is a polygon with $2k$ cusps. Its closure is contained in $V_i$
and meets $\partial V_i$ only at the cusps. It 
contains the singular point $u_i$. 
The cusps of the component are the vertices
of $\eta_i$. 
The component $C_i$ is 
a once punctured monogon
if $u_i$ is a pole, that is, if $k=1$,
a once punctured bigon if $u_i$ is a regular marked point, or 
an $m_i+2$-gon if $u_i$ is a zero of order $m_i$.

Let $\hat \eta$ be the union of the train tracks with stops $\eta_i$;
this union consists of $\ell+p_1+p_2$ connected components,
and it has $\sum_i(m_i+2)+p_1+2p_2$ vertices. The graph
$\hat \eta$ is transverse to the vertical foliation of $q$. By construction, 
$\hat\eta$ also is transverse to the straight line
foliation on $S$ defined by any direction for the singular euclidean metric
on $S$ which is
sufficiently close to the vertical direction.

A \emph{generalized bigon track} is a graph with all properties of a
train track except that we allow the existence of complementary
bigons, and we allow complementary annuli. 
Out of the train track with stops $\hat \eta$ we construct
a generalized bigon track $\eta$ on $S$ by inductively replacing a
stop by a switch and adding additional branches
as follows. 

For $i\leq s$ let 
$\beta$ be a vertical side of the polygon $\partial V_i$. 
Then for small enough $d>0$, 
$\beta$ is a side of an euclidean (right angled) rectangle $R_d$ in $S$ 
of width $d$ which intersects the polygonal disk $V_i$ precisely in $\beta$. 
We require that the interior of $R_d$ is disjoint
from the disks $V_j$. 
The area of $R_d$ equals $\sigma d$. 
Thus by consideration of area, 
there exists a smallest number $d_0>0$ such that the
vertical side $\beta^\prime$ of $R_{d_0}$ distinct from $\beta$ 
intersects one of the polygonal disks $V_j$. Then
$\beta^\prime\cap \partial V_j$
is a (possibly degenerate) subarc of a vertical side
$\xi$ of $\partial V_j$.

There are two possibilities. In the first case,
$\beta^\prime=\xi$. Then $V_i\cup R_{d_0}\cup V_j$ contains
a horizontal saddle connection joining $u_i$ to $u_j$.
Connect the cusp of the component $\eta_i$
of $\hat \eta$ contained in $V_i$ to the cusp of the
component $\eta_j$ of
$\hat \eta$ contained in $V_j$ by the subsegment of the
saddle connection which is contained in $R_{d_0}$.
This construction yields a new generalized bigon track  $\hat\eta^\prime$
with stops, and
the number of stops of $\hat \eta^\prime$ equals the number of stops of
$\hat \eta$ minus two.

The second possibility is that $\beta^\prime\cap \xi$ is a proper
subarc of $\xi$ (perhaps degenerate to a single point).
Then precisely one of the endpoints of $\xi$ is contained in 
$\beta^\prime$; denote this point by $z$.
The point $z$ is a vertex of the polygon $V_j$ and hence it is 
contained in a horizontal side of the polygon $V_j$
and determines
a branch $b$ of the train track
with stops $\hat \eta$. 
Connect the midpoint $y$ of the
side $\beta$ of $R_{d_0}$ (which is a stop of $\eta_i$) to 
an interior point of the branch $b$ with an arc $\nu$ 
in such a way that the union
$\eta_i\cup \nu\cup \eta_j$ is a generalized 
bigon track $\hat \eta^\prime$ with stops, and
that there is an arc of class $C^1$ contained in
$\hat \eta^\prime$ 
connecting the stop $y$ of $\eta_i$ to the endpoint of
the branch $b$ distinct
from $z$. In $\hat \eta^\prime$, 
the midpoint $y$ of the vertical side $\beta$ of 
$V_i$ (which was a stop in
$\eta_i$) is a trivalent switch. The generalized
bigon track $\hat \eta^\prime$ 
has fewer stops than $\hat \eta$.

Doing this construction with each of the 
stops of $\hat \eta$ replaces $\hat \eta$ by
a generalized bigon track $\eta$.
This can be done in such a way that each branch of $\eta$
is a smooth arc whose tangent line is everywhere close
to the horizontal subbundle of the tangent bundle of
$S_{g,n}-\Sigma$.

We show next that a complementary component of $\eta$ 
which does not contain a singular point of $q$ either is a bigon 
or an annulus with no singular point on the boundary. For this it
suffices to show that the Euler characteristic
of each complementary region which does not contain any marked
point vanishes. This Euler characteristic is computed by
giving each cusp in its boundary the value $-1$
(see \cite{PH92} for this computation).  
Thus a disk with 3 cusps
at the boundary has Euler characteristic $-1$.
The sum of the Euler characteristics of the 
complementary regions of $\eta$ containing a zero of $q$, a puncture or
a marked regular point equals
the Euler characteristic of $S_{g,n}$. By the Gauss Bonnet theorem, 
there are no complementary
monogons of $\eta$ without marked point in the interior.
Namely, up to adjusting the constant $\sigma$, the total geodesic curvature of any 
branch of $\eta$ can be chosen to be arbitrarily small, while the total number of 
branches is bounded from above by a constant independent of $\sigma$. 
Thus the Euler characteristic of every complementary
component of $\eta$ is non-positive and hence
the  Euler characteristic of every complementary component
not containing a singular point of $q$ or a marked regular point has
to vanish. Hence each such component either is a bigon or an annulus. 
An annulus component corresponds to a horizontal cylinder of $q$.

To construct a train track out of $\eta$ we begin with collapsing
successively the complementary bigons of $\eta $. Namely,  
the set of all directions for the flat metric defined by $q$ which 
are tangent to some saddle connection is
countable and hence we can find arbitrarily near the vertical
direction a direction 
which is not tangent to any saddle connection. 
By construction of $\eta$,
we may assume that this direction is transverse to
$\eta$. For simplicity
of exposition we will call this direction vertical in the sequel.
We use the singular foliation defined by this direction
as follows. 

Let $B$ be a complementary component of $\eta $ which is a
bigon. The boundary $\partial B$ of this bigon consists of two arcs
$a_1,a_2$ which are nearly
horizontal and which meet tangentially at their endpoints.
The vertical foliation is transverse to these
sides, and non-singular in the bigon. 

Let $x\in a_1$ be an interior point. We claim that $x$ is the starting point of a vertical segment
whose interior is contained in the interior of $B$ and whose endpoint
$F(x)$ is contained in the interior of the second side $a_2$ of $\partial B$. Namely, by transversality
and compactness, $x$ is the starting point of a vertical arc $\gamma$ whose interior is 
contained in the interior of $B$ and whose second endpoint $y$ is contained in 
a side of $\partial B$. 
If $y$ is contained in the same side $a_1$ of $\partial B$ as $\gamma(0)$
then $y$ bounds together with the subarc of $a_1$ connecting 
$x$ to $y$ an euclidean disk whose boundary consists of two smooth
arcs with small curvature 
which meet at the endpoints with an angle close to $\pi/2$. However, this violates
the Gauss Bonnet theorem. Thus indeed, $B$ is foliated by
vertical arcs with one endpoint on $a_1$ and the second endpoint on
$a_2$. 

Now observe that although the boundary of $B$ may not be
embedded in $S_{g,n}$ (we only know that the interior of $B$ is embedded), 
the two endpoints of any vertical arc as above
are distinct since there is no vertical saddle connection and hence no
vertical closed geodesic by assumption. This means that we can
collapse these vertical arcs
to points and collapse in this way the bigon $B$ to a single arc. 
Let $\hat \theta$ be the generalized bigon track
obtained in this way.

There is a 
map $F_0:S_{g,n}\to S_{g,n}$ of class $C^1$ which is
homotopic to the identity, which equals the identity
in a small neighborhood of the bigon $B$ and which maps $\eta$ to $\hat \theta$ 
by collapsing the vertical arcs crossing through $B$.
As the sides of $B$ are
nearly horizontal, the differential of the restriction of the 
collapsing map $F_0$ to each horizontal arc vanishes nowhere.

Using once more the fact that vertical trajectories do not contain loops, we can repeat this
process with any other bigon. In finitely many such steps
we construct a generalized
bigon track $\theta$ and a map 
$F:S_{g,n}\to S_{g,n}$ with the following properties.
\begin{enumerate}
\item $\theta$ does not have any complementary bigon components.
\item $F$ is homotopic to the identity and of class $C^1$. 
\item $F(\eta)=\theta$. 
\item The differential of the restriction of $F$
  to the horizontal foliation of $q$ 
  vanishes nowhere, and it maps the intersection of the horizontal
foliation of $q$ with the bigon complementary components of $\eta$ to 
smoothly immersed arcs in $\theta$.
\end{enumerate}

The generalized bigon track $\theta$ may not be a
train track as it may have complementary components which 
are annuli. However, the above construction can also be used to collapse annuli to circles. 
To this end let $A$ be a complementary annulus of $\theta$. 
By construction of 
$\theta$, $A$ is contained in a horizontal 
cylinder $C$ for $q$, and its closure does not contain a singular point of $q$.
Furthermore, its boundary curves are 
transverse to the vertical foliation. 

Let $a_1,a_2$ be the two boundary curves of $A$. 
For a point $x\in a_1$, there is a unique subarc 
$v(x)$ of a vertical trajectory
starting at $x$ which is entirely contained in $A$ and connects
$x$ to a point $\psi(x)$ contained in the boundary of $A$.
Using once more the Gauss Bonnet theorem, we conclude that 
in fact $\psi(x)\in a_2$. 
As there are no vertical cylinders 
and the closure of $A$ does not
contain singular points, we have $\psi(x)\not=x$. 
Furthermore, the arc $v(x)$ 
depends smoothly on $x$. 

By the discussion in the previous paragraph, 
for each $x\in a_1$ we can collapse the arc $v(x)$ to a point. 
The result is a new generalized bigon
track which carries the horizontal measured
geodesic lamination of $q$ and is such that the number of complementary
components which are annuli is strictly smaller than the 
number of annuli components of $\theta$. Repeating this
construction with all the finitely many annuli components of $S_{g,n}-\theta$, 
we construct
in this way from $\theta$ a train track $\tau$.
There is a map $F:S_{g,n}\to S_{g,n}$ of class $C^1$ with the following property. 
\begin{enumerate}
\item $F(\eta)=\tau$.
\item $F$ equals the identity near the singular points of $q$.
\item The restriction of the differential of $F$ to the tangent bundle of the horizontal
foliation of $q$ vanishes nowhere. 
\end{enumerate}
As a consequence, the train track $\tau$ carries the horizontal
measured foliation of $q$. Furthermore, the 
vertical measured geodesic
lamination of $q$ hits $\tau$ efficiently (see \cite{PH92}) and
hence it is carried by $\tau^*$. 
Each complementary component of $\tau$ contains
precisely one singular point of $q$, and the component is 
a $k+2$-gon if and only if the singular point
is a zero of order $k$. This yields that 
$\tau$ is of topological type $(m_1,\dots,m_\ell;-p_1,p_2)$. 

We are left with showing that  $\tau$ is large. Now by construction,
$\tau$ carries the horizontal geodesic lamination of $e^{is}q$
provided that $s$ is sufficiently
close to $0$. But the set of directions
for the singular euclidean metric defined
by $q$ so that the horizontal foliation in this direction is minimal and of the 
type predicted by the number and multiplicities of the
zeros of $q$ is dense \cite{KMS86}. This implies that 
$\tau$ carries a geodesic lamination which is minimal, large and of the same 
topological type as $q$. Similarly, for $s$ sufficiently close to zero,
the vertical measured geodesic lamination of $e^{is}q$ hits 
$\tau$ sufficiently. Thus as before, $\tau^*$ carries 
a minimal large geodesic lamination of the same topological type
as $\tau$. In other words, $\tau$ 
has all the properties required in the
proposition.

Now if $q$ is an abelian differential, then the horizontal and
vertical foliations of $q$ are orientable. As the initial 
generalized bigon track 
$\hat \eta$ is constructed from the horizontal foliation of  
$q$, it inherits an orientation from the horizontal foliation of
$q$. The collapsing construction uses the orientable
vertical foliation, and it is straightforward that this
construction respects orientations as well. Then the resulting
train track $\tau$ is orientable. 
\end{proof}

We summarize the discussion in this section as follows.

Let ${\cal Q}$ be a component of the stratum ${\cal Q}(m_1,\dots,m_\ell;-p_1,p_2)$
of ${\cal Q}(S_{g,n})$ (or of the stratum ${\cal H}(m_1/2,\dots,m_\ell/2;p)$ 
of ${\cal H}(S_{g,p})$) and let $\tilde {\cal Q}$ be
the preimage of ${\cal Q}$ in $\tilde {\cal Q}(S_{g,n})$
(or in $\tilde {\cal H}(S_{g,n})$). 
Then there is a collection
\[{\cal L\cal T}(\tilde{\cal Q})
\subset {\cal L\cal T}(m_1,\dots,m_\ell;-p_1,p_2)\]
of large marked train tracks $\tau$ of the same 
topological type as ${\cal Q}$ such that for every $\tau\in 
{\cal L\cal T}(\tilde {\cal Q})$, the set 
${\cal Q}(\tau)$
contains an open path connected
subset of $\tilde {\cal Q}$.

The set ${\cal L\cal T}(\tilde {\cal Q})$ 
is invariant under the action of the mapping class group.
Its quotient ${\cal L\cal T}({\cal Q})$ 
under this action is finite and is called
the \emph{set of combinatorial models for ${\cal Q}$}.
The set
\[
  \cup_{\tau\in {\cal L\cal T}(\tilde {\cal Q})}
 {\cal Q}(\tau)\] 
is closed, ${\rm Mod}(S)$-invariant and contains
$\tilde {\cal Q}$ as an open dense subset, that is, it coincides
with the closure of $\tilde {\cal Q}$ in 
$\tilde {\cal Q}(S_{g,n})$.

\begin{lemma}\label{splitinv}
Let ${\cal Q}$ be a component of a stratum, with preimage
$\tilde {\cal Q}$ in $\tilde {\cal Q}(S_{g,n})$, let 
$\tau\in {\cal L\cal T}(\tilde {\cal Q})$ and let $\eta$ be a large
train track of the same topological type as $\tau$ 
which is carried by $\tau$. 
Then $\eta\in {\cal L\cal T}(\tilde {\cal Q})$.
\end{lemma}
\begin{proof}
A point in ${\cal Q}(\tau)$ is defined by a pair $(\lambda,\nu)$ where
$\lambda\in {\cal V}(\tau)$ and where $\nu\in {\cal V}^*(\tau)$.
If we choose $\lambda$ in such a way that its support 
${\rm supp}(\lambda)$ is of the same topological type as $\tau$ and 
such that $\lambda$ is carried by the train track $\eta$, then 
$(\lambda,\nu)$ defines a 
differential in ${\cal Q}(\eta)\cap {\cal Q}(\tau)$.
It then follows from Proposition \ref{structure} that $\eta\in
{\cal L\cal T}({\cal Q})$.
\end{proof}

As a fairly immediate consequence of the above discussion
and Section 3 of \cite{H09}, we obtain a method to construct
large train tracks of a given topological type.
Namely, for a fixed choice of a complete hyperbolic 
metric on $S$ of finite volume and numbers $a>0,\epsilon >0$ there
is a notion of \emph{$a$-long} train track which 
\emph{$\epsilon$-follows} a large geodesic lamination 
$\lambda$. By definition, this means the following.
The \emph{straightening} of a 
train track $\tau$ with respect to the hyperbolic metric 
is obtained from $\tau$ by replacing each branch
$b$ by the geodesic segment which is homotopic with fixed endpoints to $b$.
We require that the length of each of the straightened edges is at least $a$, 
that their tangent lines are contained
in the $\epsilon$-neighborhood of the projectivized tangent bundle of 
$\lambda$ and that moreover the straightening of every
trainpath on $\tau$ is a piecewise geodesic whose exterior angles
at the breakpoints are not bigger than $\epsilon$. 

Lemma 3.2 of \cite{H09} shows that for every geodesic lamination
$\lambda$ on $S_{g,n}$ and every $\epsilon>0$ there is an $a$-long
generic transversely recurrent train track $\tau$ which carries
$\lambda$ and $\epsilon$-follows $\lambda$.

\begin{corollary}\label{followgood}
Let $\tau\in {\cal L\cal T}(\tilde{\cal Q})$ and let
$\lambda$ be a minimal large geodesic lamination of the same
topological type as $\tau$ which is carried by $\tau$.
Then for sufficiently small $\epsilon >0$, 
an $a$-long train track $\eta$ which $\epsilon$-follows
$\lambda$ is contained in ${\cal L\cal T}(\tilde {\cal Q})$.
\end{corollary}
\begin{proof}
By construction, if $\lambda$ is large, then for sufficiently small
$\epsilon$ and sufficiently large $a>0$, an $a$-long train track $\eta$
which $\epsilon$-follows $\lambda$ is of the same topological type as
$\lambda$. Furthermore,
$\eta$ carries a minimal  large geodesic
lamination of the same topological type as $\eta$ and hence
$\eta$ is fully recurrent and transversely recurrent.

If $\lambda$ is carried by a large train track $\tau$
then for sufficiently small $\epsilon>0$ and sufficiently
large $a>0$, $\eta$ is carried by $\tau$ (see Section 3 of \cite{H09}).
Then a large geodesic lamination which is carried by
$\tau^*$ is carried by $\eta^*$ and hence $\eta$ is large as claimed.
\end{proof}

\section{Components  of strata: The principal boundary}\label{sec:components}

In Section 3, a combinatorial
model for every component ${\cal Q}$ of a stratum in 
${\cal Q }(S_{g,n})$ or in ${\cal H}(S_{g,n})$ was constructed.
The purpose of this section is to refine this construction
and obtain models for 
specific types of degenerations of the
stratum.

We begin with introducing the degenerations we
are interested in. 
The framework for these degenerations is in the spirit of 
the  ``you see what you get''
partial compactification of strata introduced in \cite{MW15}.
Although we will not make use of this work,
we reproduce Definition 2.2 of \cite{MW15}.
In its formulation, $\Sigma_i$ is the singular set of 
the quadratic differential $q_i$ on the Riemann surface $X_i$.

\begin{definition}\label{degeneration}
Say that $(X_j,q_j,\Sigma_j)$ \emph{converges} to 
$(X,q,\Sigma)$ if there are decreasing neighborhoods $U_j\subset X$
with $\cap U_j=\Sigma$ such that the following holds.
There are maps $g_j:X-U_j\to X_j$ that are diffeomorphisms
onto their range, such that
\begin{enumerate}
\item $g_j^*(q_j)$ converges to $q$
in the compact open topology on $X-\Sigma$.
\item The injectivity radius at points not in the image of $g_j$
goes to zero uniformly in $j$.
 \end{enumerate}  
\end{definition}

With this definition, we allow to erase 
zero area components of a limiting surface with nodes.

There are two specific types of degenerating 
sequences which will 
be used in the sequel.

\noindent
{\bf 1) The shrinking half-pillowcase:} 

Let $q$ be a quadratic differential on $S_{g,n}$.
Choose a singular or marked regular point 
$x$ on $S_{g,n}$ and cut $S_{g,n}$ open along 
a geodesic segment $\alpha$ issuing from $x$ 
of length $s>0$. If $x$ is a singular point of $q$ then we
require that $x$ is the only singular point contained in
$\alpha$, and if $x$ is a regular point then we require
that $\alpha$ does not contain any singular point.
The cut open surface has a geodesic circle as boundary. Glue 
a foliated cylinder $C$ to this circle 
whose opposite boundary is divided into two arcs of the same 
length which are identified to form half of a pillowcase
as shown in Figure C.
\begin{figure}[ht]
\begin{center}
\psfrag{Figure E}{Figure E} 
\includegraphics 
[width=0.7\textwidth] 
{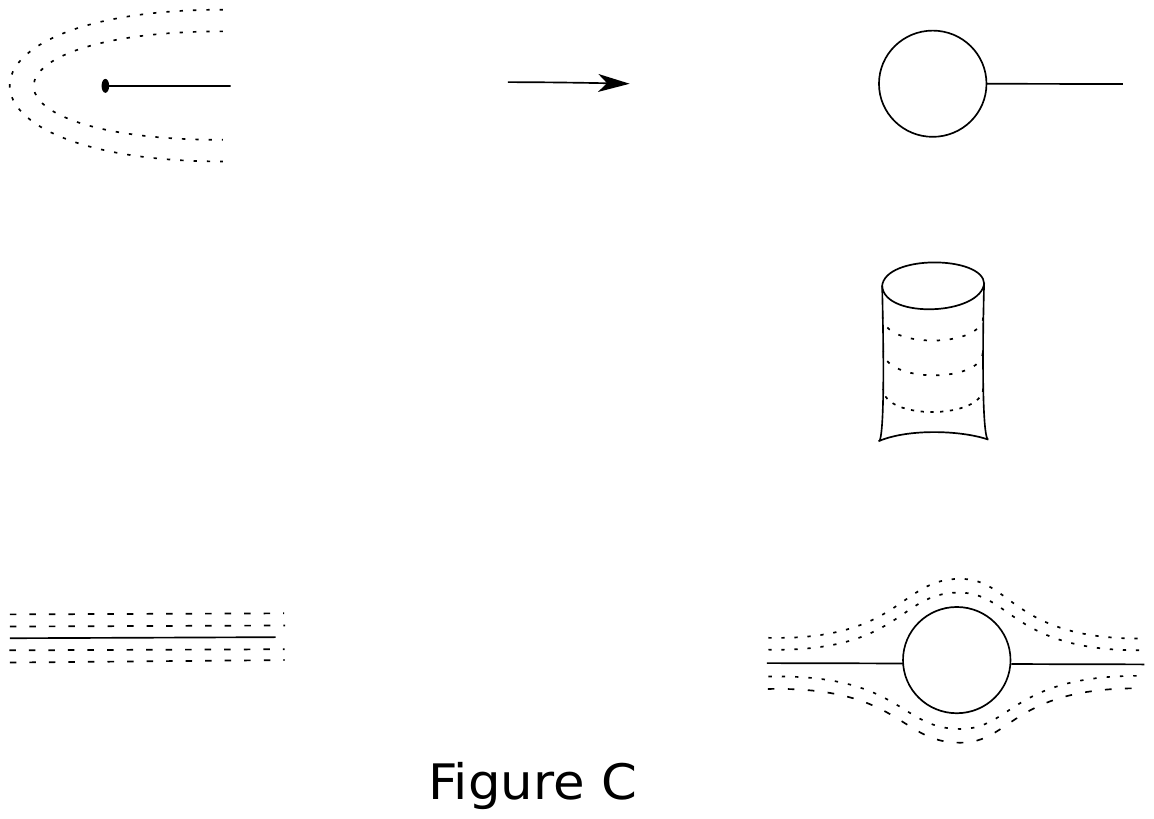}
\end{center}
\end{figure}
This does not change the genus of $S_{g,n}$, but it adds
two punctures to $S_{g,n}$, and it increases the cone angle 
of each of the endpoints of $\alpha$ by $\pi$. Note that 
for the fixed point $x$, the 
half-pillowcase is described by 4 real parameters:
The direction and the length of its cutting arc $\alpha$, 
the height of the cylinder and the position 
of one of the simple poles on the top of the half-pillowcase, which
is determined by the choice of a point on the 
(oriented) cutting arc $\alpha$.

We call a sequence of quadratic differentials
containing a half-pillowcase whose circumferences and heights tend to zero and 
which degenerate in the sense of Definition \ref{degeneration}
to the surface with the half-pillowcase removed a
\emph{shrinking half-pillowcase}.
The areas of the half-pillowcases, that is,
the products of their circumferences and 
heights, tend to zero.

\noindent
{\bf 2) The shrinking cylinder:}

Let again $q$ be a quadratic differential defining a flat metric on $S_{g,n}$.
Given two singular or marked regular points $x_1\not=x_2$ on $S_{g,n}$, cut
the surface $S_{g,n}$ open along
two geodesic arcs of the same length and the same direction starting at
$x_i$ 
with no singular point in the interior,
and glue the two 
boundary circles to the two boundary circles of a flat cylinder. 
The result is a singular flat metric on the surface $S_{g+1,n}$. 
The core curve of the attached cylinder in $S_{g+1,n}$ is
non-separating, and the direction of the geodesics in its core
equals the direction of the cutting arcs used in the construction.
We leave it to the reader to check that this construction depends again on 
4 real parameters (this fact is not needed in the sequel, see also
\cite{EMZ03}). 

There is a modification of this construction as follows.
Cut the surface $S_{g,n}$ open along a single geodesic arc
$\alpha$ issuing from a singular or marked regular point $x$ 
with at most one singular point on the boundary
and no interior singular point. Identify the endpoints of
$\alpha$; the resulting surface has two geodesic
boundary circles of the same length and the same direction.
Glue the boundary of a flat cylinder to these two boundary
circles. As before, the result of this construction is 
a singular flat metric on the surface $S_{g+1,n}$.

We call a sequence of flat surfaces containing a nonseparating cylinder
which degenerate to a surface with nodes by
shrinking the width and the heights of the cylinder to zero a
\emph{shrinking cylinder}. Note that the constructions 
discussed above change the area of the singular flat metric,
which can be corrected with the usual area renormalization.

Our goal is to construct
combinatorial models for quadratic or abelian 
differentials which capture 
these two types of 
degenerations to differentials on surfaces with nodes.
These models will then be used to 
construct periodic orbits of the Teichm\"uller flow 
in the thin part of moduli space.
Assume from now on that
$3g-3+n\geq 5$. This rules out spheres with at most $7$ punctures,
tori with at most $4$ punctures and a surface of genus $2$ with 
at most $1$ puncture.

\begin{definition}\label{basiccurve}
An essential simple closed curve $c$ on 
$S_{g,n}$ is called \emph{elementary} if either 
\begin{enumerate}
\item[a)]
$c$ is non-separating or 
\item[b)]
$n\geq 2$ and $c$ decomposes
$S_{g,n}$ into a surface $S_{g,n-1}$ and a twice punctured disk.
\end{enumerate}

An \emph{elementary pair} is a 
pair $(c_1,c_2)$ consisting of disjoint elementary curves
$c_1,c_2$ on $S$.
If both $c_1$ and $c_2$ are non-separating then we require that
$S_{g,n}-(c_1\cup c_2)$ is connected.
\end{definition}

Since $3g-3+n\geq 5$ by assumption, the complement 
$S_{g,n}-(c_1\cup c_2)$ of  
an elementary pair in $S=S_{g,n}$ contains a (unique) component 
which is not a three holed sphere.
In the sequel we tacitly identify a complementary
component of a curve $c$ in $S$ 
(or any curve system) with its metric completion,
that is, we view $S-c$ as a surface with boundary. 


\begin{definition}\label{primitivevertex}
A \emph{primitive vertex cycle} for 
a large train track $\tau$
is a simple closed curve $c$ embedded in $\tau$ which
consists of a large branch and a small branch. 
\end{definition}

If $c$ is a primitive vertex cycle in $\tau$ then there are 
two half-branches incident on the two
switches of $\tau$ in $c$ which are not contained in $c$.
Since $\tau$ is large by assumption, these two half-branches
lie on the two different sides of $c$ in an annulus
neighborhood of $c$ in $S$.
Namely, otherwise there is a complementary
component of $\tau$ containing a simple closed curve 
which
is neither contractible nor homotopic into a puncture. 

We call a
primitive vertex cycle of a large train track 
$\tau$ \emph{clean} if its underlying simple
closed curve $c$ is elementary and if moreover a branch
$b$ which is incident on a switch in $c$ and which
is not contained in $c$ 
satisfies one of the two following conditions.
\begin{enumerate}
\item $b$ is a small branch.
\item $n\geq 2$, $c$ is separating and 
$b$ is contained in the twice punctured disk component of $S-c$.
\end{enumerate}

As there are two types of elementary curves, there
are two types of clean vertex cycles. To relate these
types to the degeneration of quadratic differentials, 
note that removing from $\tau$ a clean vertex cycle $c$ as well as 
all the branches of $\tau$ adjacent to $c$ and branches
in a $3$-holed sphere component of $S-c$ 
yields a
train track $\tau^\prime$ 
on the complementary component $S_0$ of 
$S-c$ 
which is not a three holed sphere.

\noindent
{\sl Type I:} {\sl The shrinking half-pillowcase.}

If $n\geq 2$ and if $c$ is a separating
clean vertex cycle of $\tau$, then there is 
a complementary component $C$ for the  
train track $\tau^\prime$ on $S_0$
which is an annulus whose core curve is homotopic to $c$.
There is a component $\gamma$
of $\partial C$ embedded in
$\tau^\prime$, and $\gamma$ 
contains at least one cusp of $\tau^\prime$ (since otherwise  
$\tau$ has a complementary
component which is a bigon).

Let $S_0^\prime$ be the surface obtained from $S_0$ by
replacing the boundary circle of $S_0$ 
by a marked point (puncture).
The genus of $S_0^\prime$ coincides with the genus of $S$, and the
number of marked points has decreased by one. 
If the component $\gamma$ of $\partial C$  
contains at most two cusps, 
then $\tau^\prime$ is a large train track on 
$S_0^\prime$, where the puncture replacing the curve $c$ may be 
a marked regular point. 
By  Section \ref{strata}, $\tau^\prime$ 
determines a component of a stratum 
of differentials on $S^\prime=S_0^\prime$.  
Passing from $S^\prime$ back to $S$ 
corresponds to a shrinking half-pillowcase obtained by
cutting $S^\prime$ open along a geodesic segment
$\alpha$ for the flat
metric defined by a quadratic differential on $S^\prime$ 
starting at a simple pole or a marked regular point and not 
containing any other singular points.
If $\gamma$ contains at least three cusps then 
remove from $S_0^\prime$ the marked point enclosed by $\gamma$
and denote the resulting surface by $S^\prime$. 
In both cases, $\tau^\prime$ is a large train track on 
the surface $S^\prime$.

\noindent
{\sl Type II:} The shrinking cylinder. 

If $c$ is a non-separating clean vertex cycle of $\tau$ 
then both branches incident on $c$ are small.
The genus of the surface with boundary 
$S_0$ obtained by cutting $S$ open
along $c$ equals $g-1$. Let $S_0^\prime$ be obtained from 
$S_0$ by replacing the two boundary circles (which are 
copies of $c$) by a marked point. 
As before, $\tau^\prime$ defines a large train track 
on the surface $S^\prime$ which either coincides with $S_0^\prime$
or is obtained from $S_0^\prime$ by removing one or both of the 
special marked points. These choices depend on the
number of cusps of the complementary components
of $c$ in $\tau^\prime$.

Let as before ${\cal Q}$ be a component of a stratum of quadratic or
abelian differentials on $S$.
Call a  train track $\tau$ \emph{in special form for ${\cal Q}$}
if $\tau\in {\cal L\cal T}({\cal Q})$ 
and 
if there is an elementary pair $(c_1,c_2)$ for $S$ with
the following additional property. 
\begin{enumerate}
\item[$(*)$] $\tau$ contains
each of the curves $c_1,c_2$ as a clean vertex cycle, and 
the graph obtained by removal of $c_1,c_2$ and their 
adjacent branches as well as all branches contained in 
a 3-hold sphere component of $S-(c_1\cup c_2)$ 
is a large train track on a subsurface of $S$

\end{enumerate}

The rest of this section is devoted to the construction 
of train tracks in special form for all components of 
strata of abelian or quadratic differentials 
on $S_{g,n}$ without marked regular points provided that 
$3g-3+n\geq 5$.

For simplicity, write
${\cal Q}(m_1,\dots,m_\ell;-n)$ instead of 
${\cal Q}(m_1,\dots,m_\ell;-n,0)$,
and write ${\cal H}(k_1,\dots,k_s)$ instead of 
${\cal H}(k_1,\dots,k_s;0)$.
Motivated by the strategy in \cite{KZ03} and \cite{L08}, 
the idea is to start with explicit train tracks for 
components of strata on surfaces with small complexity 
and uses these train tracks as building blocks for the 
construction of 
train tracks for all strata. The first type of modification
consists in subdividing complementary components as follows.

Let $C$ be a complementary component of a train track 
$\eta$ which is a disk with 
$k\geq 4$ cusps on its boundary $\partial C$. 
Then $C$ can be subdivided
into two components by adding a small branch which connects
two non-adjacent sides of the component. 
The resulting train track $\tau$ is a simple extension of $\eta$ as defined
in Section 2.
If $\eta$ is orientable and
if the number of cusps of $\partial C$ at least six, then this
subdivision can be done 
in such a way that the components have
an even number of cusps and that $\tau$ is orientable as well.
By Proposition \ref{simpleex}, 
if $\tau,\eta$ are either both orientable or 
both non-orientable then
$\tau$ is large if and only if this holds true for $\eta$.
In the sequel we always choose subdivisions of complementary
components of orientable train tracks in such a way that
the resulting train track is orientable.

Following \cite{L08}, strata of quadratic differentials with at least
three simple poles are connected. 
We use this fact to observe

\begin{lemma}\label{spheretorus}
Components of strata of differentials
on the two-sphere $S^2$ with $n\geq 7$ simple poles 
or on the two-torus $T^2$ with $n\geq 4$ simple poles admit a train track
in special form.
\end{lemma}
\begin{proof} 
Figure D shows large train tracks with at least two clean
vertex cycles $c_1,c_2$ on $S_{0,n}$ for $n=6,7$
and a single complementary component which is not a once-punctured 
monogon. Such a  train track belongs to the stratum 
of meromorphic differentials
on $S=S^2$ with a single zero and  
$6$ or $7$ simple poles since such strata are connected \cite{L08}. 

The train track $\tau$ for the stratum of 
differentials on $S_{0,6}$ with $6$ poles is not 
in special form.
Namely, removal of the
vertex cycles $c_1,c_2$ 
as well as all adjacent 
branches and all branches in the three-holed sphere 
components of $S_{0,6}-(c_1\cup c_2)$ from $\tau$ 
does not result in a large
train track for a subsurface of $S_{0,6}$.

In contrast, the  train track $\tau$ for the stratum of differentials 
on $S_{0,7}$ with a single zero and 7 poles is in special form: 
The train track obtained from $\tau$ by removal of 
the vertex cycles $c_1,c_2$, all adjacent branches and all branches in the 
three-holed sphere components of $S_{0,7}-(c_1\cup c_2)$ 
is a large train track for the stratum of 
differentials on $S^2$ with 4 simple poles. 

To construct train tracks in special form for strata 
of differentials on $S^2$ with a single
zero and at least 8 simple poles just attach more 
copies of a circle enclosing two punctures and 
containing a once puncture monogon to one of the 
two train tracks shown in Figure D.
\begin{figure}[ht]
\begin{center}
\psfrag{a}{$c_1$}
\psfrag{b}{$c_2$}
\psfrag{c}{$c_3$}
\includegraphics 
[width=0.7\textwidth] 
{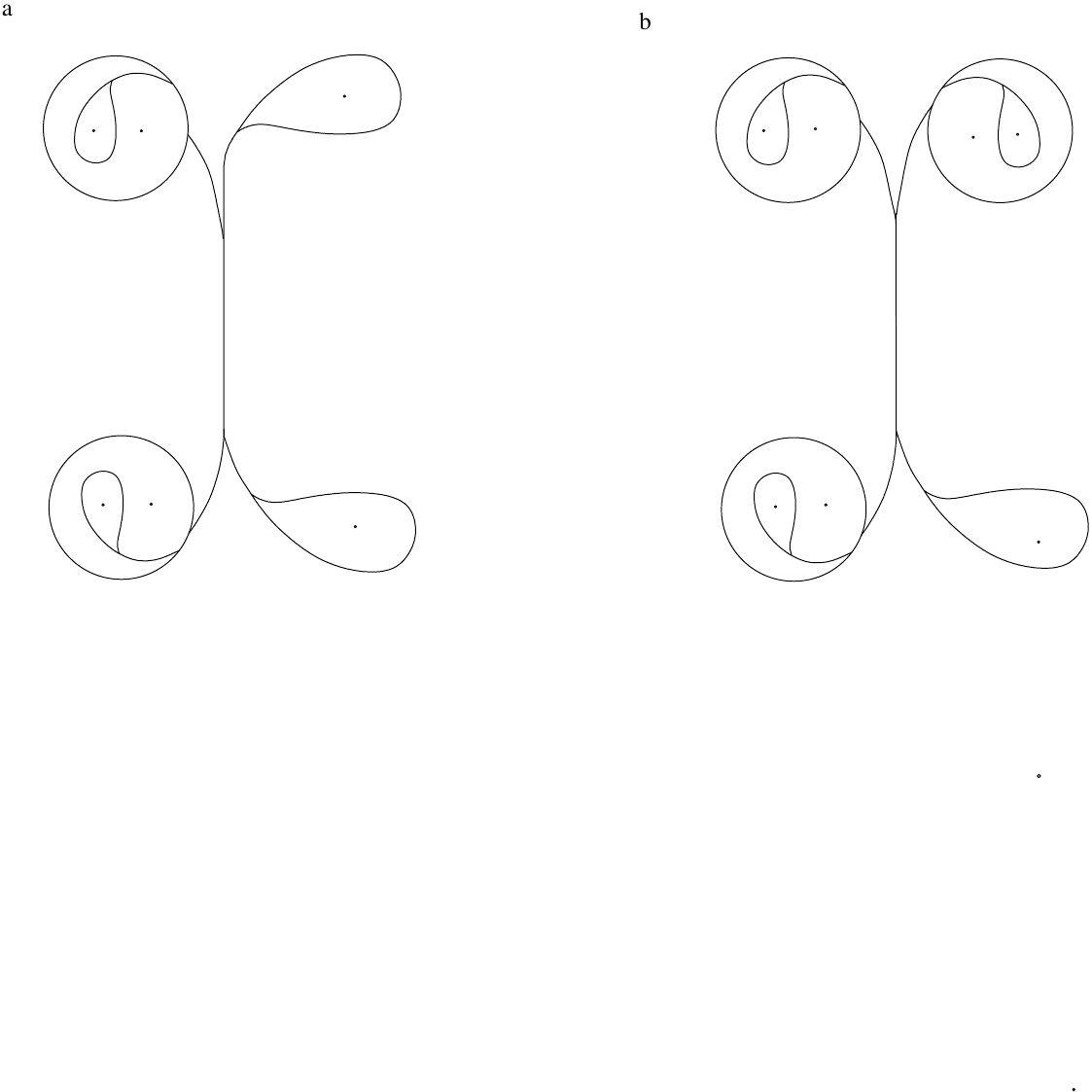}
\end{center}
\end{figure}
Train tracks in special form for arbitrary
strata of differentials on $S^2$
with at least 7 simple poles are obtained from the train tracks
for strata with a single zero by subdivision of complementary
components.

Figure E shows large train tracks 
containing at least two clean vertex cycles for the stratum of differentials
on the torus $T^2$ with a single zero and 
3 or 4 simple poles. As before, the train track for
the stratum with $4$ simple poles is in special form. 
To construct train tracks
in special form for a stratum of differentials on 
a torus with a single zero and 
at least 5 simple poles, attach more copies of
a circle enclosing two punctures and containing a 
once punctured monogon.
\begin{figure}[ht]
\begin{center}
\psfrag{a}{$c_1$}
\psfrag{b}{$c_2$}
\psfrag{c}{$c_3$}
\psfrag{Figure B}{Figure G}
\includegraphics 
[width=0.7\textwidth] 
{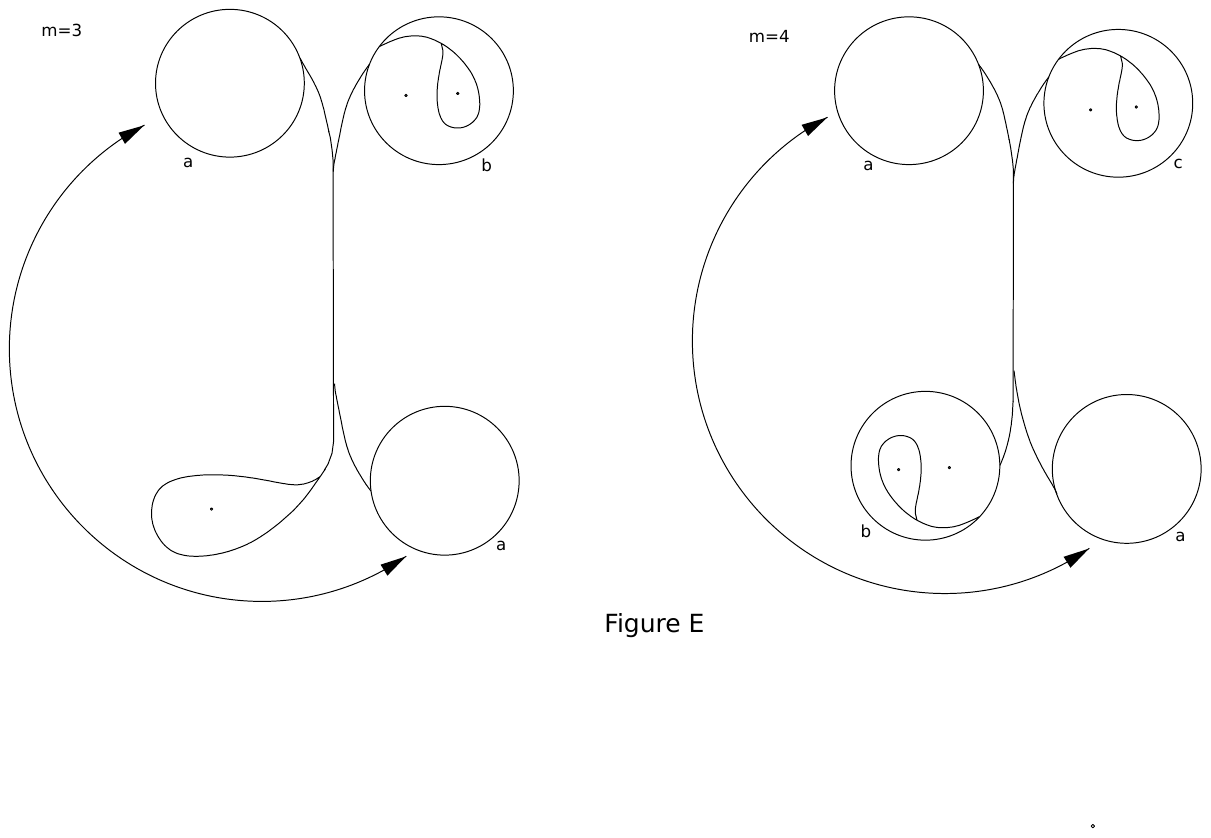}
\end{center}
\end{figure}
As before, train tracks in special form for arbitrary 
strata of differentials on the torus
with at least 4 simple poles are constructed 
from these train tracks by
subdivision of complementary components. 
\end{proof}

\begin{remark} 
Lemma \ref{spheretorus} shows the existence of train tracks
in special form for strata of differentials on $S_{0,n}$ or 
$S_{1,n}$ where $3g-3+n\geq 4$, which is slightly better than
what appears in the statement of Theorem \ref{main}. 
\end{remark}

Call a component of a stratum of abelian or quadratic
differentials \emph{hyperelliptic} if
it consists of differentials on hyperelliptic
surfaces which are invariant under the
hyperelliptic involution.
Lemma \ref{spheretorus} is used to show

\begin{lemma}\label{hyperelliptic}
Let ${\cal Q}$ be 
a hyperelliptic component of a stratum of
quadratic or abelian differentials on a
surface of genus at least 3. Then there is a train track
$\tau$ in special form for ${\cal Q}$.
\end{lemma}
\begin{proof}
Let ${\cal Q}$ be a hyperelliptic component
of a stratum of quadratic differentials on 
a surface $S$ of genus $g\geq 3$ with $n\geq 0$ simple poles.
Such a hyperelliptic
component is obtained by pull-back of a  stratum $\hat{\cal Q}$ of 
quadratic differentials on the punctured sphere 
$S_{0,m}$ with 
a double branched covering map.

By the main result of \cite{L04} (see also 
Theorem 1.2 of \cite{L08}), 
the component $\hat{\cal Q}$ consists of differentials with 
$n\geq 8$ simple poles, and the cover is ramified at 
all or at all but one of the poles.

Let $\eta$ be a train track in special form 
for $\hat {\cal Q}$ 
as constructed in Lemma \ref{spheretorus},
with two clean vertex cycles $c_1,c_2$. 
The vertex cycles 
$c_1,c_2$ 
cut from the punctured sphere two twice punctured disks $P_1,P_2$.

Choose the branched covering in such a way that it is
ramified at each of the punctures in $P_1,P_2$.
The preimage of $\eta$ under this covering is 
an embedded graph $\hat \eta$ in the surface $S$. 
The preimage of the vertex cycle $c_i$
consists of two embedded 
simple closed curves which bound an
embedded annulus $A_i$. The annulus $A_i$ contains 
the preimage in $\hat \eta$ of the two
punctures in the 3-holed sphere component of  $S_{0,m}-c_i$
as shown in the
middle part of Figure F. It 
is subdivided by the preimage of $\eta$
into two bigons with an interior marked point.
The marked point is a preimage of one of the ramification points.
Remove these marked points and 
the branches 
in the interior of the annulus $A_i$ and identify the two 
boundary circles of $A_i$ as shown in Figure F so that
they form a single 
simple closed curve $v_i$ $(i=1,2)$.
\begin{figure}[ht]
\begin{center}
\includegraphics 
[width=0.7\textwidth] 
{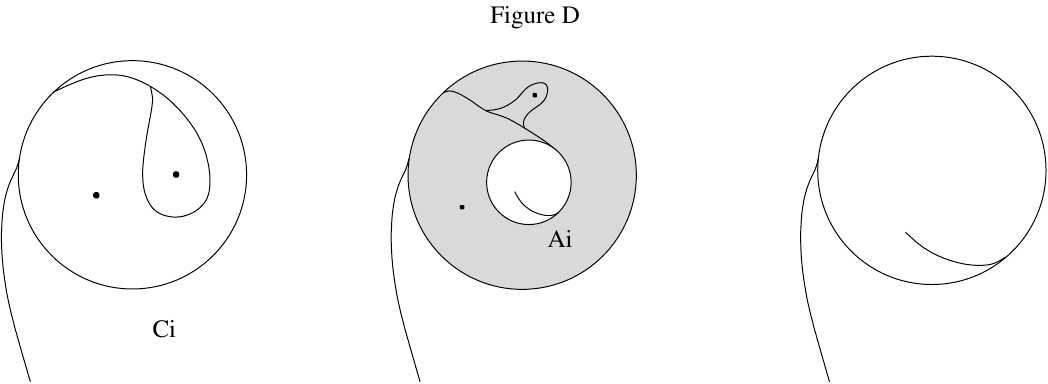}
\end{center}
\end{figure}

Collapse each remaining bigon in $\hat \eta$ 
containing a single preimage of a ramification point 
to a single arc as described above. The resulting graph $\tau$ 
is a large train track on $S$ 
which is contained in ${\cal L\cal T}({\cal Q})$. 
The train track
$\tau$ contains the curves $v_i$ 
as primitive vertex cycles. By construction, the
curves $v_i$ are non-separating and do not form a 
bounding pair and hence they define an elementary
curve system.
Moreover, since $c_i$ is a clean primitive vertex cycle for 
$\eta$, the primitive vertex cycle  
$v_i$ is clean for $\tau$. As the train track $\eta$ for $\hat {\cal Q}$
was in special form, naturality of the construction implies that 
$\tau$ is in special form for ${\cal Q}$.

The same reasoning also applies for hyperelliptic
components of abelian differentials. 
Namely, in this case the branched cover defining
the component is ramified at each of the simple
poles on the two-sphere and thus it is ramified at at least $8$ simple
poles. The above argument then shows that there
is a train track in special form for the 
component.
\end{proof}

To treat non-hyperelliptic components 
we construct from a large train track $\eta$ of
topological type $(m_1,\dots,m_\ell;-n)$
on a surface of genus $g\geq 0$ with $n$ punctures a 
train track $\tau$ of type $(m_1,\dots,m_\ell +4;-n)$ on a surface
of genus $g+1$ by \emph{attaching a handle} as follows.
 
The train track $\eta$ has a complementary  
polygon $P$ with $m_\ell +2$ sides. Attach two arcs
$b_1,b_2$ of class $C^1$ 
to the interior of two branches of $\eta$ which
are contained in two different sides of the polygon $P$ in such
a way that $b_1,b_2$ are disjoint and embedded in $P$.
Attach a simple closed curve $c_i\subset P$ of class $C^1$
to the arc $b_i$ which meets $b_i$ only at its free
endpoint and is tangent to $b_i$ $(i=1,2)$.
We require that the curves $c_1,c_2$ 
are disjoint and bound disjoint
embedded disks $D_1,D_2$ in the interior of $P$.

Remove the interiors of the disks $D_1,D_2$ from $P$. The 
boundary of the resulting surface consists of the curves
$c_1,c_2$. Glue $c_1$ to $c_2$ with a diffeomorphism
which reverses the boundary orientation of $D_i$. 
The result is a surface of genus $g+1$ with $n$ punctures which 
carries a train track $\tau$ of topological 
type $(m_1,\dots,m_\ell+4;-n)$. 
It contains the image of the curves $c_i$ under
a gluing map as a clean vertex cycle. 
Note that if $\eta$ is orientable, then for a suitable
choice of the arcs $b_1,b_2$, the train track $\tau$ is
orientable as well. In the sequel we always assume that 
the construction preserves orientability if applicable. 
We then call $\tau$ the
train track obtained from $\eta$ 
by \emph{attaching a handle}.

\begin{lemma}\label{handlebubble}
The train track $\tau$ obtained 
from $\eta$ by attaching a handle
is large. Moreover, it is orientable if and only if
this holds true for $\eta$.
\end{lemma}
\begin{proof} 
By construction, 
the train track $\tau$ is orientable
if and only if this holds true for $\eta$.
Moreover, $\eta$ can be viewed as a subtrack of $\tau$. 

Now $\eta$ is a large train track and hence 
it carries a minimal
large geodesic lamination of type $(m_1,\dots,m_\ell;-n)$. 
This geodesic lamination defines a minimal geodesic lamination
$\lambda_0$ carried by $\tau$. The train track $\tau$ contains a 
primitive vertex cycle $c_0$ which is disjoint from 
$\lambda_0$ and which is the image of the curves 
$c_1,c_2$ under the glueing process. The union 
$\lambda_0\cup c_0$ is a geodesic lamination carried by
$\tau$. This lamination is not large, but it is a 
sublamination of a large geodesic lamination which is 
the union of $\lambda_0\cup c_0$ with two isolated leaves
which pass through the two branches of $\tau$ connecting $c_0$ to 
the subtrack $\eta$ and which spiral from one side about $\lambda_0$, from
the other side about $c_0$. Thus $\tau$ carries a large geodesic
lamination. The same argument also shows that the
dual bigon track 
$\tau^*$ carries a
large geodesic lamination. In other words, $\tau$ is large.
\end{proof} 

For the construction of train tracks in special form for all components
of strata we use the classifiction of components 
due to Kontsevich and Zorich \cite{KZ03} (for abelian differentials)
and Lanneau \cite{L08} (for quadratic differentials).

\begin{proposition}\label{lanneau}
\begin{enumerate}
\item For every $g\geq 4$ the stratum
${\cal H}(2g-2)$ has three connected components.
One of these components is hyperelliptic, the
other two are distinguished by the parity of the
spin structure they define.
\item The stratum ${\cal H}(4)$ has two components.
One of the components is hyperelliptic, the other
consists of abelian differentials defining an odd
spin structure.
\item ${\cal H}(2)$ is connected.
\item For every $g\not=3,4$ and every $n\geq 0$ 
the stratum ${\cal Q}(4g-4+n;-n)$ is connected.
\item 
The strata ${\cal Q}(12;0)$ and ${\cal Q}(9;-1)$ 
have two connected components, and
${\cal Q}(4;0)=\emptyset$.
\end{enumerate}
\end{proposition}

We are now ready to show

\begin{proposition}\label{primitive}
Let ${\cal Q}$ be a 
non-hyperelliptic component of a stratum of 
quadratic or abelian differentials on $S$ 
where $3g-3+n\geq 5$. 
Then there is a train track $\tau$ 
in special form for ${\cal Q}$.
\end{proposition} 
\begin{proof} 
We divide the proof of the proposition into four steps.
The case $g=0$ and $g=1$ is covered by Lemma \ref{spheretorus}.

{\sl Step 1:} 
Strata of quadratic differentials with at least
three poles.

By the classification of Lanneau \cite{L08}, every stratum in moduli
space consisting of meromorphic quadratic differentials with
at least three poles
is connected. Thus for $n\geq 3$ and any $g\geq 2$, 
we can construct a large train track in special 
form for the stratum ${\cal Q}(4g-4+n;-n)$ by attaching 
handles to a train track on the
torus as described in 
Lemma \ref{handlebubble}. Although for $n=3$, this train track 
is not in special form, it contains 2 clean vertex cycles, and 
the train track obtained from it by attaching a handle is in special form.
These train tracks can be subdivided
to train tracks in special form for any stratum of 
quadratic differentials with at least three poles.

{\sl Step 2:} 
Strata of abelian differentials with a single zero.

The moduli space ${\cal H}(2)$ of abelian differentials with
a single zero on a surface of genus $2$ is connected.
It consists of
differentials which define an \emph{even spin structure}.
Figure G below shows a large train track $\eta\in 
{\cal L\cal T}({\cal H}(2))$. It contains a single 
clean vertex cycle. Attaching a handle to $\eta$ results in a train 
track in special form on a surface of genus $3$.

\begin{figure}[ht]
\begin{center}
\psfrag{Figure}{Figure I} 
\includegraphics 
[width=0.9\textwidth] 
{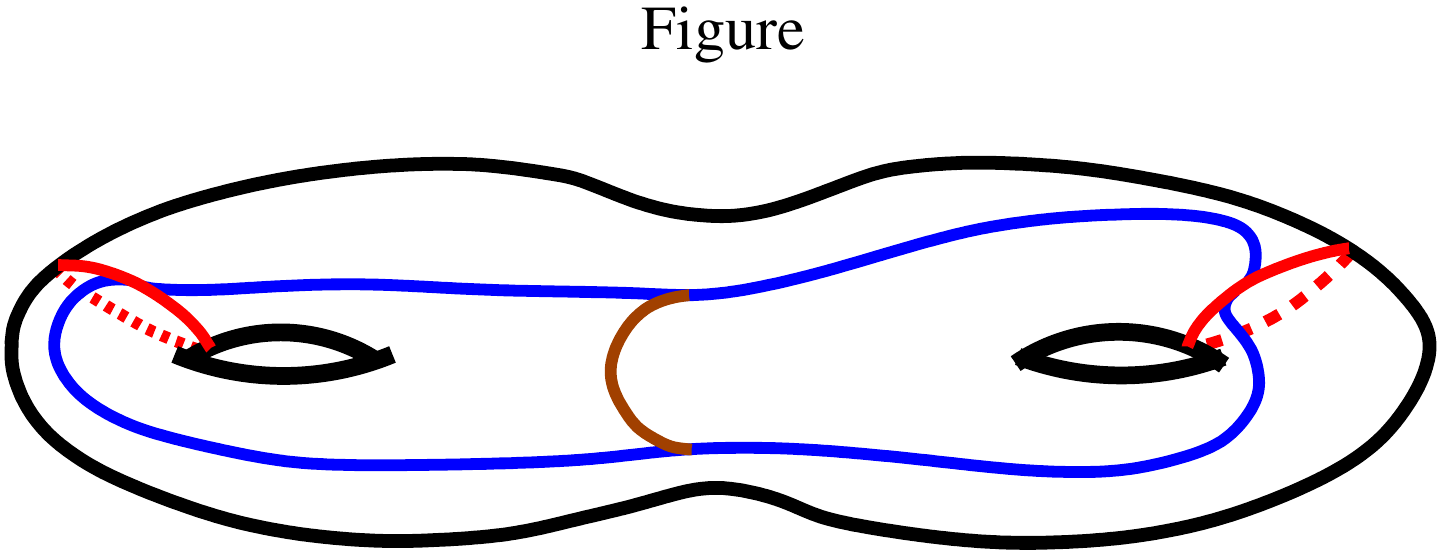}
\end{center}
\end{figure}

For $g=3$, the stratum ${\cal H}(4)$ consists of two components.
One of these components is hyperelliptic. The two components
are distinguished by the \emph{parity of the spin structure} 
they define \cite{KZ03}. The parity of the spin structure
for the hyperelliptic component is even.
By Lemma \ref{hyperelliptic},  
it suffices to show that attaching a handle to the train track 
$\eta$ in Figure G results in a train track in special form for 
the component with odd spin structure.
To this end we compute from a large train track 
$\tau\in {\cal L\cal T}({\cal Q})$ 
the parity of the spin structure
of the component ${\cal Q}$.

The parity of the spin structure defined by an abelian
differential $\omega$ 
can be calculated as follows
(see p.643 of \cite{KZ03}). 
For a smooth simple closed curve $\alpha$ on $S$ not 
passing through a zero of $\omega$ define 
$ind_\alpha\in \mathbb{Z}$ to be the total change of 
angle, divided by $2\pi$, between the tangent of $\alpha$ and 
the vector tangent to the vertical foliation of 
$\omega$, equipped with a fixed orientation.
Let $\{\alpha_i,\beta_i\mid i=1,\dots,g\}$ be
any system of $2g$ smooth simple closed curves
which define a symplectic basis for $H_1(S,\mathbb{Z})$
with the above property. 
Then 
\[{\rm Arf}(\omega)=\sum_{i=1}^g(ind_{\alpha_i}+1)
(ind_{\beta_i}+1)(\text{mod } 2).\]

This formula enables us to calculate the parity
using a train track. Namely, a large orientable
train track $\tau$ of type $(4g-4;0)$ 
has a single complementary component $C$ which is
a $4g$-gon. Let $\alpha$ be a smooth simple closed
curve on $S$ which intersects $\tau$ transversely in 
finitely many points contained in the interior
of some branches of $\tau$. Define the \emph{index} 
$r_\tau(\alpha)\in \mathbb{Z}/2\mathbb{Z}$ 
of $\alpha$ as follows.

Choose a numbering of the sides of the complementary
region of $\tau$ in counter-clockwise order.
Choose also an orientation of $\alpha$.
A transverse intersection point
$p\in \alpha\cap \tau$ is contained in precisely
two sides $s_1,s_2$ of $C$. 
Write $r(p)=s_2-s_1+1\,(\text{mod }2) =s_1-s_2+1\,(\text{mod }2)$ 
and define 
\[r_\tau(\alpha)=\sum_pr(p)\in \mathbb{Z}/2\mathbb{Z}.\]
Note that if $\alpha^\prime$ is isotopic to $\alpha$
with an isotopy which moves some subarc of 
$\alpha$ across a switch then this number is 
unchanged, and the same holds true if $\alpha$ is fixed and 
$\tau$ is 
modified by a split.

Choose smooth simple closed curves 
$\{\alpha_i,\beta_i\mid i=1,\dots,g\}$ which define a symplectic
basis of $H_1(S,\mathbb{Z})$.
Assume that each of the
curves $\alpha_i$ intersects $\tau$ in finitely many
points which are contained in the interior of some branch of 
$\tau$. Define
\[\phi(\tau)=\sum_{i=1}^g(r_\tau(\alpha_i)+1)
(r_\tau(\beta_i)+1)\in \mathbb{Z}/2\mathbb{Z}\]
and call this number the \emph{parity of the spin structure}
of $\tau$.

Recall that $\eta$ is a large train track with a single
complementary component $C$. 
If the train track $\tau$ is obtained 
from $\eta$ by attaching a handle, 
then the parity of the spin structure of 
$\tau$ can be calculated from the parity
of the spin structure of $\eta$ as follows.
There is a primitive vertex cycle $\alpha_1$ for $\tau$ 
which is disjoint from $\eta$ (it
goes around the handle). This vertex cycle $\alpha_1$ 
satisfies $r_\tau (\alpha_1)=0$ since up to
homotopy, it has a unique intersection point with
$\tau$ which is contained 
in one of the small branches adjacent to the primitive
vertex cycle $\alpha_1$. Then this
intersection point is contained  
in two consecutive
sides of the complementary component of $\tau$.

There is a second curve $\beta_1$ 
in the handle which
intersects $\alpha_1$ in a single point, and
it intersects $\tau$ in a single point $q$ as well.
Let $s_i,s_j$ $(i<j)$ be the  
sides of the complementary component
$C$ of $\eta$ at which 
branches of $\tau-\eta$ are attached. If we choose
$s_j=s_i+1$ then Figure H shows that 
$r_\tau(\beta_1)=0$.
\begin{figure}
\begin{center}
\psfrag{Figure J}{Figure J} 
\includegraphics 
[width=0.5\textwidth] 
{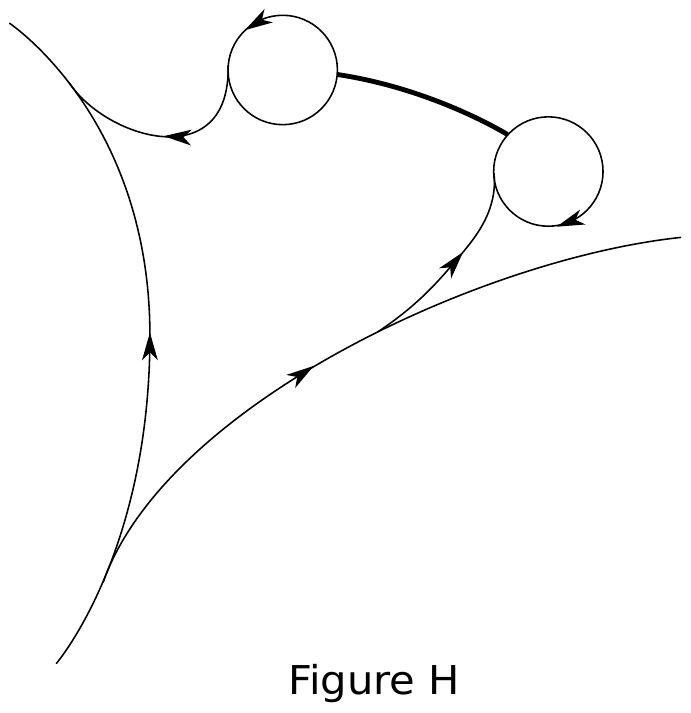}
\end{center}
\end{figure}

The curves on $S$ used to calculate the parity of the spin
structure for $\eta$ can be chosen to be disjoint
from $\tau-\eta$ viewed as a subgraph of the complementary
component $C$. Then 
the indices of the curves used for $\eta$ 
do not change mod 2 and hence 
the parity of the spin structure of 
$\tau$ is opposite to
the parity of the spin structure for $\eta$. 
In particular, attaching a handle to the train track shown
in Figure G results in a train track in special form for the
component of ${\cal H}(4)$ with odd spin structure.
Together with the construction for hyperelliptic 
components, we obtain a train track in special form
for each of the two components of ${\cal H}(4)$.

Using again Proposition \ref{lanneau},
for $k\geq 3$ the two different non-hyperelliptic 
components of ${\cal H}(2k)$ are distinguished by the
parity of the spin structure they define.
It follows from the above discussion that attaching a handle
to a train track in special form for 
a component of ${\cal H}(2k-2)$ with even (or odd) spin structure
is a train track in special form for a component of 
${\cal H}(2k)$ with odd (or even) spin structure. 
Now it is easy to see that a train track for a hyperelliptic
component can only arise by this construction from
a train track for a hyperelliptic component.
Since the parity of the spin structure of a hyperelliptic
component is even, none of the two train tracks arising
from attaching a handle to one of the train tracks
in special form for a component of ${\cal H}(4)$ is a train 
track for a hyperelliptic component. 
Thus by induction beginning with ${\cal H}(4)$, we obtain in this
way for each $k\geq 1$ and for each non-hyperelliptic 
component ${\cal Q}$ of 
${\cal H}(2k)$ 
a train track in special form for ${\cal Q}$.

{\sl Step 3:} Strata of quadratic differentials
with a single zero and at most two poles.

By the classification of Lanneau \cite{L08}, 
strata of quadratic 
differentials with a single zero and at most
two poles are connected.

To obtain a train track in special form for this stratum 
on a surface of genus
$g=2$ with 
$m=2$ punctures, 
attach to the train track shown in Figure G a circle as shown in 
Figure $D$ enclosing
two once-puncture monogons. Similarly, to obtain
a train track in special form for a stratum with a single
zero and a single pole on a surface of genus $3$, attach
a train track in special form on a surface of genus $3$ a
once punctured monogon.

A train track in special form for a stratum in higher genus
can be obtained by attaching handles
to the train track for genus $2$ or $3$.

{\sl Step 4:} Subdividing complementary components.

Following \cite{L08}, we say that a component 
${\cal Q}$ of a stratum in ${\cal Q}(S)$ for a surface $S$ 
of genus $g\geq 2$  
is \emph{adjacent} to a component ${\cal Q}_0$ of another stratum 
if ${\cal Q}_0$ is contained in 
the closure $\overline{\cal Q}$ of ${\cal Q}$ in ${\cal Q}(S)$.
Here we allow that poles merge with zeros and disappear. 

Lanneau \cite{L08} showed that with the exception of 
one sporadic component in each of the strata 
${\cal Q}(9;-1)$, ${\cal Q}(3,6;-1)$,
${\cal Q}(3,3,3;-1)$, any non-hyperelliptic component of  
a stratum with at least two distinct types of zeros or poles
is adjacent to ${\cal Q}(4g-4)$. For such a component,
train tracks in special form can be obtained from
train tracks in special form for 
components of strata with a single zero by
subdivision of complementary components.

For the completion of the proof 
of the proposition we are left with the investigation of 
the sporadic components in genus $g=3,4$ as 
listed in the classification of Lanneau \cite{L08}.

The sporadic component for $g=4$ is a component of 
${\cal Q}(12)$ which can be checked
explicitly. The sporadic component
of ${\cal Q}(3,3,3;-1)$ is adjacent to the
sporadic component of ${\cal Q}(3,6;-1)$, and the
sporadic component of ${\cal Q}(3,6;-1)$
is adjacent to the sporadic component of ${\cal Q}(9;-1)$ \cite{L08}.
Using Step 2 above, it is therefore enough to construct
a train track with the required properties which belongs to 
the sporadic component of ${\cal Q}(9;-1)$. However, the sporadic
component of ${\cal Q}(9;-1)$ admits a quadratic differential with a 
two-cylinder-decomposition which can be used to construct 
a train track as required
(compare the table in \cite{L08}).
This completes the proof of the proposition.
\end{proof}

\begin{remark} Although we use the classification result
of Kontsevich-Zorich and of Lanneau in our construction,
the construction can be used to give a alternative proof for
the classification.
\end{remark}

\begin{remark} \label{parity}
The computation of the parity of the spin structure in the proof of
Proposition \ref{primitive} is similar to the computation of parities
in the transition of a component of a stratum to a component of its
principal boundary in \cite{EMZ03}. Unfortunately we can not 
use these computations for our needs.
\end{remark}

\section{Coding of pseudo-Anosov mapping classes}\label{sec:coding }

In this section we combine the results from Sections 1-4 and from 
\cite{H13} and \cite{H22} to set up the main properties 
needed for the proof of Theorem \ref{main}.
We continue to use the notations from sections 1-4.

The number $m>0$ of branches of 
a large train track 
$\tau\in {\cal L\cal T}(m_1,\dots,m_\ell;-n,p)$
only depends on the topological type of $\tau$.
A numbering of the
branches of $\tau$ defines
an embedding of the cone ${\cal V}(\tau)$ of transverse measures on $\tau$
onto a closed convex cone in $\mathbb{R}^m$ determined
by the switch conditions. 
For the standard basis $e_1,\dots,e_m$
of $\mathbb{R}^m$, this embedding associates
to a measure $\mu\in {\cal V}(\tau)$ 
the vector $\sum_i\mu(i)e_i\in \mathbb{R}^m$ where we
identify a branch of $\tau$ with its number.
If $\sigma\prec\tau$ then the transformation
${\cal V}(\sigma)\to {\cal V}(\tau)$ induced by
a carrying map $\sigma\to \tau$ is linear in these coordinates.

The mapping class group ${\rm Mod}(S)$ acts on 
marked train tracks by precomposition of marking. 
A \emph{train track expansion} of a mapping class
$\phi\in {\rm Mod}(S)$ is a train track $\tau$ such that 
$\phi(\tau)\prec\tau$. Then 
the composition of the isomorphism
${\cal V}(\tau)\to {\cal V}(\phi\tau)=\phi({\cal V}(\tau))$
with a carrying map ${\cal V}(\phi\tau)\to {\cal V}(\tau)$
is given by a linear map 
\[A(\phi,\tau):\mathbb{R}^m\to
\mathbb{R}^m.\]
The matrix describing this map with respect to the standard
basis of $\mathbb{R}^m$ has non-negative entries.

By the Perron Frobenius theorem,
an $(m,m)$-matrix $A$ with non-negative entries
admits an eigenvector with non-negative
entries. The corresponding eigenvalue $\alpha$
is positive.
If some power of $A$ is positive, then the 
generalized eigenspace for
$\alpha$ is one-dimensional, and $\alpha$ is
bigger than the absolute value of 
any other eigenvalue of $A$. Furthermore, an eigenvector
of $A$ for this eigenvalue $\alpha$ 
has positive entries and will be called 
a \emph{Perron Frobenius
eigenvector}.

We next collect some information on 
the relation between a pseudo-Anosov mapping class
$\phi$, a train track expansion $\tau$ for $\phi$ 
 and the linear map $A(\phi,\tau)$. 
We begin with Corollary 3.2 of \cite{P88}.
For its formulation, recall that a pseudo-Anosov mapping class
$\phi$ admits an invariant flow line
for the Teichm\"uller flow on the Teichm\"uller 
space of area one abelian or quadratic differentials, which is the 
unit cotangent line of an \emph{axis} of $\phi$.
The axis is prescribed by a pair of projective
measured geodesic laminations 
$([\lambda^h],[\lambda^v])\in {\cal P\cal M\cal L}^2$ which are
the attracting and 
repelling fixed points for the action of $\phi$ on ${\cal P\cal M\cal L}$, 
respectively. 
Recall from Lemma \ref{polyhedron} that 
the cone ${\cal V}(\tau)$ of all non-negative solutions 
to the switch condition on $\tau$ can naturally
be identified with the 
subset of ${\cal M\cal L}$ of measured geodesic laminations
which are carried by $\tau$.

\begin{lemma}\label{perron}
Let $\tau\in {\cal L\cal T}(m_1,\dots,m_\ell;-n,p)$ and 
let $\phi\in {\rm Mod}(S)$ 
be such that $\phi(\tau)\prec\tau$ and that the matrix
$A(\phi,\tau)$ is positive. Then $\phi$ is pseudo-Anosov.
The unit cotangent line of its 
axis
intersects
${\cal Q}(\tau)$. The horizontal measured geodesic lamination 
of a differential in this cotangent line is
a Perron Frobenius eigenvector of the linear map 
$A(\phi,\tau)$.
\end{lemma}
\begin{proof}
It follows from Corollary 3.2 of \cite{P88} that
$\phi$ is pseudo-Anosov and that the
attracting fixed point for
its action on ${\cal P\cal M\cal L}$ is the projectivization
of a Perron-Frobenius eigenvector $\lambda$ of 
the matrix $A(\phi,\tau)$. This eigenvector is positive
and unique up to scale. 

As $\phi(\tau)\prec \tau$, we have 
$\phi^{-1}(\tau^*)\prec \tau^*$ and hence 
${\cal V}^*(\phi^{-1}(\tau))\subset {\cal V}^*(\tau)$. 
Since $\phi$ acts with north-south dynamics on 
${\cal P\cal M\cal L}$, we conclude that 
the attracting fixed point of $\phi^{-1}$ is a measured
geodesic lamination whose support is carried by
$\tau^*$. This then implies that 
the unit cotangent line of the axis of $\phi$
intersects ${\cal Q}(\tau)$.
%
%
%
%
%
\end{proof}


The following 
lemma gives a more geometric approach to
the study of periodic orbits in a component 
${\cal Q}$ 
of a stratum. 
%
In its formulation, we assume that the surface 
$S$ is equipped with a complete finite volume hyperbolic metric. 
The notion of an $a$-long train track which $\epsilon$-follows a
geodesic lamination was introduced at the end of
Section \ref{strata}. We denote as before by
$\tilde {\cal Q}$ the preimage of 
${\cal Q}$ 
in the Teichm\"uller space of marked differentials.

\begin{lemma}\label{attracting}
If $q\in \tilde {\cal Q}$ is a point on the 
cotangent line of an axis of the pseudo-Anosov mapping class $\phi$ 
and if $q\in {\cal Q}(\tau)$ for some 
$\tau\in{\cal L\cal T}(\tilde {\cal Q})$, then 
for sufficiently large $k>0$, the train track
$\phi^k (\tau)$ is $a$-long and $\epsilon$-follows 
the support of the attracting fixed point 
$[\mu]$ for the action of $\phi$ on ${\cal P\cal M\cal L}$.
\end{lemma}
\begin{proof}
Let $\tau\in {\cal L\cal T}(\tilde {\cal Q})$
and assume that $q\in {\cal Q}(\tau)$ for a differential 
$q\in \tilde {\cal Q}$ on the contangent line of
the axis of the pseudo-Anosov mapping class $\phi$.
Denote by $\hat \mu={\rm supp}([\mu]),
\hat \nu={\rm supp}([\nu])$ the supports of the attracting
and repelling 
fixed points for the action of $\phi$ on ${\cal P\cal M\cal L}$,
respectively. 
Then $\hat \mu$ is carried by $\tau$, furthermore $\hat \mu$ is minimal
and its topological type coincides with the topological type of $\tau$.
In particular, the carrying map $\hat \mu\to \tau$ induces a bijection 
between the complementary polygons of $\hat \mu$ and the complementary
components of $\tau$. 

Represent $\hat \mu,\hat \nu$ by geodesics for the hyperbolic metric $h^\prime$
in the conformal class of $q$. Then up to homotopy, we may
assume that the switches of $\tau$ are contained in 
the intersection $S-(\hat \mu\cup \hat \nu)$ of the complementary
components of $\hat \mu,\hat \nu$. Furthermore, the weight 
deposited by $\hat \nu$ on any branch of $\tau$ is positive.

Following the proof of Lemma 6.2 of \cite{CB88},
define an equivalence relation $\sim$ on $S$ by
$x\sim y$ if either
\begin{enumerate}
\item[(i)] $x,y$ are in the closure of the same component of
$S-(\hat \mu\cup \hat \nu)$ or
\item[(ii)] $x,y$ are in the closure of the same component of
  $\hat \mu-\hat \nu$ or 
\item[(iii)] $x,y$ are in the closure of the same component of
  $\hat \nu-\hat \mu$ or
\item[(iv)] $x=y$.
\end{enumerate}
Lemma 6.2 of \cite{CB88} shows that $S/\sim$ is a
finite type suface homeomorphic to $S$, and the 
measured laminations
$\mu,\nu$ project to transverse singular measured foliations
${\cal F}^s,{\cal F}^u$ on $S$ which up to isotopy equal the
vertical and the horizontal measured foliations
of the differential $q$.

It follows from the construction that the switches of $\tau$ project
to singular points of $q$, and any singular point is the image of 
precisely one switch.
Let $b$ be a branch of $\tau$. Then $b$ defines the 
homotopy class with fixed endpoints 
of a geodesic arc $\beta$ for the singular euclidean 
metric $h$ defined by $q$ which connects $x$ to $y$. 
The arc $\beta$ is a concatenation of saddle connections. Since there are
no vertical saddle connection for $q$, 
the geodesic arc $\beta$ is transverse to the
vertical foliation of $q$. Doing this construction with each of 
the branches of $\tau$ yields 
an embedded graph ${\cal G}\subset S$ with the following 
properties.
\begin{enumerate}
\item ${\cal G}$ is homotopic to $\tau$. 
\item The vertices of ${\cal G}$ are precisely the singular points of $q$.
\item The edges of ${\cal G}$ are
geodesic segments for $h$ whose directions are
uniformly bounded away from the vertical direction. 
\end{enumerate}
Note that ${\cal G}$ may not admit
a tangent at its vertices, and in general, ${\cal G}$, is
not a train track.

Now $\phi$ can be represented by a
homeomorphism of $S$
which is smooth outside the singular set of $q$ and 
which expands the 
horizontal foliation of $q$ and contracts the vertical one. Thus 
the image of ${\cal G}$ under $\phi$ is a piecewise geodesic graph
in the singular euclidean surface $(S,h)$ whose edges have directions which are
closer to the horizontal direction than 
the directions of the edges of ${\cal G}$ with respect to the
conformal structure on $S$ defined by $q$ or $h$.
Iteration then yields that as $k\to \infty$, 
up to passing to a subsequence the graphs
$\phi^k({\cal G})$ converge in the Hausdorff topology 
to a closed 
subset of the horizontal foliation for the flat metric $h$.
 Since the edges of $\phi^k{\cal G}$ are geodesics for
$h$ whose lengths tend to infinity with $k$, this set is 
leaf saturated, that is, it is a union of leaves. 
By minimality of the horizontal foliation of $q$, this limit equals
the horizontal foliation of $q$.

The horizontal geodesic lamination $\hat \mu$ of $q$ is obtained
from the horizontal foliation by cutting $S$ open along the
separatrices of the foliation and straightening the complement with
respect to the hyperbolic metric $h^\prime$ on $S$. 
This implies that up to isotopy,  
the hyperbolic straightening of 
$\phi^k({\cal G})$, obtained by replacing
each edge by a geodesic segment for the 
hyperbolic metric $h^\prime$,  
converges as $k\to \infty$ in the Hausdorff topology 
to $\hat \mu$. This shows the lemma.
\end{proof}

The proof of Lemma \ref{attracting} relied on the fact that
given a pseudo-Anosov mapping class $\phi$ and
an abelian or quadratic differential $q$ on the cotangent line
of an axis of $\phi$, the mapping class
$\phi$ can be represented by a homeomorphism
of $S$ which is smooth outside
the singular points of $q$ and preserves the vertical 
and the horizontal measured foliation of the singular euclidean metric
defined by $q$. However, $\phi$ may permute zeros
of $q$ of the same order, it
may permute poles, and if $\phi$ fixes a zero, then is may permute
vertical and horizontal separatrices coming out of the zero and only 
preserve their cyclic order.

We say that $\phi$ \emph{preserves the combinatorics} if
it fixes each singular point of $q$ and preserves each horizontal and
vertical separatrix coming out of a zero. 
This property does not depend on the choice of the differential 
$q$ on the cotangent line of an axis of $\phi$.
As the
number of singular points of $q$, counted with multiplicity, is a
topological invariant, each pseudo-Anosov mapping class
$\phi$ has a positive power which preserves the combinatorics, and the 
degree of this power is bounded from above
by a constant only depending on the Euler characteristic
of $S$. 
Note that similar constraints arise in the 
context of zippered rectangles 
which are used by Veech \cite{V86} to study the 
Teichm\"uller flow.

For a pseudo-Anosov mapping class
$\phi\in {\rm Mod}(S)$
and a differential $q$ on the cotangent line of its axis, 
contained in the preimage $\tilde {\cal Q}$ 
in $\tilde {\cal Q}(S)$ of a component 
${\cal Q}$ of a stratum of quadratic or abelian differentials,
define 
\[ {\cal L\cal T}(\phi)=\{\tau\in
  {\cal L\cal T}(\tilde {\cal Q})\mid q\in {\cal Q}(\tau)\}.\]
Note that this only depends on $\phi$ but not on the choice
of $q$.

The following observation is the main technical result of this section. 
For the purpose of its proof and later use,
define a \emph{splitting and shifting sequence} of a
train track $\tau$ to be a sequence of modifications of $\tau$ 
by splitting or shifting moves. If $\eta$ is carried
by $\tau$ then $\tau$ can be connected to $\eta$ by a 
splitting and shifting sequence \cite{PH92}.

\begin{lemma}\label{traintrackexpansion}
  If $\phi$ is a pseudo-Anosov mapping class which preserves the
  combinatorics, with cotangent line in $\tilde {\cal Q}$,
  then 
\[{\cal L\cal T}(\phi)=\{\tau\in {\cal L\cal T}(\tilde {\cal Q})\mid
\phi(\tau)\prec \tau\}.\] 
\end{lemma}
\begin{proof} It is shown in 
\cite{P88} that every pseudo-Anosov mapping class $\phi$ has a train track 
expansion $\tau$. Furthermore, we may assume that 
$\tau\in {\cal L\cal T}(\tilde {\cal Q})$ where $\tilde {\cal Q}$ is 
a component of a stratum in the Teichm\"uller space of 
abelian or quadratic differentials which contains the cotangent
line of the axis of $\phi$.

We show first that such a train track $\tau$ is contained in 
${\cal L\cal T}(\phi)$. To see this note that by 
induction, we have $\phi^k(\tau)\prec \tau$ and hence 
$\phi^k {\cal V}(\tau)\subset {\cal V}(\tau)$
for all $k\geq 1$. Since 
${\cal V}(\tau)$ is a cone over a compact convex polyhedron and hence
its projectivization $P{\cal V}(\tau)$ 
is homeomorphic to a closed topological ball, and since
the action of $\phi$ on ${\cal V}(\tau)$ commutes with rescaling
and hence descends to an action on $P{\cal V}(\tau)$, 
the
Brouwer fixed point theorem yields that $\phi$ has a fixed point in  
$P{\cal V}(\tau)$. 
But $\phi$ acts on 
${\cal P\cal M\cal L}\supset P{\cal V}(\tau)$ with north-south dynamics and therefore this
fixed point has to be the attracting fixed point of $\phi$. 
As a consequence, the support $\mu$ of the 
horizontal measured geodesic lamination of 
a point $q\in \tilde {\cal Q}$ on the
cotangent line of the axis of $\phi$ is carried by 
$\tau$.

Reversing the role of $\tau$ and its dual bigon track  and 
replacing $\phi$ to $\phi^{-1}$, this argument also shows that
the vertical measured geodesic lamination 
of $q$ is carried by the dual 
bigon track $\tau^*$
of $\tau$ and hence $\tau\in {\cal L\cal T}(\phi)$ as claimed.

The inclusion of the set of train track expansions of $\phi$ into 
${\cal L\cal T}(\phi)$ holds true for all pseudo-Anosov mapping classes.
To show the reverse inclusion we assume from now on that
the pseudo-Anosov mapping class $\phi$ preserves the combinatorics.
Our goal is 
to show that $\phi(\tau)\prec \tau$ for every
$\tau\in {\cal L\cal T}(\phi)$.

To this end let as before $\mu$ be the support of the horizontal
measured geodesic lamination of the differential $q$.  
Then $\mu$ is minimal and filling, and its combinatorial type
coincides with the combinatorial type of any 
$\tau\in {\cal L\cal T}(\phi)$. If
$\eta$ is obtained from 
$\tau\in {\cal L\cal T}(\phi)$
by a \emph{$\mu$-split}, that is, by a split
with the 
 property that $\eta$ carries $\mu$, then
 $\eta\in {\cal L\cal T}(\phi)$. 

 The infinite cyclic subgroup $\Gamma$ of ${\rm Mod}(S)$ generated by 
$\phi$ acts on 
the set ${\cal L\cal T}(\phi)\subset {\cal L\cal T}(\tilde {\cal Q})$
as a group of permutations. 
 Let ${\cal E}\subset {\cal L\cal T}(\phi)$ be the subset of all train track
 expansions of $\phi$. We noted above that 
 this set is non-empty. Moreover, it is clearly 
 $\Gamma$-invariant and 
 invariant under the shift operation on train tracks.
 We claim that if $\tau\in {\cal E}$ and if 
 $\eta$ is obtained from $\tau$ by a $\mu$-split, 
 then $\eta\in {\cal E}$. 

Thus let $\tau\in {\cal E}$ and let $e$ be a large branch of $\tau$.
We know that 
$\phi(\tau)\prec\tau$. As a consequence,
if $\phi(\tau)$ is carried by the $\mu$-split 
$\sigma$ of $\tau$ at $e$,
then since the $\mu$-split $\phi(\sigma)$
of $\phi(\tau)$ at the branch
$\phi(e)$ is 
carried by $\phi(\tau)$,
we have $\phi(\sigma)\prec \sigma$ and we are done. So assume that 
$\phi(\tau)$ is not carried by the $\mu$-split
of $\tau$ at $e$. 

Let 
$A$ be a \emph{foliated neighborhood} of $\tau$ in $S$; this is 
a neighborhood of $\tau$ which is foliated by compact arcs, called 
\emph{ties}, which are 
transverse to $\tau$ and intersect
$\tau$ in precisely one point. There is a collapsing map  
$F:A\to \tau$ which collapses each of these ties to a point. 
Since $\phi(\tau)\prec \tau$, the train track $\phi(\tau)$ can be isotoped 
to be embedded in $A$ and transverse to the ties. 
The restriction of the collapsing map $F:A\to \tau$ to $\phi(\tau)$ is a 
carrying map. We may assume that no switch of $\phi(\tau)$ is mapped
by $F$ to a switch of $\tau$. 

Let $v$ be an endpoint of the large branch $e$. A \emph{cutting arc} 
for $\phi(\tau)$ and $v$ is an embedded arc 
$\gamma:[0,d]\to F^{-1}(e)$ beginning at $\gamma(0)=v$
which is transverse to the ties of $A$, which is disjoint from $\phi(\tau)$
except possibly at its endpoint and such that the length of 
$F(\gamma)\subset e$ is maximal among arcs with these properties. 
The maximality condition implies that either $F(\gamma[0,d])=e$, that is,
$\gamma$ crosses through the foliated rectangle $F^{-1}(e)$, or 
that $\gamma(d)$ is a switch of $\phi(\tau)$ contained in the interior of 
$F^{-1}(e)$, and the component of $S-\phi(\tau)$ which contains 
$v$ has a cusp at $\gamma(d)$. In particular, since 
$\phi$ preserves the combinatorics by assumption, if 
$\gamma(d)$ is a switch contained in the interior of 
$F^{-1}(e)$, then we have $\gamma(d)=\phi(v)$ and 
$\gamma(d)$ is an endpoint of the large branch $\phi(e)$. 

By Lemma A.2 of \cite{H09}, if $\phi(\tau)$ is not carried by a split of 
$\tau$ at $e$, then the following holds true. Let $v,v^\prime$ be the 
endpoints of $e$ and let $\gamma:[0,d]\to F^{-1}(e)$,
$\gamma^\prime:[0,d^\prime]\to F^{-1}(e)$ be the cutting arcs for 
$\phi(\tau)$ and $v,v^\prime$. Then $\gamma(d),\gamma^\prime(d^\prime)$
are switches of $\phi(\tau)$ contained in the interior of 
$F^{-1}(e)$, and there is a trainpath
$\rho:[0,m]\to \phi(\tau)\cap F^{-1}(e)$ connecting 
$\rho(0)=\gamma(d)$ to $\rho(1)=\gamma^\prime(d^\prime)$.

Since $\phi$ preserves the combinatorics, it follows from the above 
discussion that the trainpath $\rho[0,m]$ consists of the single 
large branch $\phi(e)$. Then
a collision $\beta$ of $\tau$ at $e$, that is, a split followed
by removal of the diagonal, carries $\phi(\beta)$. 
As a consequence, this collision is a train track expansion of
$\phi$. 

On the other hand, 
the combinatorial type of $\beta$ does not coincide with the
combinatorial type of $\tau$. More precisely, $\beta$ does not carry
any geodesic lamination whose topological type equals the 
topological type of the support of the horizontal measured geodesic
lamination of $q$. Together this is 
a contradiction to the beginning of this proof.

To summarize, we showed so far that if $\tau\in {\cal E}$ then so is any 
$\mu$-split of $\tau$. 
To complete the proof of the lemma, it now suffices to 
show the following. Let $\eta\in {\cal L\cal T}(\phi)$ 
be arbitrary; then there exists $\tau\in {\cal E}$ such that
$\eta\prec \tau$. Namely, by \cite{PH92}, in this case 
$\tau$ can be connected to $\eta$ by a splitting and shifting sequence,
and since $\eta$ carries $\mu$, 
by induction and what we have established so far, 
any train track in the sequence is contained in ${\cal E}$.

For $\tau\in {\cal E}$ we have 
$\phi^{-k+1}(\tau)\prec \phi^{-k}(\tau)$ 
and hence
$\phi^{-k}(\tau)\in {\cal E}$ for all $k$. By invariance under the 
action of $\Gamma$, it thus suffices to show that there is some 
$k>0$ such that $\phi^k\eta\prec \tau$. By Lemma 3.2 of 
\cite{H09}, this follows if for a complete 
finite volume hyperbolic metric on $S$, a given number $\epsilon >0$
and all sufficiently large 
$k$, the train track $\phi^k\eta$ is $a$-long and $\epsilon$-follows 
$\mu$ 
where $a>0$ only depends on the hyperbolic metric.
That this is indeed the case was established in Lemma \ref{attracting}. 
\end{proof}

For a component
$\tilde{\cal Q}$ of a stratum 
$\tilde{\cal Q}(m_1,\dots,m_\ell;-n,p)$ we write as before 
$\tau\in {\cal L\cal T}(\tilde {\cal Q})$ if 
$\tau\in {\cal L\cal T}(m_1,\dots,m_\ell;-n,p)$
and if moreover the set ${\cal Q}(\tau)\subset \tilde {\cal Q}(S)$ 
is contained in the
closure of $\tilde{\cal Q}$. 
Denote by ${\rm Stab}(\tilde {\cal Q})$
the stabilizer 
of $\tilde {\cal Q}$ in ${\rm Mod}(S)$.


\begin{lemma}\label{carryingcontrol}
  For every component ${\cal Q}$
  of a stratum there exists a number 
$\kappa=\kappa({\cal Q})>0$ with the following property. 
If $\tau,\sigma\in {\cal L\cal T}(\tilde {\cal Q})$ and
if $\sigma\prec \tau$ then there exists 
some $\phi\in {\rm Stab}(\tilde {\cal Q})$ such that 
$\phi(\tau)\prec \sigma$ and that $\sigma$ can 
be connected to $\phi(\tau)$ by a shifting and splitting 
sequence of length at most $\kappa$.
\end{lemma}
\begin{proof}
  Let $\tau,\sigma
  \in {\cal L\cal T}(\tilde {\cal Q})$ with
  $\sigma\prec\tau$. Then ${\cal Q}(\tau)\cap {\cal Q}(\sigma)$ 
  contains an open subset of $\tilde {\cal Q}$.
Since cotangent lines of axes of pseudo-Anosov
mapping classes are dense in $\tilde {\cal Q}$, there exists a pseudo-Anosov
mapping class $\phi\in {\rm Stab}(\tilde {\cal Q})$
and a differential $q\in {\cal Q}(\tau)\cap {\cal Q}(\sigma)$
on the contangent line of an axis of $\phi$.
Lemma \ref{attracting} and Lemma 3.2 of \cite{H09} together show
that there exists some $k\geq 1$ such that
$\phi^k(\sigma)\prec\tau$. Since the number of
${\rm Mod}(S)$-orbits of train tracks in
${\cal L\cal T}(\tilde {\cal Q})$ is finite, by invariance this suffices for the
proof of the lemma.
\end{proof}

Let ${\cal T}(S)$ be the Teichm\"uller space of $S$. 
The \emph{translation length} of a pseudo-Anosov element
$\phi\in {\rm Mod}(S)$ is the minimal 
Teichm\"uller distance
between a point $x\in {\cal T}(S)$ and its image under $\phi$. 
The translation length of $\phi$ only depends on the
conjugacy class of $\phi$, and it equals 
the length of the periodic orbit of the Teichm\"uller flow 
which is the projection of the cotangent line of an axis of $\phi$
to the moduli space of quadratic or abelian differentials. 

Our goal is to construct periodic orbits in the thin part of 
a stratum of abelian or quadratic differentials by deforming 
periodic orbits in components of the principal 
boundary of the stratum.
To implement this idea we evoke 
counting results in components of the principal boundary. 
We begin with the setup suitable for our goal. 

The moduli space of abelian or quadratic differentials is the
quotient of the Teichm\"uller space of abelian or quadratic 
differentials by the action of the mapping class group
${\rm Mod}(S)$ of $S$.
To avoid some technical difficulties, we replace ${\rm Mod}(S)$
by a torsion free finite index subgroup and replace 
each component ${\cal Q}$ 
of a stratum by the corresponding finite
cover in the sense of orbifolds. Such a component is a smooth manifold, 
and every 
periodic orbit for the Teichm\"uller flow in such a cover, again 
denoted by ${\cal Q}$, 
defines a nontrivial free homotopy class in ${\cal Q}$. 
No two distinct such orbits defined the same free homotopy class.

There is a further finite cover of a component ${\cal Q}$ defined as follows. 
Number all zeros or poles for a differential $q\in {\cal Q}$ 
which occur with 
multiplicity greater than one in an arbitrary way. Number furthermore the
horizontal separatrices of a differential $q$ at each singular point in 
a counterclockwise order. 
%
This construction defined a fiber bundle over
${\cal Q}$ with finite fiber.
A point in the fiber over $q$ determines a numbering of the zeros and poles
of $q$ and a horizontal separatrix coming out of each zero. 
If we denote by $\hat {\cal Q}$ a connected component of this fiber bundle, 
then $\hat {\cal Q}\to {\cal Q}$ is a finite connected covering. 
The Teichm\"uller flow $\Phi^t$ 
naturally lifts to a flow on this covering, 
and the flow preserves the lift 
of the Masur Veech measure. Furthermore, a periodic orbit
for $\Phi^t$ in $\hat {\cal Q}$ is defined by a pseudo-Anosov 
mapping class which preserves the combinatorics and, in 
particular, fixes singular points. From now on we 
always denote by $\hat {\cal Q}$ this finite covering (in the orbifold
sense) of the component ${\cal Q}$.

Call a point $q\in 
\hat {\cal Q}$ \emph{recurrent} if 
for every neighborhood $U$ of $q$ there is a sequence of 
times $t_i\to \infty, s_j\to -\infty$ such that
$\Phi^{t_i}q\in U, \Phi^{s_j}q\in U$.
The set of such points has full measure for every 
$\Phi^t$-invariant 
probability measure on $\hat {\cal Q}$. 

Denote by $\lambda$ the normalized lift of the Masur Veech measure
on $\hat {\cal Q}$. 
Given a contractible set $U\subset \hat {\cal Q}$, a point 
$u\in U$ and a number $T>>0$ such that 
$\Phi^Tu\in U$, a \emph{characteristic curve}
of the \emph{pseudo-orbit}
$(u,\Phi^Tu)$ is a closed curve containing the orbit segment
$\cup_{0\leq t\leq T}\Phi^tu$ which is obtained by connecting 
$\Phi^Tu$ to $u$ with an arc which is entirely contained in $U$. 
The following
result is Proposition 4.5 of \cite{H22}.

\begin{proposition}\label{prop45}
Let $q\in \hat {\cal Q}$  be a recurrent point. 
Then for every neighborhood $U$ of $q$ 
and for all $\delta >0,\eta>0$ there are contractible closed 
neighborhoods $Z_1\subset Z_2\subset U$ of $q$ and 
there is a Borel set $Z_0\subset Z_1$ 
and a number $R_0>0$ with the following properties.
\begin{enumerate}
\item $\lambda(Z_2)\leq (1-\delta)^{-1}\lambda(Z_0)$.
\item For some integer $m>1/\delta$,  
a $\Phi^t$-orbit intersects $Z_1,Z_2$ in arcs of length 
$2t_0<\eta/m$. 
\item The set 
\[Z_3=\cup_{-t_0(m-2)\leq t\leq t_0(m-2)}\Phi^tZ_1\]
is contained in the interior of 
\[Z_4=\cup_{-t_0m\leq t\leq t_0m}\Phi^tZ_2\subset U.\] 
\item Let $z\in Z_0$ and let $T>R_0$ be such that
$\Phi^Tz\in Z_3$. Then there 
exists a path connected set $B(z)\subset Z_2$ containing $z$, 
with 
\[\quad \quad\Phi^TB(z)\subset Z_4\text{ and } 
\lambda(B(z))\in [(1-\delta)e^{-hT}\lambda(Z_1),
(1-\delta)^{-1}e^{-hT}\lambda(Z_1)].\] 
There is a periodic orbit $\gamma$ for $\Phi^t$ of length contained in
$[T-mt_0,T+mt_0]$ such that for each $u\in B(z)$, 
the characteristic curve of the pseudo-orbit 
$(u,\Phi^Tu)$ with endpoints in $Z_4$ 
determines the same component $\gamma(z,T)$ of 
$\gamma\cap Z_4$.
If $u\in Z_0-B(z)$ and if $\Phi^Tu\in Z_3$, then the
arc $\gamma(u,T)$ is disjoint from $\gamma(z,T)$. 
\end{enumerate}
\end{proposition}

\begin{remark}\label{good}
In the formulation of Proposition 4.5 of \cite{H22}, 
the point $q$ is required to be a good recurrent point.
The additional constraint comes from the difficulty that
a component of a stratum is not a manifold. 
In the finite manifold cover  $\hat {\cal Q}$ of ${\cal Q}$,
every point is good. 
\end{remark}

For a set 
${\cal P\cal A}$ of conjugacy classes of pseudo-Anosov elements
in ${\rm Mod}(S)$ and for 
$R>0$ let
$n({\cal P\cal A},R)$ be the number of elements in 
${\cal P\cal A}$ consisting of
conjugacy classes of translation length at most $R$.
We use Proposition \ref{prop45} to establish 
a counting result for periodic orbits in components of strata
with combinatorial constraints.

\begin{proposition}\label{countone} 
For any component $\tilde {\cal Q}$
of the stratum 
$\tilde{\cal Q}(m_1,\dots,m_\ell;-m,p)$ there are numbers 
$d=d(\tilde {\cal Q})>0$, $R_0=R_0(\tilde {\cal Q})>0$ with the following
properties. 
Let $\tau\in {\cal L\cal T}(\tilde {\cal Q})$ and let 
${\cal P\cal A}(\tau,d)$ be the set 
of all conjugacy classes of pseudo-Anosov elements 
$\phi\in {\rm Stab}(\tilde {\cal Q})$ 
with the following properties.
\begin{enumerate}
\item[a)] $\phi(\tau)\prec\tau$.
\item[b)] The matrix $A(\phi,\tau)$ is positive, 
and the ratios of the entries
of $A(\phi,\tau)$ are bounded from above by $d$.
\item[c)] $\phi$ preserves the combinatorics. 
\end{enumerate}
Then we have  
\[n({\cal P\cal A}(\tau,c),R)\geq \frac{1}{dh(\tilde {\cal Q})R}e^{h(\tilde {\cal Q})R}\]
for all $R\geq R_0$.
\end{proposition}
\begin{proof}
Let $\tau\in {\cal L\cal T}(\tilde {\cal Q})$. We claim that 
we can find
a pseudo-Anosov mapping class $\phi\in {\rm Mod}(S)$
with the following properties.
\begin{enumerate}
\item $\phi (\tau)\prec \tau$.
\item The image of any branch of 
$\phi (\tau)$ under a carrying map $\phi(\tau)\to \tau$ is 
surjective. 
\item $\phi$ preserves the combinatorics.
\end{enumerate}

To see that such a mapping class exists, note that
we can always find a splitting sequence starting at 
$\tau$ whose endpoint $\eta$ has the property that 
the restriction of the carrying map to each branch of 
$\eta$ is onto $\tau$ (see \cite{H09} for more information). 
By Lemma \ref{carryingcontrol}, there is 
some $\phi\in {\rm Mod}(S)$ such that
$\phi(\tau)\prec \eta$, and this mapping class $\phi$ has the required properties.
Note that $\phi$ is pseudo-Anosov by Lemma \ref{perron} since 
$A(\phi,\tau)$ is a positive matrix. 
By perhaps passing to a finite power, we may assume that 
$\phi$ preserves the combinatorics.

Denote by ${\cal Q}_0(\tau)$ the set of quadratic differentials 
$q\in {\cal Q}(\tau)$ whose horizontal measured geodesic 
lamination deposits the weight one on $\tau$; then 
${\cal Q}_0(\tau)$ is a local cross section for the
Teichm\"uller flow on $\tilde {\cal Q}$.
Let $\hat {\cal U}\subset {\cal Q}_0(\tau)$ be the set of 
all differentials $v\in {\cal Q}_0(\tau)$ whose horizontal measured geodesic lamination
is carried by $\phi(\tau)$ and  whose vertical measured geodesic 
lamination is carried by the dual $\phi^{-1}(\tau^*)$ 
of $\phi^{-1}(\tau)$. 
By possibly replacing $\phi$ by a positive power, we may assume
that $\hat {\cal U}\subset \tilde {\cal Q}$, that is, no point in 
$\hat {\cal U}$ is contained in a stratum of the boundary of $\tilde {\cal Q}$. 

To be more precise, 
let $\mu,\nu$ be the attracting and repelling fixed point for the 
action of $\phi$ on ${\cal P\cal M\cal L}$, respectively. If 
$v\in \hat {\cal U}$ then the horizontal measured geodesic lamination
of $v$ is contained in $\phi{\cal V}(\tau)$, and the vertical measured geodesic
lamination is contained in $\phi^{-1}{\cal V}^*(\tau)$. 
As $\cap_{k\geq 0}\phi^k{\cal V}(\tau)$ equals the line whose projective class equals
$\mu$, and $\cap_{k\geq 0}\phi^{-k}{\cal V}^*(\tau)$ equals the line whose 
projective class equals $\nu$, and as $\tilde {\cal Q}$ is an open subset of 
${\cal Q}(\tau)$, we obtain that by perhaps replacing $\phi$ by $\phi^k$ for 
a suitably chosen $k>0$, we may assume that $\hat {\cal U}\subset \tilde {\cal Q}$.

For every $\epsilon >0$, 
$\tilde {\cal U}(\epsilon)=\cup_{t\in (-\epsilon,\epsilon)}\Phi^t \hat {\cal U}\cap \tilde {\cal Q}$ 
is an open subset of 
$\tilde {\cal Q}$. By perhaps decreasing $\tilde {\cal U}(\epsilon)$ we may assume that
$\tilde {\cal U}(\epsilon)$ 
projects homeomorphically to 
an open subset ${\cal U}(\epsilon)$ of the finite cover 
$\hat {\cal Q}$ of ${\cal Q}$. 
The unit cotangent line of the axis of the mapping class $\phi$ descends to 
a periodic orbit of $\Phi^t$ in $\hat {\cal Q}$. 

Let $q\in {\cal U}(\epsilon)$ be a point on the periodic orbit defined by $\phi$.
Choose an open neighborhood $V\subset {\cal U}(\epsilon)$ of $q$ which is
sufficiently small that the following holds true. 
Let $T>0$ be the period of the periodic point $q$.
%
%
If $w\in V$ then the $\Phi^t$-orbit of $w$ intersects 
${\cal U}(\epsilon)$ in $\Phi^Tw, \Phi^{2T}w$, $\Phi^{-T}w$ and 
$\Phi^{-2T}w$. Furthermore, there is an open tubular neighborhood 
$N\subset \hat {\cal Q}$ of the periodic orbit through $q$ which is homeomorphic to a ball 
bundle over a circle and such that for every $w\in V$, we have
$\cup_{0\leq t\leq T}\Phi^Tw\in N,\cup_{-T\leq t\leq 0}\Phi^tw\in N$.
Let $U=V\cap \Phi^TV\cap \Phi^{-T}V\subset {\cal U}(\epsilon)$.
Then $U$ is an open neighborhood of $q$ in ${\cal U}(\epsilon)$.



Let $Z_0\subset Z_1\subset Z_2\subset U$ 
and $Z_3\subset Z_4\subset U$ and $R_0>0$ 
be as in Proposition \ref{prop45}. 
If $w\in Z_0$ and if 
$\Phi^R w\in Z_3$ for some $R>\max\{4T,R_0\}$ then 
Proposition \ref{prop45} shows that the pseudo-orbit 
$(w,\Phi^Rw)$ determines a periodic orbit for 
$\Phi^t$, and this orbit defines up to conjugation 
a pseudo-Anosov mapping class $\psi$. 
We may choose a point 
$w(\psi)\in {\cal Q}_0(\tau)$ on the cotangent line of the axis of 
$\psi$. By the construction of $\hat {\cal Q}$, the mapping class
$\psi$ preserves the combinatorics. Thus 
by Lemma \ref{traintrackexpansion}, 
we have $\psi(\tau)\prec \tau$.

Consider the pseudo-orbit $(\Phi^{-T}w,\Phi^{T+R}w)$ with endpoints in 
$V$. Let $\beta\subset \tilde {\cal Q}$ be a lift of this pseudo-orbit 
to $\tilde {\cal Q}$ and let as before $P:\tilde {\cal Q}(S)\to {\cal T}(S)$ be the 
canonical projection. Up to replacing all mapping classes by conjugates
and renaming, by the choice of $V$, the mapping class
$\phi \circ \psi \circ \phi$ maps $P\beta(-T)$ into a small neighborhood of 
$P\beta(R+T)$. Up to adjusting the set $N$ which entered the definition of 
the open set $U$ and taking into account that a torsion free finite index subgroup
of ${\rm Mod}(S)$ 
acts properly discontinuously and freely on ${\cal T}(S)$, 
the mapping class $\phi\circ \psi \circ  \phi$ is unique with this property. 
Moreover, the translation length of $\phi\circ \psi \circ \phi$ does not exceed 
$R+2T+\chi$ where $\chi>0$ is a constant only depending on 
the choice of $U$. Namely, the translation length of $\phi \circ \psi \circ \phi$
is the infimum of the Teichm\"uller distance between a point in 
${\cal T}(S)$ and its image under $\phi\circ \psi\circ \phi$, and this 
distance does not exceed $2T+R+\chi$ for a universal constant 
$\chi$ by the above discussion.

The map $\psi$ has $\tau$ as a train track expansion, and the same holds true for 
$\phi$. 
Concatenation shows that $\phi\circ \psi \circ \phi$ has 
$\tau$ as a train track expansion as well. 
We claim that
there is a number $d>0$ not depending on $\psi$
such that the ratios of the entries of the matrix 
\[A=A(\phi,\tau)A(\psi,\tau)A(\phi,\tau)\]
are bounded from above by $d$.

Namely, let $\ell>0$ be the maximum of the ratios of the entries
of the matrix $A(\phi,\tau)$. 
Then up to a factor of at most $\ell$, the entries in 
each line of the matrix $A(\psi,\tau)A(\phi,\tau)$ 
coincide with a fixed multiple of 
the sum of the entries
of the matrix $A(\psi,\tau)$ in the same line. 
In particular,
since $A(\psi,\tau)$ can not have a line all of whose entries
vanish, the matrix $A(\psi,\tau)A(\phi,\tau)$ is positive, and
the ratios of its entries in a fixed line are bounded from
above by $\ell$.

Similarly, up to a factor of at most $\ell$, the 
entries in each row of the matrix 
$A$ coincide  with 
a fixed multiple of the sum of the entries
of $A(\psi,\tau)A(\phi,\tau)$ in the same row.  
By the discussion in the previous paragraph, this implies
that the ratios of the entries of the positive matrix $A$ are
bounded from above by $\ell^2$. Lemma \ref{perron} 
shows that $\phi\circ \psi \circ \phi$ is pseudo-Anosov. 
Furthermore, we have $\phi\circ \psi\circ \phi({\cal V}(\tau))\subset 
\phi({\cal V}(\tau))$ and $\phi^{-1}\psi^{-1}\phi^{-1}({\cal V}^*(\tau))\subset 
\phi^{-1}({\cal V}^*(\tau))$. Using once more the Brouwer fixed point theorem, this 
implies that the attracting fixed point for the action of 
$\phi \circ \psi\circ \phi$ on ${\cal P\cal M\cal L}$ 
is contained in the projectivization of $\phi({\cal V}(\tau))$, and the repelling 
fixed point is contained in the projectivization of 
$\phi^{-1}({\cal V}^*(\tau))$. As a consequence, the 
cotangent line of an axis of $\phi\circ \psi \circ \phi$ is contained in 
$\tilde {\cal Q}$.
Thus any mapping class of the form $\phi\circ \psi \circ \phi$ 
constructed from a pseudo-orbit $(u,\Phi^Ru)$ with $R>R_0$, 
$u\in Z_0$ and $\Phi^Ru\in Z_3$ 
has properties (a) -(c) stated in the proposition.

We are left with counting the number of such periodic orbits. 
Put $h=h(\tilde {\cal Q})$. Note that by Proposition \ref{prop45},
a point $w\in Z_0$ with $\Phi^Rw\in Z_3$ determines a subsegment of 
a periodic orbit passing through $Z_4$ of uniform length $2t_0$, and
points which determine the same segment are contained in  
a subset $B\subset Z_2$ with $\Phi^R(B)\subset Z_4$ 
whose measure is bounded from 
above by $(1-\delta)^{-1}e^{-hR}\lambda(Z_1)$. 
Thus each such set accounts for at most the volume
$(1-\delta)^{-1}e^{-hR}\lambda(Z_1)$ from 
$\lambda(\Phi^RZ_0\cap Z_3)$ and hence 
there are at least $(1-\delta)e^{hR}\lambda(\Phi^RZ_0\cap Z_3)/
\lambda(Z_1)$ such orbits. 
As a consequence, to complete the 
proof of the proposition it suffices to show that for sufficiently large $R>>T$, 
the Masur Veech measure of $\Phi^RZ_0\cap Z_3$ is bounded from below by
a universal constant. 
Namely, as the length of an intersection component of an orbit with 
the set $Z_3$ equals 
a fixed number $2t_0>0$, a periodic orbit of length $R$ can not intersect 
$Z_3$ in more than
$R/2t_0$ components.

Now the normalized Masur Veech measure $\lambda$ is mixing for the 
Teichm\"uller flow $\Phi^t$ on $\hat {\cal Q}$ 
and hence for a given number $\epsilon >0$ 
we have 
\[\lambda(\Phi^R Z_0\cap Z_3)\geq (1-\epsilon)\lambda(Z_0)\lambda(Z_3)\] 
for all sufficiently large $R$, say for all $R\geq R_0$. Then the statement of the 
proposition holds true for all $R\geq R_0$ and for a number $d>0$ depending on 
the component $\hat {\cal Q}$, 
the choice of the pseudo-Anosov mapping class $\phi$
and the choice of the sets 
$Z_0,Z_3$.
\end{proof}

\section{Periodic orbits in the thin part of strata}
\label{periodicthin}

In this section we only consider strata of differentials 
without marked regular points. This means that 
strata of abelian differentials are
defined on surfaces without punctures, and 
a puncture of a surface $S$ corresponds to a simple pole of 
a quadratic differential. 
Our main goal is to prove Theorem \ref{main} from 
the introduction. Strata of differentials with marked regular points
will be used in the proof. 

Let $S$ be a surface of genus $g\geq 0$ with $n\geq 0$
punctures and $3g-3+n\geq 5$. 
Recall that a marked quadratic (or abelian) 
differential $q\in \tilde{\cal Q}(S)$ 
defines a marked singular euclidean metric of area one 
on the surface $S$
with singularities at the zeros and at the poles of 
the differential. There is a unique 
finite volume complete hyperbolic
metric on $S$ for the underlying conformal structure.

Define the \emph{systole} of this hyperbolic metric to be
the smallest length of a simple closed geodesic.
We write ${\cal T}(S)_\epsilon\subset {\cal T}(S)$ 
for the space of all 
marked complete hyperbolic metrics on $S$ 
of finite volume whose systole is at least
$\epsilon$. The mapping class group ${\rm Mod}(S)$ 
acts properly and
cocompactly on ${\cal T}(S)_\epsilon$.

Given a marked abelian or quadratic differential 
$q\in \tilde {\cal Q}(S)$, 
the \emph{$q$-length} of an \emph{essential} 
simple closed curve $c$, that is, 
a simple closed curve which is not contractible
and not freely homotopic into a puncture,
is defined to be the infimum of the lengths 
with respect to the singular euclidean metric 
of any curve which is freely homotopic to $c$.

The following observation is a fairly easy consequence
of invariance under the action of the mapping class
group and cocompactness. 
A much stronger and more precise version is due 
to Rafi \cite{R14}.
For its formulation, let 
\[P:\tilde{\cal Q}(S)\to {\cal T}(S)\] 
be the canonical projection which
associates to a marked area one
quadratic (or abelian) differential its underlying 
marked hyperbolic metric.

\begin{lemma}\label{systole}
For every $\chi >0$ there is a number 
$\delta=\delta(S,\chi) >0$
with the following property. Let $q\in \tilde{\cal Q}(S)$
and assume that there is an essential simple closed
curve on $S$ of $q$-length at most $\delta$;
then $Pq\not\in {\cal T}(S)_\chi$.
\end{lemma}
\begin{proof}
By the collar lemma for hyperbolic surfaces, every
cusp of a hyperbolic surface has a standard embedded neighborhood.
Such a neighborhood is homeomorphic to a punctured disk.
Up to a geometric adjustment of the neighborhoods, 
the hyperbolic distance between any two such 
cusp neighborhoods is bounded from below by 
a universal positive constant.


For a complete finite volume
hyperbolic metric $h$ on $S$ 
let $K(h)\subset (S,h)$ be the complement of the standard neighborhoods 
of the cusps of $S$.
By Lemma 3.3 of \cite{M94}, for every $\chi>0$ there
is a number $L=L(\chi)>1$ such that for every 
$q\in \tilde{\cal Q}(S)$ with $Pq\in {\cal T}(S)_\chi$,
the singular euclidean metric defined by $q$ is
$L$-bilipschitz equivalent to the hyperbolic metric $Pq$ 
on the compact set $K=K(Pq)$. 
By the choice of the standard neighborhoods
of the cusps, 
every essential simple closed curve 
on $S$ either is contained in $K$, or it 
intersects the set $K$ in a union of arcs whose hyperbolic
length is bounded from below by some fixed number
$c\in (0,\chi)$ only depending on the topological type of $S$ and on $\chi$.
Then the $q$-length of any essential simple closed
curve on $S$ is not smaller than $c/L$.
This shows the lemma.
\end{proof}

\begin{remark}\label{notshort}
Lemma \ref{systole} does not state that an essential simple closed 
curve of short $q$-length has short hyperbolic length for the 
hyperbolic metric underlying the  
complex structure of $q$. In fact, this is not true in general \cite{R14}. 
\end{remark}

For $\chi>0$ define
\[\tilde{\cal Q}(S)_\chi=\{q\in \tilde{\cal Q}(S)\mid 
Pq\in {\cal T}(S)_{\chi}\}.\] The sets
$\tilde{\cal Q}(S)_\chi$ are invariant
under the action of ${\rm Mod}(S)$ on
$\tilde{\cal Q}(S)$. Their projections
\[{\cal Q}(S)_\chi=\tilde{\cal Q}(S)_\chi/
{\rm Mod}(S)\subset {\cal Q}(S)\] 
to ${\cal Q}(S)$ are compact and satisfy
${\cal Q}(S)_\chi\subset {\cal Q}(S)_\delta$ for
$\chi>\delta$ and $\cup_{\chi >0}{\cal Q}(S)_\chi=
{\cal Q}(S)$.

For a component 
$Q$ of a stratum 
of quadratic or abelian differentials and for 
$\chi >0, R>0$ let 
\[n({\cal Q},R)^{<\chi}\geq 0\]
be the number of periodic orbits for the 
Teichm\"uller flow $\Phi^t$ of length at most $R$ which 
are contained in ${\cal Q}-{\cal Q}(S)_\chi$.
Our goal is to show that 
for every $\chi >0$ and for sufficiently large $R$ depending on 
$\chi$, 
the number of such orbits is at least 
$de^{(h({\cal Q})-1)R}/R$ for a number $d=d({\cal Q},\chi)>0$.

The strategy is to use a train track $\tau$ in 
special form for ${\cal Q}$ as a combinatorial tool to lift periodic 
orbits on components of the principal boundary of ${\cal Q}$
into the thin part of ${\cal Q}$.   
We recall the important
properties of such a train track $\tau$.

\begin{enumerate}
\item $\tau$ contains two clean
vertex cycles $c_1,c_2$. 
Either the surface $S-c_1-c_2$ is connected, or 
it consists of a connected component $N$ 
which is different from a sphere with at most three holes
or a torus with at most one hole and one or
two additional components which are 
twice punctured disks. 
\item Let $S_0$ (or $S_1,S_2$) 
be the surface which is obtained from the 
component of $S-c_1-c_2$ (or of $S-c_1,S-c_2$) 
different from a sphere with at most three holes
by replacing the boundary 
circles by punctures.
Then the graph $\sigma_0$ on $S_0$ 
(or $\sigma_1,\sigma_2$ on $S_1,S_2$) obtained by
removing from $\tau$ all branches which are
incident on a switch
in $c_1\cup c_2$ (or incident on a switch in $c_1,c_2$) 
is a connected large train track, either on $S_i$ or on the
surface obtained from $S_i$ by removing one or two of 
the newborn punctures. 
\end{enumerate}

We showed in Section 4 that for $i=1,2$ 
there is a component 
${\cal Q}_i$ of a stratum of differentials for the surface $S_i$
(possibly with marked regular points or with some of the newborn 
punctures removed) such that
$\sigma_i\in {\cal L\cal T}(\tilde{\cal Q}_i)$
where $\tilde {\cal Q}_i$ is the preimage of ${\cal Q}_i$. 
The component ${\cal Q}_i$ is determined as follows.

{\sl Case 1:} 
$n\geq 2$ and the simple closed curve $c_i$ is 
separating.

Then $c_i$ bounds a twice punctured disk.
There is a simple closed curve $\alpha$ embedded in the train 
track $\sigma_i$ 
(however with cusps) which is
freely homotopic to $c_i$. 
Viewing $\sigma_i$ as a train track on the surface $S_i$, 
the curve $\alpha$ encloses
the added marked point on $S_i$ (which replaces the circle 
$c_i$). We require that a differential in ${\cal Q}_i$ has 
\begin{enumerate}
\item[$\bullet$] a simple pole at this marked point in the
case that $\alpha$ has a single cusp,
\item[$\bullet$] a regular point if 
$\alpha$ is a bigon or 
\item[$\bullet$] a zero if $\alpha$ has at least
three cusps  (in which
case we remove the marked point). 
\end{enumerate}

Let $h({\cal Q}_i)$ be the dimension of the complex algebraic orbifold
containing ${\cal Q}_i$ as a real hypersurface. We have
$h({\cal Q}_i)=h({\cal Q})-2$. A degeneration of 
differentials in ${\cal Q}$ to a differential in 
${\cal Q}_i$ corresponds to a shrinking half-pillowcase.

{\sl Case 2:} 
The simple closed curve $c_i$ is non-separating.

Then $S-c_i$ has 
two boundary components. Each of these
boundary components is contained in a complementary component of $\sigma_i$
(these components may coincide). If there are two distinct
such components then each of these components 
is an annulus. One boundary component
of such an annulus is the curve $c_i$,
and the second boundary component $\alpha$ is contained in $\sigma_i$.
As before, put a simple pole on the puncture of $S_i$ which replaces the curve $c_i$ if
$\alpha$ has a single cusp, view the puncture of $S_i$ as a regular marked point if 
$\alpha$ is a bigon, and remove the puncture and view is as 
a zero if $\alpha$ contains at least three
cusps. Once again, $h({\cal Q}_i)=h({\cal Q})-2$. A degeneration of differentials in 
${\cal Q}$ to a differential in ${\cal Q}_i$ corresponds to 
a shrinking cylinder.  
If both boundary components of $S-c_i$ are contained in the same
complementary component of $\sigma_i$ then we proceed in exactly the same way.

Our goal is to use the growth estimate from Section 5 for periodic 
orbits in ${\cal Q}_i$ which are defined by
pseudo-Anosov mapping classes preserving the combinatorics, 
with train track expansion
$\sigma_i$. Such mapping classes can be lifted to 
reducible mapping classes on $S$. We then 
concatenate suitably chosen such lifts 
to a pseudo-Anosov element
for $S$ with train track expansion $\tau$. We show that this can
be accomplished in such a way that the axis of this pseudo-Anosov
mapping class is contained in ${\cal T}(S)-{\cal T}(S)_\chi$
where $\chi>0$ is an arbitrarily chosen constant. 

Recall from Proposition \ref{countone} the counting constants
$d_i=d({\cal Q}_i)>0$
$(i=0,1,2)$. 
Choose elements $\phi_i\in {\rm Mod}(S_i)$
with the properties stated in Proposition \ref{countone}
for the numbers $d_i$. Each of these mapping classes 
has $\sigma_i$ as a train track expansion and 
fixes each of the complementary components of 
$\sigma_i$. These mapping classes 
can be lifted to the mapping
class group of the bordered surface 
$S-c_i$ (or of the bordered surface $S-(c_1\cup c_2)$ if $i=0$)
which consists of isotopy classes of diffeomorphisms 
 fixing the curve $c_i$ (or $c_1\cup c_2$) pointwise.
Such an extension in turn can be viewed as a
reducible element of 
${\rm Mod}(S)$. 

For a bordered surface $X$ with a distinguished boundary 
component $c$ and the surface $X^\prime$ 
obtained from $X$ by replacing $c$ by a 
puncture, there exists an exact sequence
\[1\to \mathbb{Z}\to {\rm Mod}(X)\to {\rm Mod}(X^\prime)\to 1\]
where the infinite cyclic group $\mathbb{Z}$ is the group of Dehn twists
about the boundary curve $c$. Thus 
a choice of an extension of $\phi_i$ as described in the
previous paragraph is by no means
unique. However, as $\phi_i$ is required to 
preserve the combinatorics, 
we can use the train track $\tau$ to construct 
from $\phi_i$ a unique such extension.

\begin{lemma}\label{dragging}
There is a natural choice of an
extension of $\phi_i$ to an element
of ${\rm Mod}(S)$, again denoted by $\phi_i$, which fulfills 
$\phi_i(\tau)\prec \tau$ and 
is such that $\phi_i(\tau)$ is obtained from 
$\tau$ by a splitting and shifting sequence of $\tau$ which 
does not involve a split at any branch in $\tau-\sigma_i$.
\end{lemma}
\begin{proof} Since $\phi(\sigma_i)\prec \sigma_i$,
there exists a splitting
and shifting sequence connecting $\sigma_i$ to $\phi_i(\sigma_i)$. 
Our goal is to extend this sequence of a splitting and
shifting sequence of $\tau$.

By construction, there are at most two 
small branches $b_1,b_2$ of $\tau$ which connect
$\sigma_i$ to the primitive vertex cycle $c_i$ of $\tau$.
Let $e$ be a large branch of $\sigma_i$. 
If none of the two small branches $b_1,b_2$ of $\tau$ 
has an endpoint in $e$, then $e$ is a large branch of $\tau$ and
a split of $\sigma_i$ at $e$ can be viewed as a split of $\tau$ at $e$. 
The split track contains $c_i$ as a clean vertex cycle. 

If an endpoint of one of the branches $b_1,b_2$ is contained in $e$ 
then there is a modification
of $\tau$ by a sequence of shifts and splits at branches of $\tau$ contained in $e$ 
such that the modified
train track $\tau^\prime$ is large, it contains $e$ as a large branch and 
the images of $b_1,b_2$ in $\tau^\prime$ are small branches in $\tau^\prime$.
These splits and shifts move an endpoint of $b_i$ contained in $e$ across
an endpoint of $e$. In this way we obtain 
from a splitting and shifting sequence connecting 
$\sigma_i$ to $\phi_i(\sigma_i)$ a splitting and shifting sequence 
which connects $\tau$ to a train track $\hat \phi_i(\tau)$ containing $c_i$ as 
a primitive vertex cycle and such that the restriction of a carrying 
map $\hat \phi_i(\tau)\to \tau$ to $\hat \phi_i(\sigma_i)$ coincides with the 
carrying map $\phi_i(\sigma_i)\to \sigma_i$. 

Since the positions of the endpoints of the small branches $b_i$ are 
moved along sides of complementary components of $\sigma_i$, 
the train track $\hat \phi_i(\tau)$ may not be isomorphic to $\tau$.
However, since $\phi$ preserves the combinatorics by assumption, it
maps the complementary component $C$ of $\sigma_i$ which contains
the curve $c_i$ to the complementary component 
$\phi_i(C)$ of $\phi_i(\sigma_i)$ containing
$c_i$, and it maps a side of $C$ containing an endpoint of one of the
small branches $b_i$ to the side of $\phi_i(C)$ containing
the endpoint of the same branch (with the canonical identification).
Thus there is a unique way to slide the endpoints of $b_1,b_2$ along
these sides in such a way that the extension of $\phi_i(\sigma_i)$
constructed in this way is isomorphic to $\tau$ with an 
isomorphism (that is, a mapping class), 
which preserves $c_i$ and restricts to $\phi_i$ on $\sigma_i$.
This is what we wanted to show.
\end{proof}

In the sequel we will always identify the map $\phi_i$ with this
particular extension, that is, we view $\phi_i$ as a reducible
mapping class for $S$ with $\phi_i(\tau)\prec\tau$ and such that
there is a carrying map $\phi_i(\tau)\to \tau$ which is the 
identity on $\tau-\sigma_i$. The map $\phi_0$ is viewed as a reducible
element of ${\rm Mod}(S)$ fixing $c_1\cup c_2$ 
pointwise.

In the statement of the following lemma,
$\circ$ means
composition from the left, that is, $a\circ b$ represents the mapping
class obtained by applying $b$ first followed by an
application of $a$. This amounts to 
using the notational convention
that $(\phi\circ\psi)(\tau)$ is the train track obtained from
$\tau$ by first changing the marking with $\psi$ and
afterwards with $\phi$.

\begin{lemma}\label{pA2}
For every $k>0$ the mapping class 
\[\zeta(k)=(\phi_0^k\circ \phi_2\circ 
\phi_0^{k})\circ (\phi_0^k\circ\phi_1\circ\phi_0^{k})\]
is pseudo-Anosov.
\end{lemma}
\begin{proof}
Note first that $\zeta(k)(\tau)\prec \tau$ for all $k$. 
Namely, by assumption on $\phi_i$ 
we have $\phi_i(\tau)\prec\tau$ (see the discussion in Lemma \ref{dragging}
and its proof) 
and hence by invariance of the carrying relation under the
action of the mapping class group and induction, we conclude that
\[\zeta(k)(\tau) =
(\phi_0^k\circ \phi_2\circ \phi_0^{2k}\circ\phi_1 \circ 
\phi_0^{k}(\tau))\prec 
\tau.\]

Let as before $m$ be the number of branches of $\tau$ and let
$\mathbb{R}^m$ be the real vector space with basis
the branches of $\tau$.
Write $\mathbb{R}^m=\mathbb{R}^{\ell_0}\oplus 
\mathbb{R}^{\ell_1}\oplus \mathbb{R}^{\ell_2}$
where $\mathbb{R}^{\ell_0}$ is the real vector space with basis the 
branches of $\tau$ contained 
in $\sigma_0$ and where for $i=1,2$ the vector space
$\mathbb{R}^{\ell_i}$ has a basis consisting of the branches of 
$\tau-\sigma_{i+1}$. Then $\mathbb{R}^{\ell_0}\oplus 
\mathbb{R}^{\ell_i}$ is the real vector space with basis
the branches of $\sigma_i$.

Since $\zeta(k)^2$ is pseudo-Anosov if and only if this
is true for $\zeta(k)$, by  
Lemma \ref{perron}, it suffices to show that the linear self map
\[A(\zeta(k)^2,\tau)=A(\phi_0^k,\tau)\cdots A(\phi_1,\tau)A(\phi_0^k,\tau)\]
of $\mathbb{R}^m$ 
is given with respect to the above basis by a positive matrix.
This is equivalent to stating that 
a carrying map 
$\zeta(k)^2\tau\to \tau$ maps every branch
of $\zeta(k)^2\tau$ \emph{onto} $\tau$.

For $i=1,2$ define
\[A_i=A(\phi_0^k\circ \phi_i\circ \phi_0^k,\tau).\]
The linear map $A_1$ preserves the decomposition
$\mathbb{R}^m=\mathbb{R}^{\ell_0+\ell_1}\oplus \mathbb{R}^{\ell_2}$ 
and therefore
it is given by a matrix in block form. 
The square matrix which corresponds to the restriction of 
$A_1$ to $\mathbb{R}^{\ell_0+\ell_1}$ is positive, and the
square matrix which corresponds to the action on 
$\mathbb{R}^{\ell_2}$ is the identity. 
The same holds true for the linear map $A_2$, with the roles of $\ell_1$ and 
$\ell_2$ exchanged. 

As a consequence, 
the image of a basis vector in $\mathbb{R}^{\ell_0+\ell_1}$ under the
matrix $A_2A_1$ is a positive vector in $\mathbb{R}^m$. 
Similarly, the image of a basis vector in $\mathbb{R}^{\ell_0+\ell_2}$ under
the matrix $A_1A_2$ is a positive vector in $\mathbb{R}^m$ and hence 
the matrix
$A_2A_1A_2A_1$ is indeed positive. This is what we wanted to show.
\end{proof}

For the proof of Theorem \ref{main}, we 
count periodic orbits 
in the $\chi$-thin part of a component ${\cal Q}$ of a 
stratum by counting orbits which 
are defined by pseudo-Anosov classes of the form 
described in Lemma \ref{pA2} where we let $\phi_1$ vary, and 
we fix $\phi_0,\phi_2$ and some large enough $k$.
The next lemma shows that indeed, such mapping classes give rise
to periodic orbits in the thin part of moduli space. 
Recall that the axis of a pseudo-Anosov mapping class
$\phi$ is the Teichm\"uller geodesic which is invariant under $\phi$.


\begin{lemma}\label{geometriccontrol}
For every $\chi >0$ there exists a number $k_0=k_0(\chi) >0$
with the following property. If $k\geq k_0$ then the axis of each of 
the mapping classes $\zeta(k)$ as in 
Lemma \ref{pA2} does not intersect ${\cal T}(S)_\chi$.
\end{lemma}
\begin{proof} By Lemma \ref{systole}, it suffices to show the following.
For every $\epsilon >0$, there exists a number $k=k(\epsilon)>0$ such that
if $k>k(\epsilon)$ and if
$q\in \tilde {\cal Q}$ is contained in the cotangent line of an axis of 
$\zeta(k)$, then there exists an essential simple closed curve $c$ on $S$ of 
$q$-length at most $\epsilon$.

Using the notations from Lemma \ref{dragging}, 
the carrying map $\phi_i(\tau)\to \tau$ can be chosen in such
a way that it maps each branch of $\phi_i(\sigma_i)$ 
onto $\sigma_i$, and it induces the identity on 
$\tau-\sigma_i$ $(i=0,1,2)$.

For a transverse measure $\mu\in {\cal V}(\tau)$ and a subset
$B$ of the branches of $\tau$ denote by $\mu(B)$ the total mass
deposited by $\mu$ on the branches in $B$. Clearly $\mu(B)\geq 0$ for all 
$\mu,B$. 

By Proposition \ref{countone} and the construction,  
for $i=1,2$ 
there is a number $a_i>0$ with the following 
property. Let $\mu_i$ be a measured geodesic
lamination on $S$ which is carried by 
$\phi_i(\sigma_i)$
and which defines the transverse measure
$\mu_i\in {\cal V}(\phi_i(\sigma_i))\subset 
{\cal V}(\sigma_i)\subset {\cal V}(\tau)$;
then $\mu_i(b_1)/\mu_i(b_2)\leq a_i$ for 
any two branches $b_1,b_2$ of $\sigma_i$.

This implies 
the existence of a number $a>0$ with the following property.
Let $\mu$ be \emph{any} measured geodesic lamination
which is carried by $\phi_i(\tau)$. Assume that the
transverse measure on $\phi_i(\tau)$ defined by $\mu$ is
not supported in $\phi_i(\tau)-\phi_i(\sigma_i)$. 
Let $\hat \mu$ be the transverse
measure on $\tau$ induced
from $\mu$ by a carrying map 
$\phi_i(\tau)\to \tau$; then  
\begin{equation}\label{branchweight}
\hat \mu(b_i)\geq 2a \hat \mu(\sigma_i )\end{equation}
for every branch $b_i$ of 
$\sigma_i$ $(i=1,2)$.

Let ${\cal V}_0(\tau)\subset {\cal V}(\tau)$ be
the set of all transverse measures on $\tau$ of total mass one.
Note that ${\cal V}_0(\tau)$ is naturally homeomorphic to the
projectivization of ${\cal V}(\tau)$.
For $\epsilon \in (0,\frac{1}{2})$ and $i=0,1,2$ let
$C_i(\tau,\epsilon)$
be the closed subset of  ${\cal V}_0(\tau)$
containing all transverse measures $\nu$
with the following properties.
\begin{enumerate}
\item
$\nu(\tau-\sigma_i)\leq  \epsilon$.
\item For any branch $b_i$ of 
$\sigma_i<\tau$ we
have $\nu(b_i)\geq a$.
\end{enumerate}
Note that the set $C_i(\epsilon)$ is not empty by
the above choice of the number $a$.
Also, by the second property above, we have
$\nu(\sigma_0)\geq 2a$ for every
$\nu\in C_i(\tau,\epsilon)$ and $i=1,2$ (recall that 
$\sigma_0$ contains at least two branches).

As in the proof of Lemma \ref{pA2}, let 
\[\mathbb{R}^m=\mathbb{R}^{\ell_0}\oplus \mathbb{R}^{\ell_1}\oplus
\mathbb{R}^{\ell_2}\]
be the vector space with basis the branches of $\tau$. The subspace
$\mathbb{R}^{\ell_0}$  is spanned by the branches of 
$\tau$ contained in $\sigma_0$. 
A point $\nu\in C_i(\tau,\epsilon)$ $(i=0,1,2)$ 
can be viewed as a non-negative 
vector $v(\nu)\in \mathbb{R}^m$  
with the property that the coefficients of 
the basis elements in 
$\mathbb{R}^{\ell_0}\oplus\mathbb{R}^{\ell_i}$ 
are bounded from below by $a>0$ 
and that the sum of the coordinates equals one. 

Denote as before by $A(\phi_i,\tau)$ the 
linear map which describes the action of
$\phi_i$ on ${\cal V}(\tau)$, represented by a matrix with respect
to the basis of $\mathbb{R}^m$ given by the branches of $\tau$.
There is an induced action on ${\cal V}_0(\tau)$ by rescaling
of the total mass; we denote this action by
$\hat A(\phi_i,\tau)$. 
We claim that there is a 
constant $u>0$ only depending on $\phi_0$ such that
for every $\epsilon >0$, for every $k\geq -u\log\epsilon$
and for $i=1,2$ we have
\[\hat A(\phi_0,\tau)^k(C_{i}(\tau,\epsilon))
  \subset C_0(\tau,\epsilon).\]


To show the claim note that
the linear map $A(\phi_0,\tau)$ preserves the decomposition 
$\mathbb{R}^m=\mathbb{R}^{\ell_0}\oplus \mathbb{R}^{\ell_1}\oplus \mathbb{R}^{\ell_2}$
and hence its matrix 
is in block form. The square matrix 
$A_0$ which defines the action on 
$\mathbb{R}^{\ell_0}$ is positive and integral,
and the square matrix defining the 
action on $\mathbb{R}^{\ell_1}\oplus \mathbb{R}^{\ell_2}$ is the identity.
Since $\ell_0\geq 2$, the entries of the matrix
$A_0^k$ are bounded from below by $2^{k-1}$ and hence 
\[A(\phi_0,\tau)^k(\nu)(\sigma_0)\geq 2^{k-1}\nu(\sigma_0)\text{ 
for all }\nu\in {\cal V}(\tau).\]
Now if $\nu\in C_i(\tau,\epsilon)$ then we have 
$\nu(\sigma_0)\geq 2a$ and therefore  
\[A(\phi_0,\tau)^k(\nu)(\sigma_0)\geq a2^{k},\]
moreover $A(\phi_0,\tau)^k(\nu)(\tau-\sigma_0)\leq 1-2a$ for all $k\geq 1$.
Thus if 
\[k\geq \hat k(\epsilon)= (\log(2(1-a))-\log(a\epsilon))/\log(2)\geq 2\]
then it holds $A(\phi_0,\tau)^k(\nu)(\sigma_0)\geq
A(\phi_0,\tau)^k(\nu)(\tau-\sigma_0)/\epsilon$.
Together with the estimate (\ref{branchweight}), this shows 
\[A(\phi_0,\tau)^k(\nu)/(A(\phi_0,\tau)^k(\nu)(\tau))\in C_0(\tau,\epsilon)\] 
and consequently 
$\hat A(\phi_0,\tau)^k(C_i(\tau,\epsilon))\subset C_0(\tau,\epsilon)$
for all $k\geq \hat k(\epsilon)$.

By definition of the maps $\phi_i$, we also have 
$\hat A(\phi_i,\tau)(C_0(\tau,\epsilon))\subset C_i(\tau,\epsilon)$ for 
$i=1,2$. 
Together we deduce the existence of 
a number $k(\epsilon)\sim -\log \epsilon>0$ such that
for $k> k(\epsilon)$, 
the set $C_0(\tau,\epsilon)$ is invariant under
the map which assigns to a measured 
geodesic lamination $0\not=\mu\in {\cal V}(\tau)$
the normalized
image of $\zeta(k)\mu\in {\cal V}(\zeta(k)\tau)$
under a carrying map ${\cal V}(\zeta(k)\tau)\to
{\cal V}(\tau)$. 

Let ${\cal Q}_0(\tau)$ be the set of all differentials
$q\in {\cal Q}(\tau)$ with the property that the total mass
deposited on $\tau$ by the horizontal measured geodesic lamination of
$q$ equals one.
Then 
if $k>k(\epsilon)$ and if $q\in {\cal Q}_0(\tau)$
is contained in the $\zeta(k)$-invariant flow line
of the Teichm\"uller flow, with horizontal 
measured geodesic lamination $\mu$, 
then
\[\mu\in C_0(\tau,\epsilon).\] Namely, the open cone
$C_0(\tau,\epsilon)\subset {\cal V}_0(\tau)$
is non-empty and invariant under the action of 
the map $\zeta(k)$, viewed as a map 
${\cal V}_0(\tau)\to {\cal V}_0(\tau)$.
Since the closure of $C_0(\tau,\epsilon)$ is homeomorphic to
closed a ball and we have $\zeta(k)C_0(\tau,\epsilon)\subset
C_0(\tau,\epsilon)$, the map $\zeta(k)$ has a fixed point in
$C_0(\tau,\epsilon)$ by the Brouwer fixed point theorem.
As $\zeta(k)$ is 
pseudo-Anosov, the projectivization of this fixed point is 
the attracting fixed point for the action of 
$\zeta(k)$ on ${\cal P\cal M\cal L}$, that is, it equals the 
horizontal projective measured lamination of $q$.

The $q$-length of the simple closed curve $c_i$
is contained in the interval
\[ [\frac{1}{\sqrt{2}}(\iota(\mu,c_i)+\iota(\nu,c_i)),
\iota(\mu,c_i)+\iota(\nu,c_i)]\]
where $\mu,\nu$ is the horizontal and the
vertical measured geodesic lamination of $q$,
respectively \cite{R14}. The intersection 
numbers $\iota(\mu,c_i),\iota(\nu,c_i)$ 
can be estimated as follows. 

By Lemma 2.5 of \cite{H06}, we have 
\[\iota(\mu,c_i)\leq \mu(\tau-\sigma_i)\]
and hence as $\mu\in C_0(\tau,\epsilon)$,
it holds
\begin{equation}\label{intersection}
  \iota(\mu,c_i)\leq \epsilon.\end{equation}

The support of the vertical (or horizontal) 
measured geodesic lamination
of a quadratic differential $q$ is invariant under the 
action of the Teichm\"uller flow $\Phi^t$, 
and its transverse measure scales with the scaling constant
$e^{t/2}$ (or with the scaling constant $e^{-t/2}$- note that this 
means that the length of the horizontal measured geodesic 
lamination is decreasing along a flow line of the Teichm\"uller flow). 
Recall that $\phi_0^k\phi_1\phi_0^k(c_1)=c_1$ for all $k$.
The estimate (\ref{intersection}) thus implies that
$\iota(e^{-t/2}\mu,c_1)<\epsilon$ for all $t\geq 0$.
  In other words, for $t\geq 0$, 
  the intersection between
  the horizontal measured geodesic lamination of
  $\Phi^tq$ and $c_1$ does not exceed $\epsilon$.

Let $\kappa >0$ be such that 
$\Phi^{2\kappa} q\in {\cal Q}_0(\phi_0^k\phi_1\phi_0^k(\tau))$, that is,
the total weight deposited by $e^\kappa\mu$ on $\phi_0^k\phi_1 \phi_0^k(\tau)$ 
equals one.
 Our goal is to show that the $\Phi^tq$-length of $c_1$
is smaller than $2\epsilon$ for $0\leq t\leq \kappa$. To this end
it suffices to show that
$\iota(c_1,e^{\kappa}\nu)<\epsilon$. 
To facilitate the notations, we show this by
replacing $\zeta(k)=(\phi_0^k\circ\phi_2\circ\phi_0^k)\circ
(\phi_0^k\circ \phi_1\circ \phi_0^k)$ by its conjugate
\[\hat\zeta(k)=(\phi_0^k\circ \phi_1\circ \phi_0^k)
 \circ (\phi_0^k \circ \phi_2\circ \phi_0^k).\]
 This conjugate admits $\tau$ as train track expansion. 
 Let $\hat \nu$ be the vertical measured geodesic
lamination of the differential $\hat q\in{\cal Q}_0(\tau)$
on the cotangent line of $\hat\zeta(k)$. We claim that  
$\iota(c_1,\hat\nu)<\epsilon$ which is equivalent to stating 
that $\iota(c_1,e^\kappa \nu)<\epsilon$.

To see that this is indeed the case recall that
by construction, the
carrying map $\phi_2\circ\phi_0^k(\tau)
\prec\phi_0^k(\tau)$
maps every branch of $\phi_2(\phi_0^k(\sigma_2))$ \emph{onto}
$\phi_0^k(\sigma_2)$. More precisely, the following holds true.
Let $\hat \mu\in {\cal V}_0(\tau)$ be the horizontal
measured geodesic lamination of $\hat q$ and 
let $\chi=\chi(k)<1$ be such that the total
weight of the measured geodesic lamination
$\chi\hat \mu$ on $\phi_0^{k}(\tau)$ 
equals one. Then the $\chi\hat \mu$-weight of 
every branch in
$\phi_0^k(\sigma_2)< \phi_0^k(\tau)$ is bounded from 
\emph{below} by the number $a$ introduced in the beginning of this proof.

The intersection number 
$1=\iota(\chi\hat \mu,\chi^{-1}\hat \nu)=
\iota(\hat \mu,\hat \nu)$ can be
calculated as 
\[\iota(\chi\hat \mu,\chi^{-1}\hat \nu)=
\sum_b \chi\hat \mu(b)\chi^{-1}\hat\nu^*(b)\]
where the sum is over all branches $b$ 
of $\phi_0^{k}(\tau)=\eta$ and 
$\chi\hat \mu(b)$ and 
$\chi^{-1}\hat \nu^*(b)$ are the weights
of $b$ for the transverse or tangential measure
determined by $\chi\hat \mu,\chi^{-1}\hat \nu$ \cite{PH92}.
As this intersection number equals one, we have
\[\chi^{-1}\hat \nu^*(b)\leq 1/a\] for every 
branch $b$ of the
subtrack $\phi_0^{k}(\sigma_2)$. As $c_1\subset
\phi_0^k(\tau)$ is an embedded subtrack consisting of
precisely two branches, we obtain
\[\iota(\chi^{-1}\hat \nu,c_1)\leq 2/a.\] 

It follows from the above discussion that $\chi(k)\to\infty$ 
as $k\to \infty$. 
Thus for sufficiently large $k$, say for $k\geq k(\epsilon)$, 
we indeed have
$\iota(c_1,\hat \nu)<\epsilon$.

We showed so far the following. Let $k\geq k(\epsilon)$, let
$\zeta(k)$ be as in the statement of the lemma and let
$q\in {\cal Q}_0(\tau)$
be contained in the unit contangent line of
an axis of $\zeta(k)$.
Let $\kappa_1>0$ be such that 
$\Phi^{\kappa_1} q\in {\cal Q}_0(\phi_0^k\circ \phi_1\circ\phi_0^k(\tau));$ 
 then for every $t\in [0,\kappa_1]$, the
$\Phi^tq$-length of the curve $c_1$ is at most $2\epsilon$.
The same argument also shows that if $\kappa_2>0$ is such that 
$\Phi^{\kappa_2}q\in {\cal Q}_0(\zeta(k)(\tau))$ then
the $\Phi^tq$-length of $\phi_0^k\phi_1\phi_0^k(c_2)$ is at most
$2\epsilon$ for $\kappa_1\leq t\leq \kappa_2$. 
But $\{\Phi^tq\mid 0\leq t\leq \kappa_2\}$
is a fundamental domain for the action of
$\zeta(k)$ on the contangent line of its axis and hence
the proof of the lemma is completed.
\end{proof}

We are now ready to complete the proof of Theorem \ref{main}.

\begin{theorem}\label{existence}
Let $S$ be a closed surface 
of genus $g\geq 0$ with $n\geq 0$ punctures and 
$3g-3+n\geq 5$. Then for every component ${\cal Q}$ 
of a stratum in ${\cal Q}(S)$ or ${\cal H}(S)$ and 
for every $\epsilon >0$, there exists a number
$d=d({\cal Q},\epsilon)>0$ such that
\[n({\cal Q},R)^{<\epsilon}\geq e^{(h({\cal Q})-1)R}/dR\]
for all sufficiently large $R>0$.
\end{theorem} 
\begin{proof} For the proof of the theorem, we consider 
periodic orbits which are defined by pseudo Anosov mapping classes
$\zeta(k)=(\phi_0^k\circ \phi_2\circ \phi_0^k)(\phi_0^k\circ \phi_1
\circ \phi_0^k)$
as constructed in Lemma \ref{pA2} 
where we fix
$\phi_0$ and $\phi_2$ and vary $\phi_1$.

Let $T(\phi_i)>0$ be the translation length of the 
pseudo-Anosov mapping class $\phi_i$ acting on the surface $S-c_i$
$(i=1,2)$. 
We claim that there
is a constant $\chi >0$ only depending on 
$k$ (and the choice of $\phi_0$) such that
the translation length of $\zeta(k)$ is 
contained in the interval $(0,T(\phi_1)+T(\phi_2)+\chi]$. 

Namely, the translation length
of $\zeta(k)$ is the logarithm of the Perron
Frobenius eigenvalue of the linear map $A(\tau,\zeta(k))$ 
determined by $\zeta(k)$. This Perron Frobenius
eigenvalue does not exceed the operator norm 
$\Vert A(\zeta(k),\tau)\Vert$ of 
the linear map $A(\zeta(k),\tau)$ with respect to the
norm $\Vert \nu\Vert =\sum_b \vert\nu(b)\vert$ on
$\mathbb{R}^m$. We have 
\[\Vert A(\tau,\zeta(k))\Vert \leq
\Vert A(\sigma_1,\phi_1)\Vert 
\Vert A(\sigma_2,\phi_2)\Vert \Vert B\Vert^{4k} \]
where $\Vert A(\sigma_i,\phi_i)\Vert$ is the operator norm of the
linear map $A(\sigma_i,\phi_i)$ 
defining the map $\phi_i$ (which coincides with
the operator norm of $A(\tau,\phi_i)$)
and where
$\Vert B\Vert $ is the operator norm of the
linear map $A(\tau,\phi_0)$.

On the other hand, by the choice of $\phi_i$, 
there is a constant $\kappa>0$ such that
\[\vert T(\phi_i) -\log \Vert A(\sigma_i,\phi_i)\Vert \vert \leq \kappa.\]
The claim follows. 

By Proposition \ref{countone}, 
there exist numbers $R_0>0,d>0$ such that 
for $T\geq R_0$, the number of mapping classes 
$\phi_i\in {\rm Mod}(S_1)$ of  
translation length at most $R\geq R_0$
which can be used in the construction of the mapping classes
$\zeta(k)$ is bounded from below by
$e^{h({\cal Q})-2}/dR$. 
Then the number of elements in ${\rm Mod}(S)$ of the form 
\[\zeta(k)=
\phi_0^k\circ \phi_2\circ \phi_0^{2k}\circ \phi_1\circ \phi_0^k\]
whose translation length is at most $R\geq R_0+\chi$ 
is at least $a e^{(h({\cal Q})-2)R}/R$ where $a>0$ is computed from 
$\chi$ and from the constant $d$. In particular,  
the asymptotic growth rate of such periodic orbits 
is at least $h({\cal Q})-2$.

Let $D$ be the
Dehn twist about the curve $c_1$ in the direction
determined by $\tau$. By this we mean that
we require $D\tau\prec\tau$ and hence 
$D^k\tau\prec\tau$ for all $k\geq 0$.
Given a nonnegative integer $p\leq u e^{T(\phi_1)}$ for a suitable choice of 
$u>0$ depending on $\epsilon$, 
it follows from 
the above discussion that the mapping class
$\phi_0^k\circ \phi_2\circ\phi_0^{2k}\circ D^p\circ\phi_1\circ\phi_0^k$ 
also has the required properties.
This implies that simultaneous 
twisting about $c_1$ adds one to the exponent in the 
counting of the orbits constructed above and completes
the proof of the theorem.
\end{proof}

\begin{remark} 
 The \emph{stable length} of a pseudo-Anosov
element $g\in {\rm Mod}(S)$ on the \emph{curve
graph} ${\cal C}(S)$ of $S$ is defined to be
\[sl(g)=\lim_{k\to \infty}\frac{1}{k}d(g^kc,c).\]
This does not depend on the choice of $c\in {\cal C}(S)$.

Bowditch \cite{Bw08} showed that there is an integer
$\ell>0$ only depending on the topological type of $S$ such that
the stable length on the curve graph of every pseudo-Anosov element $\phi$
is rational with denominator $\ell$. 
The stable length of each of the (infinitely many)
pseudo-Anosov elements $\zeta(k)$ constructed
in the proof of Theorem \ref{existence} is at most 2. We expect that this
is approximately sharp, which means that the asymptotic  
growth rate of all pseudo-Anosov mapping classes
of stable length at most
$2$ is not bigger than $h({\cal Q}(S))-1$.
\end{remark}

A similar argument also yields Theorem 2 from the introduction.

\begin{theorem}\label{escape}
For every component ${\cal Q}$ of a stratum as
in Theorem \ref{existence} 
there is a Teichm\"uller geodesic
with uniquely ergodic vertical measured geodesic
lamination whose projection to moduli space
escapes with linear speed to infinity.
\end{theorem}
\begin{proof}
We argue as in the proof of Theorem \ref{existence}.
Namely, choose simple closed curves $c_1,c_2$ as
in the proof of Theorem \ref{existence} and fix
pseudo-Anosov elements $\phi_i$ of $S_i$ with the properties
stated in the proof.

Let $\tau$ be a large train track as in the proof of Theorem
\ref{existence}, with subtracks $\sigma_i$. 
Let $A_i=A(\phi_i,\sigma_i)$ and denote by
$\Vert A_i\Vert$ the operator norm of $A$.
Choose a sequence of numbers 
$(k_i)$ such that for
each $i$, \[k_i\geq 2 \sum_{j\leq i-1}k_j.\]
Let $\psi_i=\phi_1\circ \phi_0^{k_i}\circ \phi_2\circ \phi_0^{k_i}$.
Then for each $i$ we have $\psi_i\tau\prec \tau$. 
Write $\zeta_k=\psi_k\circ\dots\circ\psi_1$. We claim that
$\cap_k\zeta_k{\cal V}(\tau)$ consists of a single
ray. 

To see that this is the case, note 
from the proof of Theorem \ref{existence} 
that for each
$i$ a carrying map $(\psi_{i+1}\circ\psi_i)\tau\to \tau$
maps every branch of $(\psi_{i+1}\circ\psi_i)\tau$
\emph{onto} $\tau$ and its normalization 
contracts distances in the cone ${\cal V}_0(\tau)$
with a factor which
is independent of $i$. This implies immediately that
$\cap_k\zeta_k{\cal V}(\tau)$ consists of a single ray.
In particular, a point on this ray is a uniquely ergodic measured geodesic
lamination which fills up $S$.

To show linear escape in moduli space, let $\lambda\in 
{\cal V}_0(\tau)$ be the normalized measured geodesic lamination
contained in this ray and let $\nu$ be a measured geodesic lamination
which fills and hits $\tau$ efficiently. Then the pair
$(\lambda,\nu)$ determines a quadratic differential $q$. 
Let $\ell>0$ and
let $a>0$ be such that the
measure $e^{a/2}\lambda$ on $\phi_0^{k_\ell}\psi_{\ell -1}(\tau)$ 
is normalized. Then the arguments in the
proof of Theorem \ref{existence} show that
the intersection of the curve 
$\phi^{k_\ell}\psi^{\ell-1}c_1$ with the lamination
$e^{a/2}\lambda$ 
is at most $ce^{-a/2}$, and similarly for the intersection with 
$e^{-a/2}\nu$. This yields the theorem.
\end{proof}

\begin{remark}
By the main result of \cite{CE07}, there is a number $\epsilon >0$
so that if a Teichm\"uller geodesic in moduli space escapes
into the cusp with a speed of at most $\epsilon \log t$, then
the vertical measured geodesic lamination of a differential on the
geodesic is uniquely ergodic. The above example implies that
one can construct differentials with uniquely ergodic
vertical measured laminations and arbitrarily prescribed 
excursions into the cusp.
\end{remark}

\bigskip

\noindent
MATHEMATISCHES INSTITUT DER UNIVERSIT\"AT BONN\\
Endenicher Allee 60,\\ 
D-53115 BONN, GERMANY

\bigskip\noindent
e-mail: ursula@math.uni-bonn.de

\begin{thebibliography}{BCGGM19}


\bibitem[BCGGM19]{BCGGM19}
M. Bainbridge, D. Chen, Q. Gendron, S. Grushevsky, M. M\"oller,
{\em The moduli space of multi-scale differentials}, 
arXiv:1910.13492. 




\bibitem[Bw08]{Bw08} B.~Bowditch, {\em Tight geodesics
in the curve complex}, Invent. Math. 171 (2008), 281--300.






\bibitem[CS21]{CS21} A. Calderon, N. Salter,
{\em Higher spin mapping class groups and 
  strata of abelian differentials over Teichm\"uller space},
Adv. Math. 389 (2021), Paper No, 107926.

\bibitem[CEG87]{CEG87} R.~Canary, D.~Epstein, P.~Green,
{\em Notes on notes of Thurston}, in ``Analytical and geometric
aspects of hyperbolic space'', edited by D.~Epstein, London Math.
Soc. Lecture Notes 111, Cambridge University Press, Cambridge 1987.

\bibitem[CB88]{CB88} A.~Casson with S.~Bleiler, {\sl Automorphisms
of surfaces after Nielsen and Thurston}, Cambridge University
Press, Cambridge 1988.

\bibitem[CE07]{CE07} Y.~Cheung, A.~Eskin,
{\em Unique ergodicity of translation flows},
Fields Inst. Commun. 51 (2007), 213--221. 


\bibitem[EMZ03]{EMZ03} A.~Eskin, H.~Masur, A.~Zorich,
{\em Moduli spaces of Abelian differentials: the principal
boundary, counting problems and the Siegel-Veech constants},
Publ. Math. IEHS 97 (2003), 61--179.

\bibitem[EM11]{EM11} A.~Eskin, M.~Mirzakhani, {\em Counting
closed geodesics in moduli space}, J. Mod. Dynamics 5 (2011),
71--105.

\bibitem[EMR19]{EMR13} A.~Eskin, M.~Mirzakhani,
K.~Rafi, {\em Counting closed geodesics in strata},
Invent. Math. 215 (2019), 235-607.





\bibitem[H06]{H06} U.~Hamenst\"adt, {\em Train tracks
and the Gromov boundary of the complex of curves}, 
in ``Spaces of Kleinian groups'' (Y.~Minsky, M.~Sakuma,
C.~Series, eds.), London Math. Soc. Lec. Notes 329 (2006),
187--207.

\bibitem[H09]{H09} U.~Hamenst\"adt, {\em Geometry of the
mapping class groups I: Boundary amenability}, 
Invent. Math. 175 (2009), 545--609.



\bibitem[H10]{H10} U. Hamenst\"adt, {\em Dynamics of the 
Teichm\"uller flow on compact invariant sets}, 
J. Mod. Dynamics 4 (2010), 393--418.


\bibitem[H13]{H13} U. Hamenst\"adt, 
{\em Bowen's construction for the Teichm\"uller flow}, 
J. Mod. Dynamics 7 (2013), 498-526.

\bibitem[H21]{H21} U. Hamenst\"adt, 
{\em Generating the spin mapping class group by Dehn twists}, 
Ann. H. Lebesgue 4 (2021), 1619--1658.

\bibitem[H23]{H22} U.~Hamenst\"adt, {\em Typical properties of
    periodic Teichm\"uller geodesics: Lyapunov exponents},
Erg. Th. \& Dyn. Sys. 43 (2023), 556--584. 






\bibitem[KMS86]{KMS86} S.~Kerckhoff, H.~Masur, J.~Smillie,
{\em Ergodicity of billiard flows and quadratic
differentials}, Ann. Math. 124 (1986), 293-311.



\bibitem[KZ03]{KZ03} M.~Kontsevich, A.~Zorich,
{\em Connected components of the moduli space of
Abelian differentials with prescribed singularities},
Invent. Math 153 (2003), 631--678.


\bibitem[L04]{L04} E.~Lanneau, {\em Hyperelliptic components
of the moduli space of quadratic differentials with 
prescribed singularities}, Comm. Math. Helv. 79 (2004), 471--501.

\bibitem[L08]{L08} E.~Lanneau, {\em Connected components
of the strata of the moduli spaces of quadratic differentials},
Ann. Sci. \'Ec. Norm. Sup\'er. 41 (2008), 1--56.

\bibitem[L83]{L83} G.~Levitt, 
{\em Foliations and laminations on hyperbolic surfaces},  
Topology  22 (1983), 119--135.


\bibitem[M82]{M82} H.~Masur, {\em Interval exchange transformations
and measured foliations}, Ann. Math. 115 (1982), 169-201.







\bibitem[Mi94]{M94} Y.~Minsky, {\em On rigidity, limit sets and
end-invariants of hyperbolic 3-manifolds}, J. Amer. Math. Soc. 7
(1994), 539--588.


\bibitem[MW17]{MW15} M.~Mirzakhani, A.~Wright,
{\em The boundary of an affine invariant submanifold},
Invent. Math. 209 (2017), 927--984.





\bibitem[Mo03]{Mo03} L.~Mosher, {\em Train track expansions of measured
foliations}, unpublished manuscript.



\bibitem[P88]{P88} R.~Penner, {\em A construction
of pseudo-Anosov homeomorphisms}, 
Trans. AMS 310 (1988), 179--197.

\bibitem[PH92]{PH92} R.~Penner with J.~Harer, {\sl Combinatorics
of train tracks}, Ann. Math. Studies 125, Princeton University
Press, Princeton 1992.





\bibitem[R14]{R14} K.~Rafi, {\em Hyperbolicity in 
Teichm\"uller space}, Geom. Top. 18 (2014), 3025--3053.


\bibitem[V86]{V86} W.~Veech, {\em The Teichm\"uller geodesic flow},
Ann. of Math. 124 (1986), 441--530.








\end{thebibliography}
\end{document}